\documentclass[twoside]{article}
\usepackage[letterpaper, left=3cm, right=3cm, top=4cm, bottom=2.5cm]{geometry}
\usepackage{graphicx}
\usepackage{amsmath,amssymb,amsthm}
\usepackage{authblk}
\usepackage[utf8]{inputenc}
\usepackage{microtype}
\usepackage{mathrsfs} 
\usepackage{enumitem}
\usepackage{calc}
\usepackage{esint}
\usepackage{fancyhdr}
\usepackage{xcolor}
\usepackage[labelfont = bf]{caption}
\usepackage{doi}
\usepackage{subfigure}
\usepackage[numbers]{natbib}
\usepackage{float}
\usepackage{stmaryrd}
\usepackage{mathtools}
\usepackage{booktabs}
\usepackage{colortbl}
\usepackage{tabulary}
\usepackage{diagbox}
\usepackage{caption}
\usepackage{cleveref}
\usepackage{commath}
\usepackage{multirow}
\usepackage{relsize}

\usepackage{physics}
\usepackage{amsmath}
\usepackage{tikz}
\usepackage{pgfplots}
\pgfplotsset{compat=1.18}
\usepackage{mathdots}
\usepackage{yhmath}
\usepackage{cancel}
\usepackage{color}
\usepackage{siunitx}
\usepackage{array}
\usepackage{multirow}
\usepackage{amssymb}
\usepackage{gensymb}
\usepackage{tabularx}
\usepackage{extarrows}
\usepackage{booktabs}
\usetikzlibrary{fadings}
\usetikzlibrary{patterns}
\usetikzlibrary{shadows.blur}
\usetikzlibrary{shapes}

\usepackage{orcidlink}

\newtheorem{theorem}{Theorem}[section]
\numberwithin{equation}{section}
\newtheorem{proposition}[theorem]{Proposition}
\newtheorem{definition}[theorem]{Definition}
\newtheorem{corollary}[theorem]{Corollary}
\newtheorem{remark}[theorem]{Remark}
\newtheorem{lemma}[theorem]{Lemma}

\newtheorem{assumption}[theorem]{Assumption}

\usepackage{titlesec}
\titleformat{\section}{\normalfont\scshape\centering}{\thesection.}{0.5em}{}
\titleformat*{\subsection}{\itshape}
\titleformat*{\subsubsection}{\itshape}

 \providecommand{\keywords}[1]
 {
 	{\small\emph{Keywords:} #1}
 }
 \providecommand{\MSC}[1]
 {
 	{\small\emph{AMS MSC (2020):~~} #1}
 }

\definecolor{denim}{rgb}{0.08, 0.38, 0.74}
\definecolor{byzantium}{rgb}{0.44, 0.16, 0.39} 
\definecolor{shamrockgreen}{rgb}{0.0, 0.62, 0.38} 

\usepackage{hyperref}
\hypersetup{
	colorlinks=true,
	linkcolor=denim,
	citecolor = shamrockgreen,
	filecolor=magenta, 
	urlcolor=byzantium,
}

\usepackage{soul}

\begin{document}
	\setlength{\abovedisplayskip}{5.5pt}
	\setlength{\belowdisplayskip}{5.5pt}
	\setlength{\abovedisplayshortskip}{5.5pt}
	\setlength{\belowdisplayshortskip}{5.5pt}

	\title{\vspace{-15mm}Duality-Based \emph{A Posteriori} Error Identities for Subgradient Flows Based on the Br\'ezis--Ekeland--Nayroles Principle\thanks{This work is partially supported by the Office of Naval Research under Award No. N00014-24-1-2147, the National Science Foundation under Grant DMS-2408877, the Air Force Office of Scientific Research under Award No. FA9550-22-1-0248, and SURE-AI Centre grant 357482, Research Council of Norway. The work of H.A. and A.K. is additionally supported by the MATH+ project AA-Tech-4 “DOC-TWIN”. The contribution of A.K. is additionally funded by the Deutsche Forschungsgemeinschaft (DFG, German Research Foundation) – Project number 581203342.}\vspace{-0.5mm}
    }
	\author[1]{Harbir Antil\,\orcidlink{0000-0002-6641-1449}%
	\thanks{Email: \url{hantil@gmu.edu}}}
    \author[2]{Alex Kaltenbach\,\orcidlink{0000-0001-6478-7963}%
	\thanks{Email: \url{kaltenbach@math.tu-berlin.de}}}
    \author[3]{Keegan L.\,A. Kirk\,\orcidlink{0000-0003-1190-6708}%
	\thanks{Email: \url{kkirk8@lsu.edu}\vspace{-0.5mm}}}
	\date{\today\vspace{-0.5mm}} 
	\affil[1,2,3]{\small{Department of Mathematical Sciences and the Center for Mathematics and Artificial Intelligence (CMAI), George Mason University, Fairfax, VA 22030, USA}}
	\affil[2]{\small{Institute of Mathematics, Technical University of Berlin, Stra\ss e des 17.\ Juni 136, 10623 Berlin}}
    \affil[3]{\small{Department of Mathematics, Louisiana State University, Baton Rouge, LA 70803, USA}\vspace{-0.5mm}}
	\maketitle

	\pagestyle{fancy}
	\fancyhf{}
	\fancyheadoffset{0cm}
	\addtolength{\headheight}{-0.25cm}
	\renewcommand{\headrulewidth}{0pt} 
	\renewcommand{\footrulewidth}{0pt}
	\fancyhead[CO]{\textsc{Duality-Based \emph{A Posteriori} Error Identities for Subgradient Flows}}
	\fancyhead[CE]{\textsc{H. Antil, A. Kaltenbach, and K. Kirk}}
	\fancyhead[R]{\thepage}
	\fancyfoot[R]{}
	
	\begin{abstract}
	We derive duality-based \textit{a posteriori} error identities for a broad class of~subgradient~flows~\mbox{induced} by time-dependent convex integral functionals. Starting from the Br\'ezis--Ekeland--Nayroles~\mbox{principle}, we identify an unsteady primal energy functional and derive its Fenchel dual formulation, including strong duality and the corresponding optimality system under general normal-integrand~\mbox{assumptions}.
    This Fenchel duality framework is  used to derive
\textit{a posteriori} error~identities~for~\mbox{subgradient}~flows. In doing so, we depart from the usual duality-based \textit{a posteriori} error control framework in the unsteady setting, since the Br\'ezis--Ekeland--Nayroles formulation reveals the following unsteady feature: the minimal primal value and the maximal dual value are both prescribed by the initial datum. This allows us to pass from a combined primal-dual gap identity to separate~primal~and~dual~gap~\mbox{identities}. These identities quantify the primal and dual errors independently and admit representations in terms of generalized Bregman divergences and, under a spatial convex conjugation formula, as non-negative time-space integral quantities suitable for localization.
 The abstract framework is applied to a number of variational problems of physical interest, including
 the unsteady heat equation, the unsteady Stokes equations, the unsteady~\mbox{Navier--Lam\'e}~equations, 
 the unsteady Bingham flow through a pipe, the unsteady obstacle problem,  and the unsteady elasto-plastic torsion problem. 
	\end{abstract}
	
	
	\MSC{Primary: 35K90, 47J35, 49N15; Secondary: 49M29, 65M15}

    \medskip

	\section{Introduction}\label{sec:introduction}\thispagestyle{empty}\enlargethispage{15mm}\vspace{-0.5mm}

    \hspace{5mm}Duality-based \emph{a posteriori} error control is based on the observation that, for convex variational problems, errors can often be quantified by the violation of respective primal and dual optimality~relations. Its classical prototype is the Prager--Synge identity (\textit{cf}.\ \cite{PragerSynge1947}), which provides an exact and constant-free  \textit{a posteriori} error identity for linear elliptic problems in Hilbert space settings. In its original form, this identity relates a conforming primal approximation, satisfying the essential boundary conditions, to an equilibrated dual flux approximation, satisfying the equilibrium equation and natural boundary conditions. This idea has been developed in several directions, including equilibrated residual methods, hypercircle methods, and functional-type majorants (see, \textit{e.g.}, \cite{AinsworthOden2000,Verfuerth1996,Repin2000,Repin2008,
BraessPillweinSchoeberl2009,Vohralik2008,ErnVohralik2015}~and~the~references~therein).\linebreak
    Related \textit{a posteriori} approaches include
residual-based, hierarchical, goal-oriented, and reconstruction-based techniques
(see, \textit{e.g.},
\cite{ErikssonJohnson1991,ErikssonJohnson1995,Picasso1998,Verfuerth2003,BeckerRannacher2001,MakridakisNochetto2003,NochettoSavareVerdi2000}
and the references therein). A common feature of the duality-based approaches is that the error is represented by a computable quantity which
measures, in one form or another, the failure of admissible primal and dual objects to satisfy the
optimality~system.\newpage

    To begin with, we recall this structure for steady convex variational problems, since it provides the 
    template from which the unsteady theory will depart.\vspace{-1.25mm}

    \subsection{Duality-based \emph{a posteriori} error control for steady convex variational problems}\vspace{-0.75mm}\enlargethispage{5mm}

    \hspace{5mm}In the present paper, 
    we are interested in deriving such duality-based \textit{a posteriori} error identities for a broad class of subgradient flows induced by a (time-dependent) family of convex energy functionals $E(t,\cdot)\colon V\coloneqq W^{1,p}_D(\Omega;\mathbb{R}^{\ell})\to \mathbb{R}\cup\{+\infty\}$, $t\in I\coloneqq (0,t_{\mathtt{fin}})$, where $\Omega\subseteq \mathbb{R}^d$, $d\in \mathbb{N}$, is a bounded Lipschitz domain, $p\in (1,+\infty)$ with $p\ge \frac{2d}{d+2}$, $\ell\in \mathbb{N}$, and $t_{\mathtt{fin}}\in (0,+\infty)$, for every $v\in V$ defined by\vspace{-0.5mm}
\begin{align}\label{intro:primal_steady}
E(t,v)
\coloneqq 
G(t,\nabla v)+F(t,v)\,.\\[-6mm]\notag
\end{align}
Here, \hspace{-0.15mm}the \hspace{-0.15mm}integral \hspace{-0.15mm}functionals \hspace{-0.15mm}$G(t,\cdot)\colon\hspace{-0.175em} Y\hspace{-0.175em}\coloneqq\hspace{-0.175em} L^p(\Omega;\mathbb{R}^{\ell\times d})\hspace{-0.175em}\to\hspace{-0.175em} \mathbb{R}\cup\{+\infty\}$, $t\hspace{-0.175em}\in\hspace{-0.175em} I$, \hspace{-0.15mm}and~\hspace{-0.15mm}${F(t,\cdot)\colon \hspace{-0.175em}V\hspace{-0.175em}\to\hspace{-0.175em} \mathbb{R}\hspace{-0.175em}\cup\hspace{-0.175em}\{+\infty\}}$,~${t\hspace{-0.175em}\in\hspace{-0.175em} I}$, for every $y\in Y$ and $v\in V$, respectively, are defined by\vspace{-0.5mm}
\begin{align}\label{intro:GF}
G(t,y)
\coloneqq
\int_\Omega \phi(t,\cdot,y)\,\mathrm{d}x\,,
\qquad
F(t,v)
\coloneqq
\int_\Omega \psi(t,\cdot,v)\,\mathrm{d}x \,,\\[-6mm]\notag
\end{align}
where the (time-dependent) energy densities $\phi\colon Q\times \mathbb{R}^{\ell\times d}\to \mathbb{R}\cup\{+\infty\}$ and $\psi\colon Q\times \mathbb{R}^{\ell}\to \mathbb{R}\cup\{+\infty\}$, where $Q\coloneqq I\times \Omega$, are convex normal integrands satisfying a non-triviality condition (\textit{cf}.\ Assumption~\ref{ass:energy_densities}).\linebreak 
We refer to the problem of minimizing
the steady primal energy functional \eqref{intro:primal_steady} 
as~the~steady primal problem, \hspace{-0.1mm}and \hspace{-0.1mm}a \hspace{-0.1mm}minimizer \hspace{-0.1mm}$u_t\hspace{-0.15em}\in \hspace{-0.15em} V$, \hspace{-0.1mm}the \hspace{-0.1mm}existence \hspace{-0.1mm}of \hspace{-0.1mm}which \hspace{-0.1mm}we \hspace{-0.1mm}assume, \hspace{-0.1mm}to \hspace{-0.1mm}as \hspace{-0.1mm}the \hspace{-0.1mm}steady~\hspace{-0.1mm}\mbox{primal}~\hspace{-0.1mm}\mbox{solution}.\linebreak
For comparison with the unsteady theory developed below, in this subsection, we~freeze~a~time~\mbox{slice}~${t\in I}$, in which case, the convex variational problem \eqref{intro:primal_steady} (with \eqref{intro:GF}) is amenable to  the familiar duality-based \emph{a posteriori} error control for steady convex variational problems developed~in~the~works~\cite{BartelsMilicevic2020,BartelsKaltenbach2023,BartelsGudiKaltenbach2025,AntilBartelsKaltenbachKhandelwal2025,AntilKaltenbachKirk2026}.\linebreak Under additional assumptions on the energy density $\psi(t,\cdot,\cdot)\colon \Omega\times \mathbb{R}^{\ell}\to \mathbb{R}\cup\{+\infty\}$ (\textit{cf}.\ Assumption~\ref{ass:convex_conjugation}),\linebreak a (Fenchel) dual problem (in the sense of \cite[Rem.~4.2, p.~60/61]{EkelandTemam1999}) is not only defined~on~${Y^*\hspace{-0.175em}=\hspace{-0.175em}L^{p'}(\Omega;\mathbb{R}^{\ell\times d})}$,\linebreak but on $Y^*(\operatorname{div})\coloneqq \smash{W^{\operatorname{div},p'}_N}(\Omega;\mathbb{R}^{\ell})$
(see, \textit{e.g.}, \cite[Sec.\ 3]{BartelsKaltenbach2026}, for more  details).~More~\mbox{precisely},~it~is given via the maximization of the (time-dependent) concave dual energy functional $D(t,\cdot)\colon Y^*(\operatorname{div})\to \mathbb{R}\cup\{-\infty\}$, for every $y\in Y^*(\operatorname{div})$ defined by\vspace{-0.5mm}
    \begin{align}\label{intro:dual_steady}
        D(t,y)\coloneqq -\int_{\Omega}{\phi^*(t,y)\,\mathrm{d}x}-\int_{\Omega}{\psi^*(t,\operatorname{div}y)\,\mathrm{d}x}\,.\\[-6mm]\notag
    \end{align}
    Under a standard Fenchel duality qualification, the dual energy functional
\eqref{intro:dual_steady} admits a maximizer $z_t\in Y^*(\operatorname{div})$, called dual
solution, and a strong duality relation applies, \textit{i.e.}, there holds\vspace{-0.5mm}
\begin{align}\label{intro:strong_duality_steady}
    E(t,u_t)=D(t,z_t)\,.\\[-6mm]\notag
\end{align}
Based on 
\eqref{intro:strong_duality_steady}, for every $v\in V$ and $y\in Y^*(\operatorname{div})$, one obtains
 the \emph{steady primal-dual gap identity}\vspace{-0.5mm}
\begin{align}\label{intro:primal_dual_gap_identity_steady}
    \rho_{E(t,\cdot)}^2(v)+\rho_{-D(t,\cdot)}^2(y)
    =
    \eta_{E(t,\cdot)-D(t,\cdot)}^2(v,y)\,.
\end{align}
Here, $\rho_{E(t,\cdot)}^2\colon \hspace{-0.1em}V\hspace{-0.1em}\to\hspace{-0.1em} [0,+\infty]$ and
$\smash{\rho_{\smash{-D(t,\cdot)}}^2}\colon\hspace{-0.1em} Y^*(\operatorname{div})\hspace{-0.1em}\to\hspace{-0.1em} [0,+\infty]$
denote the optimal strong convexity~\mbox{measures} (\textit{cf}.\ Remark \ref{rem:optimal_strong_convexity_measure}) of the steady primal energy
functional \eqref{intro:primal_steady} and of the negative steady dual energy functional \eqref{intro:dual_steady}, respectively.
Moreover, $\eta_{E(t,\cdot)-D(t,\cdot)}^2\colon
V\times Y^*(\operatorname{div})\to [0,+\infty]$ denotes the steady~primal-dual gap estimator, which, for every $v\in V$ and
$y\in Y^*(\operatorname{div})$, is given via\vspace{-0.5mm}
\begin{align}\label{intro:primal_dual_gap_estimator_steady}
    \begin{aligned}
    \eta_{E(t,\cdot)-D(t,\cdot)}^2(v,y)
    &\coloneqq E(t,v)-D(t,y)
    \\
    &=
    \int_{\Omega}
    \bigl(
        \phi^*(t,\cdot,y)
        -y: \nabla v
        +\phi(t,\cdot,\nabla v)
    \bigr)\,\mathrm{d}x
   \\& \quad+
    \int_{\Omega}
    \bigl(
        \psi^*(t,\cdot,\operatorname{div}y)
        -\operatorname{div}y\cdot v
        +\psi(t,\cdot,v)
    \bigr)\,\mathrm{d}x\,.
    \end{aligned}\\[-6mm]\notag
\end{align}
By \hspace{-0.1mm}the \hspace{-0.1mm}Fenchel--Young \hspace{-0.1mm}inequality, 
\hspace{-0.1mm}the \hspace{-0.1mm}two \hspace{-0.1mm}integrands \hspace{-0.1mm}in \hspace{-0.1mm}\eqref{intro:primal_dual_gap_estimator_steady} \hspace{-0.1mm}are \hspace{-0.1mm}point-wise~\hspace{-0.1mm}non-negative~\hspace{-0.1mm}(a.e.).~\hspace{-0.1mm}Hence,~\hspace{-0.1mm}the steady primal-dual gap estimator \eqref{intro:primal_dual_gap_estimator_steady} admits a local integral decomposition, which 
naturally~gives~rise to adaptive mesh-refinement indicators measuring the violation
of the corresponding~optimality~inclusions\vspace{-0.5mm}
\begin{subequations}\label{intro:primal_optimality_inclusions_steady}
\begin{alignat}{2}
    z_t&\in \partial_a\phi(t,\cdot,\nabla u_t)
    &&\quad \text{a.e.\ in }\Omega\,,
    \\
    \operatorname{div}z_t&\in \partial_b\psi(t,\cdot,u_t)
    &&\quad \text{a.e.\ in }\Omega\,.\\[-6mm]\notag
\end{alignat}
\end{subequations}\newpage

\subsection{Duality-based \textit{a posteriori} error control for subgradient flows induced by convex variational problems}\vspace{-0.5mm}\enlargethispage{5mm}

\hspace{5mm}Having recalled duality-based \emph{a posteriori} error control for steady convex variational~problems for frozen time slices,
we now consider the subgradient flow induced by the (time-dependent)
family~of~primal energy functionals \eqref{intro:primal_steady}. More precisely,
we seek 
$u\in \mathcal{W}(I)\coloneqq {\{v\in L^p(I;V)\mid \partial_tv\in L^{p'}(I;V^*)\}}$~such~that\vspace{-4.75mm}
\begin{subequations}\label{intro:subgradient_flow}
    \begin{alignat}{2}\label{intro:subgradient_flow.1}
        \partial_tu(t)+\partial_v E(t,u(t))&\ni 0_{V^*}
        &&\quad \text{in }V^*\quad\text{for a.e.\ }t\in I\,,
        \\
        u(0)&=u_0
        &&\quad\text{in }H\,,
        \label{intro:subgradient_flow.2}
    \end{alignat}\\[-5mm]\notag
\end{subequations}
where the initial condition \eqref{intro:subgradient_flow.2} is well-defined
since $\mathcal{W}(I)\hookrightarrow C^0(\overline{I};H)$, where 
$H\coloneqq L^2(\Omega;\mathbb{R}^{\ell})$.

A natural global variational formulation of the subgradient flow formulation \eqref{intro:subgradient_flow}
is provided~by~the Br\'ezis--Ekeland--Nayroles \hspace{-0.1mm}principle
\hspace{-0.15mm}(\textit{cf}.\ \hspace{-0.15mm}\cite{BrezisEkeland1976I,BrezisEkeland1976II}
\hspace{-0.15mm}and \hspace{-0.15mm}\cite{Nayroles1976I,Nayroles1976II}, \hspace{-0.15mm}respectively; \hspace{-0.15mm}see \hspace{-0.15mm}also \hspace{-0.15mm}\cite[\hspace{-0.25mm}Sec.~\hspace{-0.25mm}3.9]{AubinCellina1984},~\hspace{-0.15mm}\mbox{\cite[\hspace{-0.25mm}Sec.~\hspace{-0.75mm}8.10]{Roubicek2013}}, \hspace{-0.1mm}\cite{Stefanelli2008}, \hspace{-0.1mm}and \hspace{-0.1mm}\cite{CariniJensenNuernberg2023}). \hspace{-0.1mm}It \hspace{-0.1mm}characterizes \hspace{-0.1mm}solutions
\hspace{-0.1mm}$u\hspace{-0.15em}\in\hspace{-0.15em} \mathcal{W}(I)$ \hspace{-0.1mm}of \hspace{-0.1mm}the \hspace{-0.1mm}subgradient \hspace{-0.1mm}flow \hspace{-0.1mm}formulation \hspace{-0.1mm}\eqref{intro:subgradient_flow} \hspace{-0.1mm}as \hspace{-0.1mm}minimizers of
the Br\'ezis--Ekeland--Nayroles energy functional
$\mathcal{E}\colon \mathcal{W}(I)\to \mathbb{R}\cup\{+\infty\}$,~for~\mbox{every}~${v\in \mathcal{W}(I)}$~\mbox{defined}~by\vspace{-0.75mm}
\begin{align}\label{intro:primal_unsteady}
     \mathcal{E}(v)
     \coloneqq
     \int_{I}{E(t,v(t))\,\mathrm{d}t}
     +\int_{I}{E^*(t,-\partial_tv(t))\,\mathrm{d}t}
     +\tfrac{1}{2} \|v(t_{\mathtt{fin}})\|^2_H
     + \chi_{\{u_0\}} (v(0))\,,\\[-6mm]\notag
\end{align}
which takes the minimal value\vspace{-1mm}\enlargethispage{2.5mm} 
\begin{align}\label{intro:primal_unsteady_minimal}
    \mathcal{E}(u)=\tfrac{1}{2}\|u_0\|_H^2\,.\\[-6mm]\notag
\end{align}
Under suitable additional assumptions on the energy densities $\phi$ and $\psi$ (\textit{cf}.\ Assumptions \ref{ass:energy_densities} and \ref{ass:sufficient_for_conjugation}), we will establish that a Fenchel dual problem to the minimization of the primal energy functional is given via the maximization of the unsteady dual energy functional $\mathcal{D}\colon L^{p'}(I;Y^*)\times\mathcal{W}(I)\to \mathbb{R}\cup\{-\infty\}$, for every $(y,\lambda)\in  L^{p'}(I;Y^*)\times \mathcal{W}(I)$ defined by\vspace{-0.75mm}
\begin{align}\label{intro:dual_unsteady}
    \begin{aligned}  \mathcal{D}(y,\lambda)&\coloneqq-\int_{I}{G^*(t,y(t))\,\mathrm{d}t}-\int_{I}{F^*(t,-L^* y(t)-\partial_t\lambda(t))\,\mathrm{d}t}-\int_{I}{E(t,\lambda(t))\,\mathrm{d}t}
          \\&\quad -\tfrac{1}{2}\|\lambda(t_{\mathtt{fin}})\|_H^2+(\lambda(0),u_0)_H\,,
          \end{aligned}\\[-6mm]\notag
\end{align}
where $L^*\colon Y^*\to V^*$ denotes the adjoint operator of the gradient operator $L\coloneqq \nabla \colon V\to Y$.
Under the additional continuity condition in
Theorem~\ref{thm:duality}(\hyperlink{thm:duality.ii}{ii}), there exists
a dual solution $(z,\mu)\in L^{p'}(I;Y^*)\times\mathcal{W}(I)$ and a
strong duality relation applies, \textit{i.e.}, there holds\vspace{-0.75mm}
\begin{align}\label{intro:strong_duality_unsteady}
\mathcal{E}(u)=\mathcal{D}(z,\mu)\,.\\[-6mm]\notag
\end{align}
This Fenchel duality framework is the starting point for our
\textit{a posteriori} error identities~for~\mbox{subgradient} flows. In contrast to
the steady setting, however, we do not merely reproduce a 
primal-dual gap identity. The Br\'ezis--Ekeland--Nayroles principle reveals the following unsteady feature:~due~to~\eqref{intro:primal_unsteady_minimal}~and~\eqref{intro:strong_duality_unsteady}, the minimal primal value and the
maximal dual value are both prescribed by the initial datum. This allows us
to pass from a combined primal-dual gap identity to separate primal and dual
gap identities. These identities quantify the primal and dual errors
independently and admit representations~by~generalized Bregman
divergences (\textit{cf}.\ \cite{Bregman1967}) and, under a spatial convex conjugation formula, as non-negative
time-space integrals suitable~for~localization. These two classes of identities have different~practical~features.\linebreak The primal
gap identities are closer to weak residual-type identities: admissible primal
approximations are comparatively easy to generate, while the evaluation of
the right-hand side may be more involved. They are related
to the estimates in \cite{Stefanelli2008,Stefanelli2009}, which are based on the
Br\'ezis--Ekeland--Nayroles principle and yield error control in a
$C^0(\overline{I};H)$-type measure. The dual gap identities, by contrast, can
be viewed as direct unsteady counterparts of the steady primal-dual gap
identities. The right-hand side is comparatively easy to evaluate, while  the generation of admissible dual
approximations~may~be~more~involved,\linebreak
since the corresponding 
unsteady equilibrium relations have to be satisfied. 
They are related to
the~func\-tional-type error identities for evolutionary problems in \cite{Repin2008,Repin2024,ApushkinskayaRepin2022,LangerRepinWolfmayr2015,MatculevichNeittaanmakiRepin2015}; in this paper, 
such identities arise
systematically from the Fenchel dual formulation of the~\mbox{Br\'ezis--Ekeland--Nayroles}~\mbox{energy}~\mbox{functional}.

\emph{The paper is organized as follows.} In Section~\ref{sec:preliminaries}, we collect the necessary preliminaries on convex integral functionals, normal integrands, and the function spaces used throughout the paper.  In Section~\ref{sec:fenchel_duality_framework}, we develop the Fenchel duality framework for subgradient flows based on the Br\'ezis--Ekeland--Nayroles principle. Section~\ref{sec:duality_based_a_posteriori_error_control}
is devoted to the resulting duality-based \textit{a posteriori} error identities, including the separate primal and dual gap identities and their localizable time-space integral representations. Finally, in Section~\ref{sec:applications}, we 
apply the abstract framework to the model problems considered in this paper.\newpage

    \section{Preliminaries}\vspace{-0.5mm}\label{sec:preliminaries}

    \subsection{Integral functionals}\vspace{-0.5mm}

    \hspace{5mm}In this subsection, we recall important definitions and results on normal integrands, the lower~compact\-ness \hspace{-0.1mm}property, \hspace{-0.1mm}integral \hspace{-0.1mm}functionals, \hspace{-0.1mm}and \hspace{-0.1mm}spatial \hspace{-0.1mm}integral \hspace{-0.1mm}reductions. \hspace{-0.1mm}For \hspace{-0.1mm}more \hspace{-0.1mm}details,~\hspace{-0.1mm}we~\hspace{-0.1mm}refer~\hspace{-0.1mm}to~\hspace{-0.1mm}\mbox{\cite{Rockafellar1968,Rockafellar1971,RockafellarWets1998}}.

Throughout the entire subsection, if not otherwise specified, let $T\subseteq \mathbb{R}^N$, $N\in \mathbb{N}$, be a Lebesgue measurable set and let $(X,\|\cdot\|_X)$ be a separable Banach space with (topological) dual space $(X^*,\|\cdot\|_{X^*})$ and duality pairing $\langle \cdot,\cdot\rangle_X\colon X^*\times X\to \mathbb{R}$,  for every $x^*\in X^*$ and $x\in X$ defined by $\langle x^*,x\rangle_X\coloneqq x^*(x)$. Moreover, we denote the Lebesgue $\sigma$-algebra on $T$ by $\mathcal{L}^{N}(T)$ and the Borel $\sigma$-algebra on $X$~by~$\mathcal{B}(X)$. Then, the space of $\mathcal{L}^{N}(T)$-$\mathcal{B}(X)$-measurable\footnote{Due to the separability of $(X,\|\cdot\|_X)$, the notions of $\mathcal{L}^{N}(T)$-$\mathcal{B}(X)$-measurability and Bochner measurability coincide (\textit{cf}.\ \cite[Thm.\ 2, p.\ 99]{Dinculeanu1966}).} functions\footnote{Throughout the entire paper, we do not distinguish between equivalence classes with respect to equality $\mathcal{L}^N$-a.e.\ and representing functions of these classes.} is denoted by 
\begin{align*}
    L^0(T;X)\coloneqq \smash{\bigl\{x\colon T\to X\mid x^{-1}(B)\in \mathcal{L}^{N}(T)\text{ for all }B\in \mathcal{B}(X) \bigr\}}\,.
\end{align*}
For $p\in [1,+\infty]$, the Bochner space of $p$-integrable functions is denoted by 
\begin{align*}
    L^p(T;X)\coloneqq \smash{\bigl\{x\in L^0(T;X)\mid \|x\|_{L^p(T;X)}<+\infty \bigr\}}\,,
\end{align*}
 and forms a Banach space when equipped with the norm $\|\cdot\|_{L^p(T;X)}\coloneqq (\int_T{\|(\cdot)(t)\|_X^p\,\mathrm{d}t})^{\smash{\frac{1}{p}}}$ if $p\in [1,+\infty)$ and $\|\cdot\|_{L^\infty(T;X)}\coloneqq \operatorname{\textup{ess\,sup}}_{t\in T}{\{\|(\cdot)(t)\|_X\}}$ else (\textit{cf}.\ \cite[Kap. IV, Satz 1.11 \& Satz 1.12]{GajewskiGroegerZacharias1974}). If $p\in [1,+\infty)$ and  \hspace{-0.1mm}$(X^*,\|\cdot\|_{X^*})$  \hspace{-0.1mm}has~\hspace{-0.1mm}the~\hspace{-0.1mm}Radon--Nikod\'ym \hspace{-0.1mm}property \hspace{-0.1mm}(\textit{e.g.}, \hspace{-0.1mm}if  \hspace{-0.1mm}$(X,\|\cdot\|_X)$~\hspace{-0.1mm}is~\hspace{-0.1mm}\mbox{reflexive}~\hspace{-0.1mm}or~\hspace{-0.1mm}${(X^*,\|\!\cdot\!\|_{X^*})}$~\mbox{separable}),\linebreak its (topological) dual space can be characterized via the isometric isomorphism $(L^p(T;X))^*\cong L^{p'}(T;X^*)$\linebreak (\textit{cf}.~\mbox{\cite[Thms.~3.2,~3.3]{BochnerTaylor1938}}), in which case we do not distinguish between functionals in $(L^p(T;X))^*$~and~functions \hspace{-0.15mm}in \hspace{-0.15mm}$L^{p'}\hspace{-0.1em}(T;X^*)$.
 \hspace{-0.25mm}Here, \hspace{-0.15mm}the \hspace{-0.15mm}Hölder \hspace{-0.15mm}conjugate \hspace{-0.15mm}exponent \hspace{-0.15mm}$p'\hspace{-0.2em}\in \hspace{-0.15em}[1,+\infty]$ \hspace{-0.15mm}is \hspace{-0.15mm}defined~\hspace{-0.15mm}by~\hspace{-0.15mm}${\frac{1}{p}\hspace{-0.15em}+\hspace{-0.15em}\frac{1}{p'}\hspace{-0.15em}=\hspace{-0.15em}1}$,~\hspace{-0.15mm}where~\hspace{-0.15mm}${\frac{1}{\infty}\hspace{-0.15em}\coloneqq\hspace{-0.15em} 0}$.
    
   Moreover, we denote by $\Gamma_0(X)$ the space of proper, convex, and lower semi-continuous~functionals~on $X$ and, for a functional $f\colon X\to \mathbb{R}\cup\{\pm\infty\}$, its effective domain by $\operatorname{dom}(f)\coloneqq \{x\in X\mid f(x)<+\infty\}$.

    To begin with, we introduce the central notion of a \emph{(convex) normal integrand} for mappings of the form
${f\colon T\times X\to \mathbb{R}\cup\{\pm\infty\}}$,  called \emph{integrands} if no additional properties are required.~This~notion provides a general sufficient condition ensuring that the composition
of a $\mathcal{L}^{N}(T)$-$\mathcal{B}(X)$-measurable function with an integrand
is $\mathcal{L}^{N}(T)$-measurable and forms the basis for the definition~of~\mbox{integral}~\mbox{functionals}~\mbox{below}.\vspace{-0.5mm}

    \begin{definition}[Normal and Carath\'eodory integrands]\label{def:normal_integrand}
        An integrand $f\colon T\times X\to \mathbb{R}\cup\{\pm\infty\}$ is called
        \begin{itemize}[noitemsep,topsep=2pt,leftmargin=!,labelwidth=\widthof{(iii)}]
            \item[(i)] \hypertarget{def:normal_integrand.i}{} \emph{normal integrand} if it is $ \mathcal{L}^N(T)\otimes \mathcal{B}(X)$-measurable (\textit{i.e.}, $f^{-1}(B)\in \mathcal{L}^N(T)\otimes \mathcal{B}(X)$ for all $B\in \mathcal{B}(\mathbb{R}\cup\{\pm\infty\})$) and $f(t,\cdot)\colon X\to \mathbb{R}\cup\{\pm\infty\}$ is proper~and~lower semi-continuous~for~\mbox{$\mathcal{L}^N$-a.e.}~$t\in T$;
            \item[(ii)] \hypertarget{def:normal_integrand.ii}{} \emph{convex normal integrand} if it is a normal integrand and 
            $f(t,\cdot)\colon X\to \mathbb{R}\cup\{+\infty\}$ is convex for $\mathcal{L}^N$-a.e.\ $t\in T$ (in particular, $f(t,\cdot)\in \Gamma_0(X)$ for $\mathcal{L}^N$-a.e.\ $t\in T$); 
            \item[(iii)] \hypertarget{def:normal_integrand.iii}{} \emph{Carath\'eodory integrand} if $ f(\cdot,x)\colon T\to \mathbb{R}$  (finite-valued) is $\mathcal{L}^N(T)$-measurable for all $x\in X$ and $f(t,\cdot)\colon X\to \mathbb{R}$ is continuous for $\mathcal{L}^N$-a.e.\ $t\in T$;
            \item[(iv)] \hypertarget{def:normal_integrand.iv}{} \emph{convex Carath\'eodory integrand} if it is a Carath\'eodory integrand and $f(t,\cdot)\colon X\to \mathbb{R}$ is convex for $\mathcal{L}^N$-a.e.\ $t\in T$ (in particular, $f(t,\cdot)\in \Gamma_0(X)$ for $\mathcal{L}^N$-a.e.\ $t\in T$).
        \end{itemize}
    \end{definition}

     The first important results on (convex) normal integrands are summarized in the following lemma.\enlargethispage{11mm}\vspace{-0.5mm}

    \begin{lemma}\label{lem:convex_normal_integrands}
    Let $f\colon T\times X\to \mathbb{R}\cup\{\pm\infty\}$ be an integrand. Then,
    the following statements apply:
    \begin{itemize}[noitemsep,topsep=2pt,leftmargin=!,labelwidth=\widthof{(iii)}]
        \item[(i)] \hypertarget{lem:convex_normal_integrands.i}{} If $f\colon T\times X\to \mathbb{R}\cup\{+\infty\}$ is a normal integrand and $x\in L^0(T;X)$, then the composition $f(\cdot,x(\cdot))\colon T\to \mathbb{R}\cup\{+\infty\}$ is $ \mathcal{L}^N(T)$-measurable;
        \item[(ii)] \hypertarget{lem:convex_normal_integrands.ii}{} If $f\colon T\times X\to \mathbb{R}\cup\{+\infty\}$ is a  normal integrand, then the (Fenchel) conjugate (with respect to the second argument) $f^*\colon T\times X^*\to \mathbb{R}\cup\{+\infty\}$, for $ \mathcal{L}^N$-a.e.\ $t\in T$ and $x^*\in X^*$ defined by\vspace{-0.5mm}
        \begin{align*}
            f^*(t,x^*)\coloneqq\sup_{x\in X}{\smash{\bigl\{\langle x^*,x\rangle_X-f(t,x)\bigr\}}}\,,\\[-6mm]\notag
        \end{align*}
        is a convex normal integrand if $\operatorname{dom}(f^*(t,\cdot))\neq\emptyset$ for $\smash{\mathcal{L}^N}$-a.e.\ $t\in T$ (in particular, if $f\colon T\times X\to \mathbb{R}\cup\{+\infty\}$ is a convex~normal~integrand);
        \item[(iii)] \hypertarget{lem:convex_normal_integrands.iii}{} If $f(\cdot,x)\colon T\to \mathbb{R}$ is (finite-valued) $\mathcal{L}^N(T)$-measurable for all $x\in X$ and 
        $f(t,\cdot)\colon  X\to \mathbb{R}$ is lower semi-continuous for $ \mathcal{L}^N$-a.e.\ $t\in T$, then $f\colon T\times X\to \mathbb{R}$ is a normal integrand.
    \end{itemize}
        
    \end{lemma}

    \begin{remark}[on Lemma \ref{lem:convex_normal_integrands}(\hyperlink{lem:convex_normal_integrands.iii}{iii})]\label{rem:convex_normal_integrands.iii}
       The mapping $f(t,\cdot)\colon  X\to \mathbb{R}$ is (finite-valued) lower semi-continuous for $\mathcal{L}^N$-a.e.\ $t\in T$, if either of the following conditions is satisfied:
       \begin{itemize}[noitemsep,topsep=2pt,leftmargin=!,labelwidth=\widthof{$\bullet$}]
           \item[$\bullet$] $f(t,\cdot)\colon  X\to \mathbb{R}$ is (finite-valued) convex for $\mathcal{L}^N$-a.e.\ $t\in T$ and $\operatorname{dim}X<+\infty$;
            \item[$\bullet$] $f(t,\cdot)\colon  X\to \mathbb{R}$ is (finite-valued) continuous for $\mathcal{L}^N$-a.e.\ $t\in T$.
       \end{itemize}
       In particular, if $f\colon T\times X\to \mathbb{R}$ is a Carath\'eodory integrand, then it is a normal integrand as well.
    \end{remark}

    \begin{proof}[Proof (of Lemma \ref{lem:convex_normal_integrands}).]
        \emph{ad (\hyperlink{lem:convex_normal_integrands.i}{i}).} If $x\in L^0(T;X)$, then $x\colon T\to X$ is $ \mathcal{L}^N(T)$-$\mathcal{B}(X)$-measurable. Therefore, the graph map $(\cdot,x(\cdot))\colon T\to T\times X$ is $ \mathcal{L}^N(T)$-$(\mathcal{L}^N(T)\otimes\mathcal{B}(X))$-measurable as a tuple of $ \mathcal{L}^N(T)$-$ \mathcal{L}^N(T)$- and $ \mathcal{L}^N(T)$-$ \mathcal{B}(X)$-measurable functions. Eventually,  $f(\cdot,x(\cdot))\colon T\to \mathbb{R}\cup\{+\infty\}$ is $ \mathcal{L}^N(T)$-measu\-rable as a composition of an $ \mathcal{L}^N(T)$-$(\mathcal{L}^N(T)\otimes \mathcal{B}(X))$-measurable and~an~\mbox{$ \mathcal{L}^N(T)\otimes \mathcal{B}(X)$-}measurable~function.

        \emph{ad (\hyperlink{lem:convex_normal_integrands.ii}{ii}).} See \cite[Prop.\ 2]{Rockafellar1971}.

        \emph{ad (\hyperlink{lem:convex_normal_integrands.iii}{iii}).} See \cite[Lem.\ 2]{Rockafellar1968}. 
    \end{proof}

    Based on Lemma \ref{lem:convex_normal_integrands}(\hyperlink{lem:convex_normal_integrands.i}{i}), we are now in the position to introduce the notion of an \emph{integral~functional} (associated with a normal integrand).

    \begin{definition}[Integral functional]\label{def:integral_functional}
        Let  $f\colon T\times X\to \mathbb{R}\cup\{+\infty\}$ be a normal integrand. Then, the \emph{integral functional} $I_f\colon L^0(T;X)\to \mathbb{R}\cup\{\pm \infty\}$ \emph{(associated with $f$)}, for every $x\in L^0(T;X)$,~is~defined~by\footnote{Here, $(\cdot)_+\coloneqq \max\{0,\cdot\},(\cdot)_-\coloneqq- \min\{0,\cdot\}\colon\mathbb{R}\to \mathbb{R}_{\ge 0}$.}
        \begin{align*}
            I_f(x)\coloneqq \begin{cases}
                \displaystyle\int_{T}{f(t,x(t))\,\mathrm{d}t}&\text{ if }(f(\cdot,x(\cdot)))_+\in L^1(T;\mathbb{R}^1)
                \,,\\
                +\infty&\text{ else}\,.
            \end{cases}
        \end{align*}
    \end{definition}

    \begin{remark}[on Definition \ref{def:integral_functional}]
        Note that for an integral functional $I_f\colon L^0(T;X)\to \mathbb{R}\cup\{\pm \infty\}$ (associated with a normal integrand $f\colon T\times X\to \mathbb{R}\cup\{+\infty\}$), we have that $x\in \operatorname{dom}(I_f)$~if~and~only~if $(f(\cdot,x(\cdot)))_+\hspace{-0.175em}\in\hspace{-0.175em} L^1(T;\mathbb{R}^1)$, \hspace{-0.15mm}in \hspace{-0.15mm}which \hspace{-0.15mm}case \hspace{-0.15mm}$I_f(x)\hspace{-0.15em}=\hspace{-0.15em}-\infty$ \hspace{-0.15mm}(equivalent~\hspace{-0.15mm}to~\hspace{-0.15mm}$(f(\cdot,x(\cdot)))_-\hspace{-0.175em}\notin \hspace{-0.175em}L^1(T;\mathbb{R}^1)$)~\hspace{-0.15mm}is~\hspace{-0.15mm}not~\hspace{-0.15mm}\mbox{excluded}.

    \end{remark}

    While normal integrands guarantee suitable measurability properties (\textit{cf}.\ Lemma \ref{lem:convex_normal_integrands}(\hyperlink{lem:convex_normal_integrands.i}{i})), they do~not, in general, ensure lower semi-continuity of the associated integral
functional. The following property, tracing back to Ioffe
(\textit{cf}.\ \cite{Ioffe1977I}; see also \cite[Sec.\ 7]{Giner2015}, for a  description in a separable~Banach~space~setting),
provides a necessary and sufficient criterion for lower semi-continuity of the
associated integral functional.\enlargethispage{5mm}

    \begin{definition}[$p$-lower \hspace{-0.1mm}compactness \hspace{-0.1mm}property]\label{def:lower_compactness_property}
        \hspace{-0.1mm}A \hspace{-0.1mm}normal \hspace{-0.1mm}integrand \hspace{-0.1mm}$f\colon \hspace{-0.15em}T\hspace{-0.1em}\times\hspace{-0.1em} X\hspace{-0.175em}\to \hspace{-0.175em}\mathbb{R}\cup\{+\infty\}$~\hspace{-0.1mm}is~\hspace{-0.1mm}said~\hspace{-0.1mm}to~\hspace{-0.1mm}have the \emph{$p$-lower compactness property} for some $p\in [1,+\infty)$ if for each sequence $\{x_n\}_{n\in \mathbb{N}}\subseteq L^p(T;X)$~with
        \begin{subequations} \label{def:lower_compactness_property.0}
        \begin{alignat}{2}\label{def:lower_compactness_property.0.1}
           x_n\to x\quad \text{ in }L^p(T;X)\quad (n\to \infty)\,,\\\label{def:lower_compactness_property.0.2}
           \sup_{n\in \mathbb{N}}{\{I_{f}(x_n)\}}<+\infty\,,
        \end{alignat}
        \end{subequations}
        it follows that $\{(f(\cdot,x_n(\cdot))_-\}_{n\in \mathbb{N}}\subseteq  L^1(T;\mathbb{R}^1)$ is uniformly integrable.
    \end{definition}

    \begin{lemma}\label{lem:lower_compactness_propery}
        Let  $f\colon T\times X\to \mathbb{R}\cup\{+\infty\}$ be a  normal integrand. Then, the following statements apply:
    \begin{itemize}[noitemsep,topsep=2pt,leftmargin=!,labelwidth=\widthof{(ii)}]
        \item[(i)] \hypertarget{lem:lower_compactness_propery.i}{} If $f\colon T\times X\to \mathbb{R}\cup\{+\infty\}$ has the $p$-lower compactness property for some $p\in [1,+\infty)$,  then $I_f\colon L^p(T;X)\to \mathbb{R}\cup\{\pm\infty\}$ is lower semi-continuous. Conversely, if $I_f\colon L^p(T;X)\to \mathbb{R}\cup\{\pm\infty\}$ is lower semi-continuous  for some $p\in [1,+\infty)$ and $I_f(x)>-\infty$ for all $x\in L^p(T;X)$, then $f\colon T\times X\to \mathbb{R}\cup\{+\infty\}$ has the $p$-lower compactness property.
 
        \item[(ii)] \hypertarget{lem:lower_compactness_propery.ii}{} If $\operatorname{dom}(I_f)\cap L^p(T;X)\neq\emptyset $ for some $p\in (1,+\infty)$ (\textit{i.e.}, there exists $x_0\in L^p(T;X)$ such that $(f(\cdot,x_0(\cdot)))_+\hspace{-0.15em}\in \hspace{-0.15em}L^1(T;\mathbb{R}^1)$),~then~the~(Fenchel) conjugate (with respect to the second~\mbox{argument})~$f^*\colon \hspace{-0.15em} T\times X^*\to \mathbb{R}\cup\{+\infty\}$ has the $p'$-lower compactness property and $I_{f^*}(x^*)>-\infty$~for~all~${x^*\in L^{p'}(T;X^*)}$.
    \end{itemize}
    \end{lemma}


    \begin{proof}

        \emph{ad (\hyperlink{lem:lower_compactness_propery.i}{i}).} See \cite[Cor.\ 7.6]{Giner2015}.
        
        \emph{ad (\hyperlink{lem:lower_compactness_propery.ii}{ii}).} For \hspace{-0.1mm}$\mathcal{L}^N$-a.e.\ \hspace{-0.1mm}$t\hspace{-0.1em}\in\hspace{-0.1em} T$ \hspace{-0.1mm}and \hspace{-0.1mm}every \hspace{-0.1mm}$x^*\hspace{-0.1em}\in\hspace{-0.1em} X^*$, \hspace{-0.1mm}by \hspace{-0.1mm}the \hspace{-0.1mm}Fenchel--Young \hspace{-0.1mm}inequality \hspace{-0.1mm}(\textit{cf}.\ \hspace{-0.1mm}\cite[\hspace{-0.1mm}Prop.~\hspace{-0.1mm}5.1,~\hspace{-0.1mm}p.~\hspace{-0.1mm}21]{EkelandTemam1999}), we \hspace{-0.1mm}have \hspace{-0.1mm}that \hspace{-0.1mm}$f^*(t,x^*)\hspace{-0.1em}\ge\hspace{-0.1em} \langle x^*,x_0(t)\rangle_X-f(t,x_0(t))$,
        \hspace{-0.1mm}which, \hspace{-0.1mm}for \hspace{-0.1mm}$\mathcal{L}^N$-a.e.\ \hspace{-0.1mm}$t\hspace{-0.1em}\in\hspace{-0.1em} T$ \hspace{-0.1mm}and \hspace{-0.1mm}every \hspace{-0.1mm}$x^*\hspace{-0.1em}\in\hspace{-0.1em} X^*$,~\hspace{-0.1mm}\mbox{implies}~\hspace{-0.1mm}that
        \begin{align}\label{lem:lower_compactness_propery.1}
            (f^*(t,x^*))_- \leq \vert \langle x^*, x_0(t)\rangle_X\vert +(f(t,x_0(t)))_+\,.
        \end{align}
        As a consequence, if $\{x_n^*\}_{n\in \mathbb{N}}\subseteq \smash{L^{p'}(T;X^*)}$ is a sequence such that
        \begin{align} 
           x_n^*\to x^*\quad \text{ in }\smash{L^{p'}(T;X^*)}\quad (n\to \infty)\,,\label{lem:lower_compactness_propery.2.1}
        \end{align}
        then, for every $n\in \mathbb{N}$ and $A\in \mathcal{L}^N(T)$, from \eqref{lem:lower_compactness_propery.1}, it follows that
        \begin{align*}
            \int_A{(f^*(t,x^*_n(t)))_-\,\mathrm{d}t}\leq  \int_A{\vert \langle x^*_n(t),x_0(t)\rangle_X\vert\,\mathrm{d}t}+ \int_A{(f(t,x_0(t)))_+\,\mathrm{d}t}\,,
        \end{align*}
        which, \hspace*{-0.15mm}due \hspace*{-0.15mm}to \hspace*{-0.15mm}\eqref{lem:lower_compactness_propery.2.1} \hspace*{-0.15mm}and  \hspace*{-0.15mm}$(f(\cdot,x_0(\cdot)))_+\hspace{-0.15em}\in\hspace{-0.15em} L^1(T;\mathbb{R}^1)$, \hspace*{-0.15mm}implies \hspace*{-0.15mm}that \hspace*{-0.15mm}$\{(f^*(\cdot,x^*_n(\cdot)))_-\}_{n\in \mathbb{N}}\hspace{-0.15em} \subseteq\hspace{-0.15em} L^1(T;\mathbb{R}^1)$~\hspace*{-0.15mm}is~\hspace*{-0.15mm}uniform\-ly integrable. In other words, $f^*\colon T\times X^*\to \mathbb{R}\cup\{+\infty\}$ has the $p'$-lower compactness property.~Moreover,\linebreak from \eqref{lem:lower_compactness_propery.1}, for every $x^*\hspace{-0.1em}\in\hspace{-0.1em} L^{p'}(T;X^*)$, it follows that ${(f^*(\cdot,x^*(\cdot)))_-\hspace{-0.1em}\in\hspace{-0.1em} L^1(T;\mathbb{R}^1)}$,~\textit{i.e.},~${I_{f^*}(x^*)\hspace{-0.1em}>\hspace{-0.1em}-\infty}$.
    \end{proof}

    \begin{corollary}\label{cor:integral_functionals}
        Let $f\colon\hspace{-0.1em} T\times X\hspace{-0.1em}\to\hspace{-0.1em} \mathbb{R}\cup\{+\infty\}$ be a convex normal integrand such that 
        $\operatorname{dom}(I_f)\cap L^p(T;X)\hspace{-0.1em}\neq\hspace{-0.1em}\emptyset$ and  $\operatorname{dom}(I_{f^*})\cap \smash{L^{p'}(T;X^*)}\neq\emptyset$ for some $p\in (1,+\infty)$.
        Then, the  integral functional 
        $I_f\colon L^p(T;X)\to \mathbb{R}\cup\{+\infty\}$ is well-defined, proper, convex, and lower semi-continuous.
    \end{corollary}

    \begin{proof}
        To begin with, we note that, by Lemma \ref{lem:convex_normal_integrands}(\hyperlink{lem:convex_normal_integrands.ii}{ii}), the Fenchel conjugate $f^*\colon T\times X^*\to \mathbb{R}\cup\{+\infty\}$ is a convex normal integrand as well and, thus, by Lemma \ref{lem:convex_normal_integrands}(\hyperlink{lem:convex_normal_integrands.i}{i}),~${I_{f^*}\colon \hspace{-0.1em}L^0(T;X^*)\hspace{-0.1em}\to\hspace{-0.1em} \mathbb{R}\hspace{-0.1em}\cup\hspace{-0.1em}\{\pm \infty\}}$~is~\mbox{well-defined}.

        Then, according to Lemma \ref{lem:lower_compactness_propery}(\hyperlink{lem:lower_compactness_propery.ii}{ii}), $f\colon T\times X\to \mathbb{R}\cup\{+\infty\}$ has the $p$-lower compactness property  and $f^*\colon T\times X^*\to \mathbb{R}\cup\{+\infty\}$ the $p'$-lower compactness property as well as $I_f(x)>-\infty$ and $I_{f^*}(x^*)>-\infty$ for all $x\in L^p(T;X)$ and $x^*\in L^{p'}(T;X^*)$, respectively. Therefore, resorting to Lemma \ref{lem:lower_compactness_propery}(\hyperlink{lem:lower_compactness_propery.i}{i}),  we conclude that $I_f\colon L^p(T;X)\to \mathbb{R}\cup\{+\infty\}$ and $I_{f^*}\colon L^{p'}(T;X^*)\to \mathbb{R}\cup\{+\infty\}$~are~\mbox{well-defined},~proper,~convex, and lower semi-continuous.
    \end{proof}

    The following lemma identifies the Fenchel conjugate of an integral functional with the integral functional associated with the Fenchel conjugate of the underlying normal integrand.
    
    \begin{lemma}[Fenchel conjugates of integral functionals]\label{lem:conjugate_of_integral_functional}
         Let $f\colon T\times X\to \mathbb{R}\cup\{+\infty\}$ be a  normal integrand such that  $\operatorname{dom}(I_f)\cap L^p(T;X)\neq \emptyset$ for some $p\in [1,+\infty)$. Then, if,  in addition,  $X$ is reflexive,
         for the Fenchel conjugate $(I_f)^*\colon \hspace{-0.1em} L^{p'}(T;X^*)\hspace{-0.1em}\to\hspace{-0.1em} \mathbb{R}\cup\{+\infty\}$ of the integral functional $I_f\colon \hspace{-0.1em}L^p(T;X)\hspace{-0.1em}\to\hspace{-0.1em} \mathbb{R}\cup\{\pm\infty\}$,\linebreak there holds $(I_f)^*(x^*)=I_{f^*}(x^*)$~for~all~${x^*\in L^{p'}(T;X^*)}$.
    \end{lemma}

    \begin{proof}
        See \cite[Thm.\ 2]{Rockafellar1971}.
    \end{proof}
    
    In the context of subgradient flows, one naturally encounters time-dependent families of normal~integ\-rands $f\colon I\times  \Omega\times \mathbb{R}^{\ell}\to \mathbb{R}\cup\{+\infty\}$, where $I\hspace{-0.1em} \coloneqq \hspace{-0.1em} (0,t_{\mathtt{fin}})$ is a finite time interval, $\Omega \hspace{-0.1em}\in\hspace{-0.1em}\mathcal{L}^d(\mathbb{R}^d)$, $d\hspace{-0.1em}\in \hspace{-0.1em} \mathbb{N}$,~and~${\ell\hspace{-0.1em}\in \hspace{-0.1em} \mathbb{N}}$.\linebreak Inasmuch as the associated spatial integral reduction ${I_f^{\Omega}\coloneqq((t,v)\mapsto I_{f(t,\cdot,\cdot)}(v))
    \colon I\times X\to \mathbb{R}\hspace{-0.15em}\cup\hspace{-0.15em}\{\pm \infty\}}$, where $X\subseteq  L^0(\Omega;\mathbb{R}^{\ell})$, serves as a building block for~the~\mbox{construction}~of (time-)integral functionals,  it is essential that it itself be a normal integrand.~The~\mbox{following}~lemma, together with the subsequent remark and corollary, shows that this property is already guaranteed under mild assumptions on the underlying space $X$ and the integrand $f\colon I\times \Omega\times \mathbb{R}^{\ell}\to \mathbb{R}\cup\{+\infty\}$.\enlargethispage{7.5mm}

    \begin{lemma}[Integral functionals as normal integrands]\label{lem:integral_functionals_as_normal_integrands}

Let $X \subseteq L^0(\Omega;\mathbb{R}^{\ell})$, where $\Omega\in  \mathcal{L}^d(\mathbb{R}^d)$, $d\in \mathbb{N}$, with finite measure and $\ell\in \mathbb{N}$, be such that convergence in $X$ implies
convergence in $\mathcal{L}^d$-measure on $\Omega$. Moreover, let $f\colon T\times\Omega\times\mathbb{R}^{\ell}
    \to \mathbb{R}\cup\{+\infty\}$~be~a convex normal integrand. 
Then, the \emph{($t$-dependent)~family of spatial integral functionals}~(or~\emph{spatial integral reduction}) $ \smash{I_f^{\Omega}}\colon T\times X\to \mathbb{R}\cup\{\pm\infty\}$, for $\mathcal{L}^N$-a.e.\ $t\in T$ and every~${v\in X}$~defined~by
\begin{align*}
     \smash{I_f^{\Omega}}(t,v)
    \coloneqq I_{f(t,\cdot,\cdot)}(v)
    \coloneqq
    \int_\Omega{f(t,\cdot,v)\,\mathrm{d}x}\,,
\end{align*}
is $\mathcal{L}^N(T)\otimes\mathcal{B}(X)$-measurable. 
If, in addition, the integral functional $\smash{I_f^{\Omega}}(t,\cdot)\colon X\to \mathbb{R}\cup\{+\infty\}$ is proper and lower semi-continuous for $\mathcal{L}^N$-a.e.\ $t\in T$, then  $ \smash{I_f^{\Omega}}\colon T\times X\to \mathbb{R}\cup\{+\infty\}$ is~a~convex~normal~integrand.
\end{lemma} 

\begin{remark}[on Lemma \ref{lem:integral_functionals_as_normal_integrands}]\label{rem:integral_functionals_as_normal_integrands}
    By the standard section property of product $\sigma$-algebras (\textit{cf}.\ \cite[Kap.~V, §1, Lem.\ 1.1]{Elstrodt2018}) and the Fubini--Tonelli theorem (\textit{cf}.\ \cite[Kap.\ V, §2, Satz 2.1]{Elstrodt2018}), for $\mathcal{L}^N$-a.e.~$t\in T$, we observe~that $f(t,\cdot,\cdot)\colon \Omega\times\mathbb{R}^{\ell}
    \to \mathbb{R}\cup\{+\infty\}$ is a convex normal integrand  and, consequently, the associated integral functional  $\smash{I_f^{\Omega}}(t,\cdot)\colon X\to \mathbb{R}\cup\{\pm\infty\}$ is well-defined (in the sense of Definition \ref{def:integral_functional}).

\end{remark}

\begin{proof}[Proof (of Lemma \ref{lem:integral_functionals_as_normal_integrands}).] The proof closely follows  the proof of \cite[Thm.\ 20]{PennanenPerkkio2018}.\enlargethispage{1.5mm}

    To begin with, it suffices to consider the case that  $f\colon T\times\Omega\times\mathbb{R}^{\ell}
    \to \mathbb{R}\cup\{+\infty\}$~is~bounded~from~\mbox{below}. Otherwise, we consider the ($t$-dependent) family of normal integrands $f^\alpha\colon T\times\Omega\times\mathbb{R}^{\ell}
    \to \mathbb{R}\cup\{+\infty\}$, ${\alpha \in (-\infty,0)}$, for $\mathcal{L}^N\otimes\mathcal{L}^d$-a.e.\ $(t,x)\in T\times \Omega$ and every $\xi\in \mathbb{R}^{\ell}$ defined by 
    \begin{align*}
        f^\alpha(t,x,\xi)\coloneqq \sup\{f(t,x,\xi),\alpha\}\,,
    \end{align*}
    and, if we can verify that $\smash{I_{f^\alpha}^{\Omega}}\colon T\times X\to \mathbb{R}\cup\{+\infty\}$ is
    $\mathcal{L}^N(T)\otimes\mathcal{B}(X)$-measurable for all $\alpha \in (-\infty,0)$,~then, for $\mathcal{L}^N$-a.e.\ $t\in T$ and every $v\in X$, since 
    \begin{align*}
        f^\alpha(t,x,v(x))\searrow f(t,x,v(x))\quad ( \alpha \to -\infty) \quad \text{ for $\mathcal{L}^d$-a.e.\ }x\in \Omega\,,
    \end{align*}
    we obtain\vspace{-0.5mm}
    \begin{align}\label{lem:integral_functionals_as_normal_integrands.0}
       \smash{\lim_{\alpha\to-\infty}{\smash{\bigl\{\smash{I_{f^\alpha}^{\Omega}}(t,v)\bigr\}}}}= \smash{I_f^{\Omega}}(t,v)
\,,
    \end{align}
    and, thus, that $\smash{I_f^{\Omega}}\colon T\times X\to \mathbb{R}\cup\{\pm\infty\}$ is
    $\mathcal{L}^N(T)\otimes\mathcal{B}(X)$-measurable. For the convergence \eqref{lem:integral_functionals_as_normal_integrands.0}, it is enough  to distinguish the following two cases:
    \begin{itemize}[noitemsep,topsep=2pt,leftmargin=!,labelwidth=\widthof{$\bullet$}]
        \item[$\bullet$] \emph{Case 1.} If $(f(t,\cdot,v))_+\in L^1(\Omega;\mathbb{R}^1)$, then, due to  $f^\alpha(t,x,v(x))\leq (f(t,x,v(x)))_+$~for~\mbox{$\mathcal{L}^d$-a.e.}~${x\in  \Omega}$, by Beppo Levi's monotone convergence theorem (\textit{cf}.\ \cite[Satz 2.7, p.\ 139]{Elstrodt2018}), we find that
        \begin{align*}
             \smash{\lim_{\alpha\to-\infty}{\smash{\bigl\{\smash{I_{f^\alpha}^{\Omega}}(t,v)\bigr\}}}}=\smash{I_f^{\Omega}}(t,v)\,;
        \end{align*}  
        \item[$\bullet$] \emph{Case 2.} If $(f(t,\cdot,v))_+\notin L^1(\Omega;\mathbb{R}^1)$, then, due to  $(f^\alpha(t,x,v(x)))_+= (f(t,x,v(x)))_+$ for $\mathcal{L}^d$-a.e.\ $x\in \Omega$, by Definition \ref{def:integral_functional}, for every $\alpha\in (-\infty,0)$, we find that
        \begin{align*}
            \smash{I_{f^\alpha}^{\Omega}}(t,v)=+\infty=\smash{I_{f}^{\Omega}}(t,v)\,.
        \end{align*} 
    \end{itemize}

For this reason, it suffices to establish the assertion under the additional assumption that~there~exists a constant $\alpha\in (-\infty,0)$ such that for $\mathcal{L}^N\otimes\mathcal{L}^d$-a.e.\ $(t,x)\in T\times \Omega$ and every $\xi\in \mathbb{R}^{\ell}$,~we~have~that
\begin{align}\label{lem:integral_functionals_as_normal_integrands.1}
    f(t,x,\xi)\ge \alpha\,.
\end{align}

Then, under the assumption \eqref{lem:integral_functionals_as_normal_integrands.1}, the proof follows in three steps:

\emph{Step 1.}\hypertarget{step 1}{} In this step, we assume that there exists a constant $\beta\in (0,\infty)$ such that for $\mathcal{L}^N\otimes\mathcal{L}^d$-a.e.\ $(t,x)\in T\times \Omega$ and every $\xi\in \mathbb{R}^{\ell}$, we have that
\begin{align}\label{lem:integral_functionals_as_normal_integrands.2}
    f(t,x,\xi)\leq \beta\,,
\end{align}
and that $f(t,x,\cdot)\colon \mathbb{R}^{\ell}\to \mathbb{R}$
is continuous for $\mathcal{L}^N\otimes\mathcal{L}^d$-a.e.\ $(t,x)\in T\times \Omega$. Next, let $\{v_j\}_{j\in \mathbb{N}}\subseteq X$ be a sequence such that $v_j\to v$ in $X$ $(j\to \infty)$. By the additional assumption on the space $X$, this implies that $v_j\to v$ in measure on
$\Omega$ $(j\to \infty)$ and for a subsequence $v_{j'}(x)\to v(x)$ $(j'\to \infty)$ for $\mathcal{L}^d$-a.e.~$x\in\Omega$, so that, by the continuity of  $f(t,x,\cdot)\colon \mathbb{R}^{\ell}\to \mathbb{R}$
for $\mathcal{L}^N\otimes\mathcal{L}^d$-a.e.\ $(t,x)\in T\times \Omega$,~we~infer~that  
\begin{align}\label{lem:integral_functionals_as_normal_integrands.3}
    f(t,x,v_{j'}(x))\to f(t,x,v(x))\quad (j'\to \infty)\quad\text{ for $\mathcal{L}^N\otimes\mathcal{L}^d$-a.e.\ }(t,x)\in T\times \Omega\,.
\end{align}
Then, owing to \eqref{lem:integral_functionals_as_normal_integrands.2} and \eqref{lem:integral_functionals_as_normal_integrands.3}, by 
Lebesgue's dominated convergence theorem (\textit{cf}.\ \cite[Satz~5.2,~p.~160]{Elstrodt2018}), 
we obtain
\begin{align}\label{lem:integral_functionals_as_normal_integrands.4}
    \smash{\smash{I_f^{\Omega}}}(t,v_{j'})\to \smash{\smash{I_f^{\Omega}}}(t,v)\quad (j'\to \infty)\quad \text{ for $\mathcal{L}^N$-a.e.\ }t\in T\,.
\end{align}
Moreover, 
by the standard convergence principle (\textit{cf}.\ \cite[Kap.\ I, Lem.\ 5.4]{GajewskiGroegerZacharias1974}), we find that~the~\mbox{convergence} \eqref{lem:integral_functionals_as_normal_integrands.4} applies for the entire sequence and, therefore, that $\smash{I_f^{\Omega}}(t,\cdot)\colon X\to \mathbb{R}$ is continuous~for~\mbox{$\mathcal{L}^N$-a.e.}~$t\in T$. Since, by the Fubini--Tonelli theorem  (\textit{cf}.\ \cite[Kap.\ V, §2, Satz 2.1]{Elstrodt2018})
$\smash{I_f^{\Omega}}(\cdot,v)\colon T\to \mathbb{R}$ is $\mathcal{L}^N(T)$-measurable for all $v\in X$, we conclude that 
$\smash{I_f^{\Omega}}\colon T\times X\to \mathbb{R}$ is a Carath\'eodory integrand and, consequently,~according to Remark~\ref{rem:convex_normal_integrands.iii}, a normal integrand and, in particular, $\mathcal{L}^N(T)\otimes\mathcal{B}(X)$-measurable.\enlargethispage{2.5mm}

\emph{Step 2.}\hypertarget{step 2}{} In this step, we assume that $f(t,x,\cdot)\colon \mathbb{R}^{\ell}\to \mathbb{R}$
is continuous~for~$\mathcal{L}^N\otimes\mathcal{L}^d$-a.e.~$(t,x)\in T\times \Omega$. 
Next, we consider the ($t$-dependent) family of normal integrands $f_\beta\colon T\times\Omega\times\mathbb{R}^{\ell}
    \to \mathbb{R}$, $\beta\in (0,+\infty)$, for $\mathcal{L}^N\otimes\mathcal{L}^d$-a.e.\ $(t,x)\in T\times \Omega$ and every $\xi\in \mathbb{R}^{\ell}$ defined by 
    \begin{align*}
        f_\beta(t,x,\xi)\coloneqq \inf\{f(t,x,\xi),\beta\}\,.
    \end{align*}
Then, by Step \hyperlink{step 1}{1}, $I_{f_\beta}^{\Omega}\colon T\times X\to \mathbb{R}$ is a Carath\'eodory integrand and, consequently, according~to~\mbox{Remark}~\ref{rem:convex_normal_integrands.iii}, a normal integrand and, in particular, $\mathcal{L}^N(T)\otimes\mathcal{B}(X)$-measurable. Therefore, for $\mathcal{L}^N$-a.e.\ $t\in T$ and every $v\in X$, since 
\begin{align*}
     \alpha \leq f_\beta(t,x,v(x))\nearrow f(t,x,v(x))\quad ( \beta \to +\infty)\quad\text{ for $\mathcal{L}^d$-a.e.\ }x\in \Omega\,,
\end{align*}
by Beppo Levi's monotone convergence theorem (\textit{cf}.\ \cite[Satz~2.7,~p.~139]{Elstrodt2018}), 
we obtain\vspace{0.5mm} 
    \begin{align*}
       \lim_{\beta\to+\infty}{\smash{\bigl\{I_{f_\beta}^{\Omega}(t,v)\bigr\}}}= I_f^{\Omega}(t,v)\,,
    \end{align*}
    and, in particular, that $\smash{I_f^{\Omega}}\colon T\times X\to \mathbb{R}\cup\{+\infty\}$~is~\mbox{$\mathcal{L}^N(T)\otimes\mathcal{B}(X)$-measurable}.

    \emph{Step 3.}\hypertarget{step 3}{} In this case, eventually, we merely assume that $ f(t,x,\cdot)\colon \mathbb{R}^{\ell}\to \mathbb{R}\cup\{+\infty\}$ is lower semi-continuous for $\mathcal{L}^N\otimes\mathcal{L}^d$-a.e.\ $(t,x)\in T\times \Omega$. Next, we consider the family of Moreau~envelopes (with respect to the last argument, \textit{cf}.\ \cite[Def.\ 1.22]{RockafellarWets1998}) $e_\lambda(f)\colon T\times \Omega\times \mathbb{R}^{\ell}\to \mathbb{R}$, $\lambda>0$, for every $\mathcal{L}^N\otimes\mathcal{L}^d$-a.e.\ $(t,x)\in T\times \Omega$ and every $\xi\in \mathbb{R}^{\ell}$ defined by 
    \begin{align*}
          e_\lambda(f)(t,x,\xi)
    \coloneqq
    \inf_{\eta\in\mathbb{R}^{\ell}}
    {\bigl\{
        f(t,x,\eta)
        +
        \tfrac{1}{2\lambda}|\xi-\eta|^2
    \bigr\}}\,,
    \end{align*}
    which, owing to \cite[Ex.\ 14.38]{RockafellarWets1998}, are Carath\'eodory integrands and, consequently, according to~\mbox{Remark}~\ref{rem:convex_normal_integrands.iii}, normal integrands, so that, by Step \hyperlink{step 2}{2}, $I_{e_\lambda(f)}^{\Omega}\colon \hspace{-0.1em}T\times X\hspace{-0.1em}\to \hspace{-0.1em}\mathbb{R}\cup\{+\infty\}$ is $\mathcal{L}^N(T)\otimes\mathcal{B}(X)$-measurable~for~all~${\lambda\hspace{-0.1em}>\hspace{-0.1em}0}$. 
    Therefore, for $\mathcal{L}^N$-a.e.\ $t\in T$ and every $v\in X$, 
    since, according to \cite[Thm.\ 1.25]{RockafellarWets1998}, we have that 
    \begin{align*}
       \alpha \leq e_\lambda(f)(t,x,\xi)\nearrow f(t,x,\xi)\quad (\lambda\to 0^+)\quad \text{ for $\mathcal{L}^d$-a.e.\ }x\in \Omega\,,
    \end{align*}
     by Beppo Levi's monotone convergence theorem (\textit{cf}.\ \cite[Satz 2.7, p.\ 139]{Elstrodt2018}), we have that\vspace{0.5mm}   
     \begin{align*}
          \lim_{\lambda\to 0^+}{\smash{\bigl\{I_{e_\lambda(f)}^{\Omega}(t,v)\bigr\}}}=I_f^{\Omega}(t,v)\,,
     \end{align*}
    we conclude that  $\smash{I_f^{\Omega}}\colon T\times X\to \mathbb{R}\cup\{+\infty\}$ is $\mathcal{L}^N(T)\otimes\mathcal{B}(X)$-measurable. 

    If, in addition, $\smash{I_f^{\Omega}}(t,\cdot)\colon X\to \mathbb{R}\cup\{+\infty\}$ is proper and lower semi-continuous for $\mathcal{L}^N$-a.e.\ $t\in T$,~then, by Definition \ref{def:normal_integrand}(\hyperlink{def:normal_integrand.ii}{ii}),  $ I_f^{\Omega}\colon T\times X\to \mathbb{R}\cup\{+\infty\}$ is a convex normal integrand.
\end{proof}

    \begin{corollary}\label{cor:integral_functionals_as_normal_integrands}
        Let the assumptions of Lemma \ref{lem:integral_functionals_as_normal_integrands} be satisfied with $X=L^p(\Omega;\mathbb{R}^{\ell})$, where $\Omega\in \mathcal{L}^d(\mathbb{R}^d)$, $d\in \mathbb{N}$, with finite measure, $\ell\in \mathbb{N}$, and $p\in (1,+\infty)$. Then, the following statements apply:
        \begin{itemize}[noitemsep,topsep=2pt,leftmargin=!,labelwidth=\widthof{(ii)}]
            \item[(i)]\hypertarget{cor:integral_functionals_as_normal_integrands.i}{} If $\operatorname{dom}(I_f^{\Omega}(t,\cdot))\cap X\neq\emptyset$ and  $\operatorname{dom}(I_{f^*}^{\Omega}(t,\cdot))\cap X^*\neq\emptyset$ for $\mathcal{L}^N$-a.e.\ $t\in T$, then $\smash{I_f^{\Omega}}\colon T\times X\to \mathbb{R}\cup\{+\infty\}$ is a convex normal integrand;
            \item[(ii)]\hypertarget{cor:integral_functionals_as_normal_integrands.ii}{} If there exist $x\hspace{-0.1em}\in\hspace{-0.1em} L^{p}(T;X)$ and $x^*\hspace{-0.1em}\in\hspace{-0.1em} L^{p'}(T;X^*)$ such that 
        $(I_f^{\Omega}(\cdot,x(\cdot)))_+,(I_{f^*}^{\Omega}(\cdot,x^*(\cdot)))_+\hspace{-0.1em}\in\hspace{-0.1em} L^1(T;\mathbb{R}^1)$, then $I_{\smash{I_f^{\Omega}}}\colon L^p(T; X)\to \mathbb{R}\cup\{+\infty\}$ is well-defined, proper, convex, and lower semi-continuous.
        \end{itemize} 
    \end{corollary}

    \begin{proof}
        \emph{ad (\hyperlink{cor:integral_functionals_as_normal_integrands.i}{i}).} The assertion is a direct consequence of Corollary \ref{cor:integral_functionals}.

        \emph{ad (\hyperlink{cor:integral_functionals_as_normal_integrands.ii}{ii}).} By the additional assumption, we have that 
        $\operatorname{dom}(I_f^{\Omega}(t,\cdot))\cap X\neq\emptyset$ and  $\operatorname{dom}(I_{f^*}^{\Omega}(t,\cdot))\cap X^*\neq\emptyset$ for $\mathcal{L}^N$-a.e.\ $t\in T$, so that, resorting to point (\hyperlink{cor:integral_functionals_as_normal_integrands.i}{i}), we infer that $\smash{I_f^{\Omega}}\colon T\times X\to \mathbb{R}\cup\{+\infty\}$ is a convex normal integrand and, thus, the associated integral functional ${I_{\smash{I_f^{\Omega}}}\colon L^p(T;X)\to \mathbb{R}\cup\{\pm \infty\}}$~is~\mbox{well-defined}. The additional assumption, then, reads $\operatorname{dom}(I_{\smash{I_f^{\Omega}}})\cap L^{p}(T;X)\neq\emptyset$ and  $\operatorname{dom}(I_{\smash{(I_f^{\Omega})^*}})\cap L^{p'}(T;X^*)\neq\emptyset$, so that the assertion is again a direct consequence of Corollary \ref{cor:integral_functionals}.
    \end{proof}

    \subsection{Function spaces}\enlargethispage{5mm}

    \hspace{5mm}For a (Lebesgue) measurable set $\omega\subseteq \mathbb{R}^N$, $N\in \mathbb{N}$, and  (Lebesgue) measurable functions, vector or tensor fields $v,w\colon \omega\to \mathbb{R}^{n}$, $n\in\{1,\ell,\ell\times d\}$, $\ell\in \mathbb{N}$, we employ the inner product $(v,w)_{\omega}\coloneqq \int_{\omega}{v\odot w\,\mathrm{d}x}$,
	whenever the right-hand side is well-defined, where $\odot\colon \mathbb{R}^{\ell}\times \mathbb{R}^{\ell}\to \mathbb{R}$ either denotes scalar multiplication (\textit{i.e.}, $v\odot w=vw$), the Euclidean inner product (\textit{i.e.}, $v\odot w=v\cdot w$), or the Frobenius inner product (\textit{i.e.}, $v\odot w=v:w$). Moreover, for a (Lebesgue) measurable set $\omega\subseteq \mathbb{R}^N$, $N\in \mathbb{N}$, and a (Lebesgue) measurable function, vector, or tensor field $v\colon \omega\to \mathbb{R}^{n}$, $n\in\{1,\ell,\ell\times d\}$, $\ell\in \mathbb{N}$, we set $\|v\|_{p,\omega}\coloneqq(\int_\omega |v(x)|^p\,\mathrm{d}x)^{1/p}$, $\|v\|_{\infty,\omega}\coloneqq\operatorname*{ess\,sup}_{x\in \omega}{\{|v(x)|\}}$, and abbreviate $\|v\|_{\omega}\coloneqq\|v\|_{2,\omega}$.\newpage

    \subsubsection{Function spaces for the steady setting}\vspace{-0.5mm}

    \hspace{5mm}For the rest of the paper,  let ${\Omega\subseteq \mathbb{R}^d}$, ${d\in\mathbb{N}}$, be a bounded Lipschitz domain such that its (topological) boundary $\partial\Omega$ is divided into two disjoint
	(relatively) open sets: a non-empty Dirichlet part $\Gamma_D\subseteq \partial\Omega$ 
    and a Neumann part $\Gamma_N\subseteq \partial\Omega$ such that $\overline{\Gamma}_D\cup\overline{\Gamma}_N=\partial\Omega$ and $\Gamma_D\cap\Gamma_N=\emptyset$. Then, for $\ell\in \mathbb{N}$~and~${p\in [1,+\infty]}$, we define\vspace{-0.5mm}
	\begin{align*} 
        W^{1,p}(\Omega;\mathbb{R}^{\ell})&\coloneqq \bigl\{v\in L^p(\Omega;\mathbb{R}^{\ell})\mid \nabla v\in L^p(\Omega;\mathbb{R}^{\ell\times d})\bigr\}\,,\\
		W^{\operatorname{div},\smash{p'}}(\Omega;\mathbb{R}^{\ell})&\coloneqq \bigl\{y\in L^{\smash{p'}}(\Omega;\mathbb{R}^{\ell\times d})\mid \operatorname{div} y\in L^{\smash{p'}}(\Omega;\mathbb{R}^{\ell})\bigr\}\,,\\[-6mm]
	\end{align*}
    where the divergence is to be understood row-wise  (\textit{i.e.}, if $y=\smash{\{y_{ij}\}_{i\in \{1,\ldots,\ell\},j\in  \{1,\ldots,d\}}}\in \smash{L^{\smash{p'}}(\Omega;\mathbb{R}^{\ell\times d})}$, then $(\operatorname{div} y)_i \coloneqq \smash{\sum_{j=1}^d{\partial_jy_{ij}}}$ for all $i= 1,\ldots,\ell$) and initially (as the gradient) in a distributional sense.

	Denote by $\textup{tr}(\cdot)\colon  W^{1,p}(\Omega;\mathbb{R}^{\ell})\to W^{\smash{1-\frac{1}{p}},p}(\partial\Omega;\mathbb{R}^{\ell})$ and ${\textup{tr}((\cdot)n)\colon W^{\operatorname{div},\smash{p'}}(\Omega;\mathbb{R}^{\ell})\to (W^{\smash{1-\frac{1}{p}},p}(\partial\Omega;\mathbb{R}^{\ell}))^*}$ the trace operator  and the normal trace operator, respectively, where $n\colon \partial\Omega\hspace{-0.1em}\to\hspace{-0.1em} \mathbb{S}^{d-1}\hspace{-0.1em}\coloneqq\hspace{-0.1em}{\{x\hspace{-0.1em}\in\hspace{-0.1em} \mathbb{R}^d\mid \vert x\vert\hspace{-0.1em}=\hspace{-0.1em}1\}}$ denotes the outer unit normal vector~field~to~$\partial\Omega$. 
	Then,  for every $v\hspace{-0.1em}\in\hspace{-0.1em}  W^{1,p}(\Omega;\mathbb{R}^{\ell})$~and~${y\hspace{-0.1em}\in\hspace{-0.1em} W^{\operatorname{div},\smash{p'}}(\Omega;\mathbb{R}^{\ell})}$, there holds the spatial integration-by-parts formula (\textit{cf}.\ \cite[Sec.\ 4.3, (4.12)]{ErnGuermond2020}):\vspace{-0.5mm}
	\begin{align*}
		\smash{(\nabla v,y)_{\Omega}+(v,\operatorname{div}\,y )_{\Omega}=\langle \textup{tr}(y n),\textup{tr}(v)\rangle_{W^{1-\smash{\frac{1}{p}},p}(\partial\Omega;\mathbb{R}^{\ell})}\,.}\\[-6mm]\notag
	\end{align*} 
	Then, we define\vspace{-0.5mm}
	\begin{align*} 
		W^{1,p}_D(\Omega;\mathbb{R}^{\ell})&\coloneqq  \bigl\{v\in 	 W^{1,p}(\Omega;\mathbb{R}^{\ell}) \mid \textup{tr}(v)=0\textup{ a.e.\ on }\Gamma_D\bigr\}\,,\\
		\smash{W_N^{\operatorname{div},\smash{p'}}(\Omega;\mathbb{R}^{\ell})}&\coloneqq  \bigl\{y\in \smash{W^{\operatorname{div},\smash{p'}}(\Omega;\mathbb{R}^{\ell})}\mid \langle\textup{tr}(y n),\textup{tr}(v)\rangle_{\partial\Omega} =0\text{ for all  }v\in W^{1,p}_D(\Omega;\mathbb{R}^{\ell})\bigr\}\,.\\[-6mm]\notag
	\end{align*}
	In what follows, for the sake~of~readability, we omit writing both $\textup{tr}(\cdot)$ and $\textup{tr}((\cdot)n)$ and, for $p\in (1,+\infty)$, we employ the abbreviations\vspace{-0.5mm}
    \begin{align*}
        V&\coloneqq W^{1,p}_D(\Omega;\mathbb{R}^{\ell})\,, \qquad \qquad\;\, H\coloneqq L^2(\Omega;\mathbb{R}^{\ell})\,,\\
        Y&\coloneqq L^p(\Omega;\mathbb{R}^{\ell\times d})\,, \qquad Y^*(\operatorname{div})\coloneqq \smash{W_N^{\operatorname{div},\smash{p'}}(\Omega;\mathbb{R}^{\ell})}\,.
    \end{align*}

    \subsubsection{Function spaces for the unsteady setting}\label{subsubsec:function_spaces_unsteady}\vspace{-0.5mm}

    \hspace{5mm}For the rest of the paper,  let $I\coloneqq (0,t_{\mathtt{fin}})$, where $t_{\mathtt{fin}}\in (0,+\infty)$, be a finite time-interval. 
    For $p\in (1,+\infty)$ with $p\ge \frac{2d}{d+2}$, a function $v\in  L^p(I;V)$ is said to have a \emph{$p'$-integrable generalized time-derivative} if there exists a function $v^*\in L^{p'}(I;V^*)$ such that for every $\varphi\in C_{\mathrm{c}}^1(I)$,~there~holds\vspace{-0.5mm}\begin{align}\label{subsubsec:function_spaces_unsteady.1}
        -\int_{I}{v(t)\varphi'(t)\,\mathrm{d}t}=\int_{I}{v^*(t)\varphi(t)\,\mathrm{d}t}\quad \text{ in }V^*\,,\\[-6mm]\notag
    \end{align}
    in which case, we define $\partial_t v\coloneqq v^*$.
    Here, the $V$-valued Bochner integral on the left-hand~side~of~\eqref{subsubsec:function_spaces_unsteady.1} is interpreted as $V^*$-valued Bochner integral by means of the Evolution triple structure~(\textit{cf}.~\cite[Sec.~23.4]{Zeidler1990IIA})\vspace{-0.5mm}
    \begin{align*}
        V\xhookrightarrow[\smash{\quad\mathrm{dense}\quad}]{\quad \iota_2 \quad} H\smash{\overset{R_H}{\cong}}H^*\xhookrightarrow[\smash{\quad\mathrm{dense}\quad}]{\quad \iota_2^* \quad} \; V^*\,,\\[-6mm]\notag
    \end{align*}
    where $\iota_2^*\colon H^*\to V^*$ is the adjoint operator of the identity mapping $\iota_2\coloneqq\operatorname{id}_{V\to H} \colon V\to H$ and $R_H\colon H\to H^*$ the Riesz isomorphism, whose explicit notation~we~will~omit~in~the~\mbox{following}.

    Then, for $p\in (1,+\infty)$ with $p\ge \frac{2d}{d+2}$, the Bochner--Sobolev space of $p$-integrable functions with $p'$-integrable generalized time-derivative is denoted by\vspace{-0.5mm}
    \begin{align*}
        \mathcal{W}(I)\coloneqq\smash{\bigl\{ v\in L^p(I;V)\mid  \partial_t v\in L^{p'}(I;V^*)\bigr\}}\,,\\[-6mm]\notag
    \end{align*}
    and forms a Banach space when equipped with the norm $\|\cdot\|_{\mathcal{W}(I)}\coloneqq \|\cdot\|_{L^p(I;V)}+\|\partial_t(\cdot)\|_{\smash{L^{p'}(I;V^*)}}$ (\textit{cf}.\ \cite[Prop.\ 23.23(i)]{Zeidler1990IIA}). Moreover, 
    \hspace{-0.15mm}for \hspace{-0.15mm}every \hspace{-0.15mm}$v,w\hspace{-0.15em}\in\hspace{-0.15em} \mathcal{W}(I)$, \hspace{-0.15mm}there \hspace{-0.15mm}exist \hspace{-0.15mm}not \hspace{-0.15mm}relabelled~\hspace{-0.15mm}\mbox{representatives}~\hspace{-0.15mm}${v,w\hspace{-0.15em}\in \hspace{-0.15em}\smash{C^0(\overline{I},H)}}$ and, for every $t,t'\hspace{-0.15em}\in\hspace{-0.15em} \overline{I}$, there holds the temporal integration-by-parts formula  (\textit{cf}.\ \cite[Prop.\ 23.23(ii)\&(iv)]{Zeidler1990IIA})\vspace{-0.5mm}
    \begin{align}\label{subsubsec:function_spaces_unsteady.3}
        \int_{t'}^{t}{\langle \partial_t v(s),w(s)\rangle_V\,\mathrm{d}s}=(v(t),w(t))_H-(v(t'),w(t'))_H-\int_{t'}^{t}{\langle \partial_t w(s),v(s)\rangle_V\,\mathrm{d}s}\,.\\[-6mm]\notag
    \end{align}  
    
    \section{Fenchel duality framework for subgradient flows}\label{sec:fenchel_duality_framework}

    \hspace{5mm}In this section, we introduce a Fenchel duality framework for a broad class of subgradient flows induced by 
    convex integral functionals based on the Br\'ezis--Ekeland--Nayroles~\mbox{principle}~(\textit{cf}.~\mbox{\cite{BrezisEkeland1976I,BrezisEkeland1976II}}~and~\mbox{\cite{Nayroles1976I,Nayroles1976II}}, respectively; see also \cite[Sec.\ 3.9]{AubinCellina1984}, \cite[Sec.\ 8.10]{Roubicek2013}, \cite{Stefanelli2008}, and \cite{CariniJensenNuernberg2023} for a recent numerical application).

    In what follows, unless otherwise specified, 
    let $Q\coloneqq I\times \Omega$ be a finite time-space cylinder, let $u_0\in H$ be an initial value, and let
    $\phi\colon Q\times \mathbb{R}^{\ell\times d}\to \mathbb{R}\cup\{+\infty\}$ and $\psi\colon Q\times \mathbb{R}^{\ell}\to \mathbb{R}\cup\{+\infty\}$ be energy densities such that the following assumptions are satisfied. 

    \begin{assumption}[on $\phi$ and $\psi$]\label{ass:energy_densities}
    	The energy densities 
        $\phi\colon Q\times \mathbb{R}^{\ell\times d}\to \mathbb{R}\cup\{+\infty\}$ and $\psi\colon Q\times \mathbb{R}^{\ell}\to \mathbb{R}\cup\{+\infty\}$ satisfy the following conditions:
    	\begin{itemize}[noitemsep,topsep=2pt,leftmargin=!,labelwidth=\widthof{(ii)}]
    		
            \item[(i)]\hypertarget{ass:energy_densities.i}{} \emph{Convex normal integrands} (\textit{cf}.\ Definition \ref{def:normal_integrand}): 
    		\begin{itemize}[noitemsep,topsep=2pt,leftmargin=!,labelwidth=\widthof{(i.b)}]
    			\item[(i.a)]\hypertarget{ass:energy_densities.i.a}{} $\phi\colon Q\times \mathbb{R}^{\ell\times d}\to \mathbb{R}\cup\{+\infty\}$ is a convex normal integrand; 
    			\item[(i.b)]\hypertarget{ass:energy_densities.i.b}{} $\psi\colon Q\times \mathbb{R}^{\ell}\to \mathbb{R}\cup\{+\infty\}$ is a convex normal integrand.
    		\end{itemize} 
            

              \item[(ii)]\hypertarget{ass:energy_densities.ii}{}
            \emph{Non-triviality:} there exist $u^\star\in L^p(I;V)$, $z^\star\in L^{p'}(I;Y^*)$, and $w^\star\in L^{p'}(Q;\mathbb{R}^{\ell})$~such~that
            \begin{align*}
                (I_{\phi}^{\Omega}(\cdot,\nabla u^\star))_+,(I_{\psi}^{\Omega}(\cdot, u^\star))_+, (I_{\phi^*}^{\Omega}(\cdot,z^\star))_+,(I_{\psi^*}^{\Omega}(\cdot, w^\star))_+\in L^1(I)\,,
            \end{align*}
             where $\phi^*\colon Q\times \mathbb{R}^{\ell\times d}\to \mathbb{R}\cup\{+\infty\}$ and $\psi^*\colon Q\times \mathbb{R}^{\ell}\to \mathbb{R}\cup\{+\infty\}$ denote the Fenchel conjugates (with respect to the last argument) of the energy densities  $\phi$ and $\psi$, respectively. 
    	\end{itemize} 
    \end{assumption} 

    \subsection{Fenchel duality framework for steady convex integral functionals} 

    \hspace*{5mm}In this subsection, we recall a Fenchel duality framework for a broad class of (time-dependent) convex integral functionals. For a detailed presentation, we refer to the textbook \cite{EkelandTemam1999} (see also \cite{Rockafellar1968,Rockafellar1971,RockafellarWets1998,BartelsKaltenbach2026}).

    If
    Assumption \ref{ass:energy_densities} is satisfied, given a  right-hand side $f\in L^{p'}(I;V^*)$, for a.e.~fixed~${t\in I}$,~we~introduce the \emph{(time-dependent) steady primal energy functional} $E(t,\cdot)\colon V\to \mathbb{R}\cup\{+\infty\}$, for every $v\in V$~defined~by 
    \begin{align}\label{eq:primal_steady}
        E(t,v)\coloneqq I_{\phi}^{\Omega}(t,\nabla v)+I_{\psi}^{\Omega}(t,v)
      -\langle f(t),v\rangle_V\,, 
    \end{align}
    and we refer to the problem that seeks to minimize the (time-dependent) steady primal energy functional \eqref{eq:primal_steady} as the \emph{(time-dependent) steady primal problem}.
    
    Then, for a.e.\ fixed time $t\in I$, introducing the steady integral functional $G(t,\cdot) \colon Y\to \mathbb{R}\cup\{+\infty\}$ (associated with $\phi(t,\cdot,\cdot)\colon \Omega\times \mathbb{R}^{\ell\times d}\to \mathbb{R}\cup\{+\infty\}$) and the affinely perturbed steady integral functional  $F(t,\cdot)\colon V\to \mathbb{R}\cup\{+\infty\}$ (associated with $\psi(t,\cdot,\cdot)\colon \Omega\times \mathbb{R}^{\ell}\to \mathbb{R}\cup\{+\infty\}$), for every $y\in Y$ and $v\in V$, respectively, defined by 
    \begin{subequations}\label{eq:G_F_steady} 
    \begin{align}\label{eq:G_steady} 
        G(t,y)&\coloneqq I_{\phi}^{\Omega}(t,y)
        \,,\\
        F(t,v)&\coloneqq I_{\psi}^{\Omega}(t,v)-\langle f(t),v\rangle_{V}
        \,,\label{eq:F_steady} 
    \end{align} 
     \end{subequations}
    as well as the notation $L\coloneqq \nabla \colon V\to Y$, 
    the \emph{steady dual energy functional} $D(t,\cdot)\colon Y^*\to \mathbb{R}\cup\{-\infty\}$, for  every $y\in Y^*$, is defined by 
    \begin{align}\label{eq:dual_steady}
        D(t,y)\coloneqq -G^*(t,y)-F^*(t,-L^* y)\,, 
    \end{align} 
    where $G^*(t,\cdot)\colon \hspace{-0.1em}Y^*\hspace{-0.1em}\to\hspace{-0.1em} \mathbb{R}\cup\{+\infty\}$ and $F^*(t,\cdot)\colon \hspace{-0.1em}V^*\hspace{-0.1em}\to\hspace{-0.1em} \mathbb{R}\cup\{+\infty\}$ denote the Fenchel~conjugates~(with~respect to \hspace{-0.1mm}the \hspace{-0.1mm}second \hspace{-0.1mm}argument) \hspace{-0.1mm}of \hspace{-0.1mm}the \hspace{-0.1mm}integral \hspace{-0.1mm}functionals \hspace{-0.1mm}$G$ \hspace{-0.1mm}and \hspace{-0.1mm}$F$, \hspace{-0.1mm}and \hspace{-0.1mm}we \hspace{-0.1mm}refer \hspace{-0.1mm}to \hspace{-0.1mm}the~\hspace{-0.1mm}\mbox{problem}~\hspace{-0.1mm}that~\hspace{-0.1mm}seeks~\hspace{-0.1mm}to~\hspace{-0.1mm}maxi\-mize \hspace{-0.1mm}the \hspace{-0.1mm}(time-dependent) \hspace{-0.1mm}steady \hspace{-0.1mm}dual \hspace{-0.1mm}energy \hspace{-0.1mm}functional \hspace{-0.1mm}\eqref{eq:dual_steady} \hspace{-0.1mm}as \hspace{-0.1mm}the \hspace{-0.1mm}\emph{(time-dependent)~\hspace{-0.1mm}steady~\hspace{-0.1mm}dual~\hspace{-0.1mm}\mbox{problem}}.\enlargethispage{5mm}
    \begin{remark}[on well-posedness of \eqref{eq:primal_steady}]
        If Assumption \ref{ass:energy_densities}(\hyperlink{ass:energy_densities.i}{i}) is satisfied, 
        according to Remark~\ref{rem:integral_functionals_as_normal_integrands},~the spatial integral reductions (\textit{cf}.~Lemma~\ref{lem:integral_functionals_as_normal_integrands}) $I_{\phi}^{\Omega}\colon I\times Y\to \mathbb{R}\cup\{\pm\infty\}$ and $ I_{\psi}^{\Omega}\colon I\times V\to \mathbb{R}\cup\{\pm\infty\}$  are well-defined. Moreover, if, in addition, Assumption \ref{ass:energy_densities}(\hyperlink{ass:energy_densities.ii}{ii}) is satisfied,  according to Corollary \ref{cor:integral_functionals_as_normal_integrands}(\hyperlink{cor:integral_functionals_as_normal_integrands.i}{i}), $G\colon \hspace{-0.1em}I\hspace{-0.1em}\times\hspace{-0.1em} Y\hspace{-0.1em}\to\hspace{-0.1em} \mathbb{R}\cup\{+\infty\}$ and $F\colon \hspace{-0.1em}I\hspace{-0.1em}\times\hspace{-0.1em} V\hspace{-0.1em}\to\hspace{-0.1em} \mathbb{R}\cup\{+\infty\}$ are convex~normal~integrands.~In~\mbox{particular},~for~a.e.~$t\in I$, the functionals $G(t,\cdot)\colon  Y\to \mathbb{R}\cup\{+\infty\}$ and $F(t,\cdot)\colon  V\to \mathbb{R}\cup\{+\infty\}$ and, consequently, by the linearity and boundedness of $L=\nabla\colon V\to Y$, the steady primal energy functional $E(t,\cdot)\colon  V\to \mathbb{R}\cup\{+\infty\}$ are well-defined, proper, convex, and lower semi-continuous.
    \end{remark}

    From a numerical point of view, it is highly desirable that the steady dual energy functional~\eqref{eq:dual_steady} --similarly to the steady primal energy functional~\eqref{eq:primal_steady}-- admits an integral-functional representation.
    Such  representations can be approximated directly by means of numerical integration~and,~unlike~abstract dual-space representations, are localizable in space (time). This spatial~(temporal) localizability is particularly useful in the contexts of \textit{a posteriori} error analysis and~adaptive~\mbox{mesh-refinement}. To this end,
    we next record sufficient conditions~under~which, for sufficiently regular dual
    tensor fields, the steady dual energy functional \eqref{eq:dual_steady} admits an
    integral-functional representation.

    The \hspace{-0.1mm}integral-functional \hspace{-0.1mm}representation \hspace{-0.1mm}of \hspace{-0.1mm}$G^*\colon \hspace{-0.15em}I\hspace{-0.1em}\times\hspace{-0.1em} Y^*\hspace{-0.15em}\to\hspace{-0.15em} \mathbb{R}\cup\{+\infty\}$ \hspace{-0.1mm}follows \hspace{-0.1mm}directly \hspace{-0.1mm}from \hspace{-0.1mm}the~\hspace{-0.1mm}convex~conju\-gation formula for integral functionals (\textit{cf}.\ Lemma \ref{lem:conjugate_of_integral_functional}) and, therefore, requires~no~\mbox{additional}~\mbox{assumption}. More precisely,  if~\mbox{Assumption}~\ref{ass:energy_densities}
is satisfied, by Lemma \ref{lem:conjugate_of_integral_functional}, for a.e.\ $t\in I$ and every $y\in Y^*$,~we~have~that
\begin{align}\label{eq:G_prime_steady}
    G^*(t,y)=I_{\phi^*}^{\Omega}(t,y) 
    \,.
\end{align}

The integral-functional representation of $F^*\colon I\times V^*\to \mathbb{R}\cup\{+\infty\}$ is more delicate, since it
requires a convex conjugation formula that identifies $-L^*y$ with
$\operatorname{div} y$ for sufficiently regular dual tensor fields. This motivates the following additional assumption on the energy density
$\psi\colon Q\times\mathbb{R}^{\ell}\to\mathbb{R}\cup\{+\infty\}$,
under which the (Fenchel) dual problem may be posed on
$Y^*(\operatorname{div})$ rather than on $Y^*$ and the steady dual energy
functional \eqref{eq:dual_steady} likewise admits an integral-functional representation.

    \begin{assumption}[Convex conjugation formula (in space)]\label{ass:convex_conjugation}
    	The energy density  $\psi\colon Q\times \mathbb{R}^{\ell}\to \mathbb{R}\cup\{+\infty\}$ is  such that 
            for a.e.\ $t\in I$ and every $v^*\in V^*$, there holds the \emph{convex conjugation formula (in space)}
    		\begin{align}\label{eq:convex_conjugation_formula}
    			F^*(t,v
            ^*)=
                \begin{cases}
                   I_{\psi^*}^{\Omega}(t,w^*)&\text{if }
                     \iota_p^*w^*=v^*+f(t)\text{ in }V^*
                     \text{ for some }w^*\in L^{p'}(\Omega;\mathbb{R}^{\ell})\,,   
                    \\
                    +\infty&\text{else}\,,
                \end{cases}
    		\end{align} 
            where \hspace{-0.15mm}$\iota_p^*\colon \hspace{-0.175em}L^{p'}(\Omega;\mathbb{R}^{\ell})\hspace{-0.175em}\to\hspace{-0.175em} V^*$ \hspace{-0.15mm}is \hspace{-0.15mm}the \hspace{-0.15mm}adjoint \hspace{-0.15mm}operator \hspace{-0.15mm}to \hspace{-0.15mm}the \hspace{-0.15mm}identity \hspace{-0.15mm}mapping \hspace{-0.15mm}${\iota_p\hspace{-0.175em}\coloneqq \hspace{-0.175em}\operatorname{id}_{V\hspace{-0.1em}\to\hspace{-0.1em} L^{p}(\Omega;\mathbb{R}^{\ell})}\colon \hspace{-0.15em}V\hspace{-0.175em}\to\hspace{-0.175em} L^{p}(\Omega;\mathbb{R}^{\ell})}$.
    \end{assumption}

    The following lemma provides a general sufficient condition on the energy density $\psi\colon Q\times \mathbb{R}^{\ell}\to \mathbb{R}\cup\{+\infty\}$ that ensures the validity of the convex conjugation formula \eqref{eq:convex_conjugation_formula} in Assumption \ref{ass:convex_conjugation}.\vspace{-0.5mm}
    
    \begin{lemma}[Sufficient condition for Assumption \ref{ass:convex_conjugation}]\label{lem:sufficient_for_conjugation}
    Let the energy density $\psi\colon Q\times\mathbb{R}^{\ell}\to \mathbb{R}\cup\{+\infty\}$ be such that for a.e.\ $t\in I$, the spatial integral reduction $ I_{\psi}^{\Omega}(t,\cdot)\colon L^p(\Omega;\mathbb{R}^{\ell})\to \mathbb{R}$ 
        is well-defined, continuous, and bounded (\textit{i.e.}, maps bounded sets in $L^p(\Omega;\mathbb{R}^{\ell})$ into bounded~sets~in~$\mathbb{R}$).
        Then, for a.e.\ $t\in I$, the convex conjugation formula \eqref{eq:convex_conjugation_formula} in Assumption \ref{ass:convex_conjugation} applies.
    \end{lemma}

    \begin{proof}
        See \cite[Lem.\ 3.5]{BartelsKaltenbach2026}.
    \end{proof}

\begin{remark}[on Assumption \ref{ass:convex_conjugation}]
        By the Krasnoselskii theorem (\textit{cf}.\ \cite[Prop.\ 1.1]{Krasnoselskii1964}), Assumption \ref{ass:convex_conjugation}  is 
        satisfied if for a.e.\ $t\in I$, the energy density $\psi(t,\cdot,\cdot)\colon \Omega\times\mathbb{R}^{\ell}\to \mathbb{R}$ is a Carath\'eodory integrand and, for every $v\in L^p(\Omega;\mathbb{R}^{\ell})$, there holds $\psi(t,\cdot,v)\in L^1(\Omega;\mathbb{R}^1)$.
    \end{remark}
    
    If, in addition, Assumption \ref{ass:convex_conjugation} is satisfied and $f=\iota_p^*f_0$ for some $f_0\in L^{p'}(Q;\mathbb{R}^{\ell})$, then,~for~a.e.~$t\in I$ and every $y\in Y^*$, we have that
    \begin{align}\label{eq:integral_representation_F_prime}
        F^*(t,-L^* y)= \begin{cases}
            I_{\psi^*}^{\Omega}(t,\operatorname{div}y+f_0(t))
            &\text{ if }y\in Y^*(\operatorname{div})\,,\\
            +\infty &\text{ else}\,,
        \end{cases}
    \end{align}
    so that it is sufficient to consider the restricted steady dual functional $D(t,\cdot)\colon Y^*(\operatorname{div})\to \mathbb{R}\cup\{-\infty\}$,  for every $y\in Y^*(\operatorname{div})$ given via
    \begin{align}\label{eq:integral_representation_D}
        D(t,y)= -I_{\phi^*}^{\Omega}(t,y)-I_{\psi^*}^{\Omega}(t,\operatorname{div}y+f_0(t))
        \,.
    \end{align}

    \subsection{Fenchel duality framework for subgradient flows induced by convex integral functionals}\label{subsec:primal_problem}

    \hspace*{5mm}In this subsection, we derive a Fenchel duality framework for a broad class of subgradient flows induced by convex integral functionals based on the Br\'ezis--Ekeland--Nayroles~\mbox{principle}~(\textit{cf}.~\mbox{\cite{BrezisEkeland1976I,BrezisEkeland1976II}}~and~\mbox{\cite{Nayroles1976I,Nayroles1976II}}). If Assumption \ref{ass:energy_densities} is satisfied, for a given right-hand side $f\in L^{p'}(I;V^*)$ and~a~given~initial~datum~${u_0\in H}$, the Br\'ezis--Ekeland--Nayroles principle  characterizes~a~solution~$u\in \mathcal{W}(I)$ 
    of the subgradient flow induced by the (time-dependent) family of steady primal energy functionals 
    \eqref{eq:primal_steady},~\textit{i.e.},
    \begin{subequations}\label{eq:subgradient_flow}
        \begin{alignat}{2}
        \label{eq:subgradient_flow.1}
            \partial_tu(t)+\partial_v E(t,u(t))&\ni 0_{V^*}&&\quad \text{ in }V^*\quad\text{ for a.e.\ }t\in I\,,\\
            u(0)&=u_0&&\quad\text{ in }H\,,\label{eq:subgradient_flow.2}
        \end{alignat}
    \end{subequations} 
    as a minimizer of the \emph{Br\'ezis--Ekeland--Nayroles energy functional}
       $\mathcal{E}\colon  \mathcal{W}(I)\to \mathbb{R} \cup \{+\infty\}$, for every $v\in \mathcal{W}(I)$ defined by
\begin{align}\label{eq:primal}
    \mathcal{E}(v)&\coloneqq I_{E}(v)+I_{E^*}(-\partial_t v)+ \tfrac{1}{2} \|v(t_{\mathtt{fin}})\|^2_H + \chi_{\{u_0\}} (v(0))\,,
\end{align} 
to which we refer as the \emph{unsteady primal energy functional} in the context of Fenchel duality theory. Here,
the indicator functional $\smash{\chi_{\{u_0\}}}\colon H\to \mathbb{R}\cup\{+\infty\}$, for every $\widehat{v}\in H$, is defined by
\begin{align*}
    \chi_{\{u_0\}}(\widehat{v})\coloneqq \begin{cases}
        0&\text{ if }\widehat{v}=u_0\text{ in }H\,,\\
        +\infty&\text{ else}\,.
    \end{cases}
\end{align*}
Moreover,  for a.e.\ fixed time $t\in I$, the functional $E^*(t,\cdot) \colon V^*\to \mathbb{R}\cup\{+\infty\}$ denotes the Fenchel conjugate (with respect~to~the~second~argument) of the steady primal energy functional $E(t,\cdot) \colon V\to \mathbb{R}\cup\{+\infty\}$.\linebreak
Rather than imposing further restrictive assumptions on the energy densities $\phi$ and $\psi$~merely~to~ensure~sol\-vability \hspace{-0.1mm}of \hspace{-0.1mm}\eqref{eq:subgradient_flow}, \hspace{-0.1mm}we \hspace{-0.1mm}assume \hspace{-0.1mm}throughout \hspace{-0.1mm}that \hspace{-0.1mm}the \hspace{-0.1mm}subgradient-flow \hspace{-0.1mm}problem \hspace{-0.1mm}\eqref{eq:subgradient_flow}~\hspace{-0.1mm}admits~\hspace{-0.1mm}a~\hspace{-0.1mm}\mbox{solution}~\hspace{-0.1mm}${u\hspace{-0.175em}\in\hspace{-0.175em}\mathcal{W}(I)}$.
By Proposition \ref{prop:brezis_ekeland_nayroles} below, $u$ is a minimizer of \eqref{eq:primal} and we refer~to~it~as~the~\emph{primal~solution}.

Whereas it is more apparent for the integrand $E\colon I\times V\to \mathbb{R}\cup\{+\infty\}$~that~the~\mbox{associated} integral functional $I_E\colon L^p(I;V)\to \mathbb{R}\cup\{+\infty\}$ is well-defined, proper, convex, and lower semi-continuous, this is less obvious for the integral functional  $I_{E^*}\colon L^p(I;V^*)\to \mathbb{R}\cup\{+\infty\}$ associated with the Fenchel~conjugate (with respect to second argument) $E^*\colon I\times V^*\to \mathbb{R}\cup\{+\infty\}$. 

\begin{remark}[on the well-posedness of \eqref{eq:primal}]\label{rem:well-posedness}
    If Assumption \ref{ass:energy_densities} is satisfied, by Corollary~\ref{cor:integral_functionals_as_normal_integrands}(\hyperlink{cor:integral_functionals_as_normal_integrands.i}{i}), the integrands $G\colon I\times Y\to \mathbb{R}\cup\{+\infty\}$ and $F\colon I\times V\to \mathbb{R}\cup\{+\infty\}$ are convex normal~integrands, which, by the continuity of $L=\nabla\colon V\to Y$, implies that  $E\colon I\times V\to \mathbb{R}\cup\{+\infty\}$ is a convex normal integrand.\linebreak As a consequence, due to Lemma \ref{lem:convex_normal_integrands}(\hyperlink{lem:convex_normal_integrands.ii}{ii}), the Fenchel conjugate (with respect to the second argument) $E^*\colon I\times V^*\to \mathbb{R}\cup\{+\infty\}$ is a convex~\mbox{normal}~\mbox{integrand}~as~well.~In~\mbox{particular}, Assumption~\ref{ass:energy_densities}(\hyperlink{ass:energy_densities.ii}{ii}), Corollary~\ref{cor:integral_functionals_as_normal_integrands}(\hyperlink{cor:integral_functionals_as_normal_integrands.ii}{ii}), the continuous embedding $L^p(I;V)\hookrightarrow L^p(Q;\mathbb{R}^\ell)$, and the~continuous~affine~\mbox{perturbation} by $f\in L^{p'}(I;V^*)$ imply that $I_G\in \Gamma_0(L^p(I;Y))$ and $I_F\in\Gamma_0(L^p(I;V))$. Since $u^\star\in\operatorname{dom}(I_G\circ\nabla)\cap\operatorname{dom}(I_F)$, it follows that $I_E=I_G\circ\nabla+I_F\in\Gamma_0(L^p(I;V))$, and Lemma~\ref{lem:conjugate_of_integral_functional} yields $I_{E^*}=(I_E)^*\in\Gamma_0(L^{p'}(I;V^*))$. 
    If, in addition, there exists $u^\sharp\in \operatorname{dom}(I_G\circ \nabla)\cap \operatorname{dom}(I_F)$ such that $I_G\circ \nabla$ or $I_F$ is continuous at $u^\sharp$,~then, for every $v^*\in L^{p'}(I;V^*)$, $I_{E^*}$ is given via the infimal convolution (\textit{cf}.~\cite[Thm.~9.4.1]{AttouchButtazzoMichaille2014}~and~Lemma~\ref{lem:conjugate_of_integral_functional})
\begin{align}\label{eq:E_prime_steady}
   I_{E^*}(v^*)= (I_{E})^*(v^*)=((I_G\circ\nabla)^*\square(I_F)^*)(v^*)=\inf_{y \in L^{p'}(I;Y^*)}{
\bigl\{
I_{G^*}(y)+I_{F^*}(v^*-L^*y) 
\bigr\}}\,.
\end{align}
\end{remark}

For the sake of completeness, 
we recall the Br\'ezis--Ekeland--Nayroles~principle (\textit{cf}.~\cite{BrezisEkeland1976I,BrezisEkeland1976II}~and~\mbox{\cite{Nayroles1976I,Nayroles1976II}}). 

\begin{proposition}
\label{prop:brezis_ekeland_nayroles}
    A function $u\in \mathcal{W}(I)$ solves \eqref{eq:subgradient_flow} if and only~if $\smash{\mathcal{E}(u)=\tfrac{1}{2}\|u_0\|_H^2=\operatorname{min}_{\smash{v\in \mathcal{W}(I)}}{\{\mathcal{E}(v)\}}}$. 
\end{proposition}

\begin{proof}
    Using the integration-by-parts formula in time \eqref{subsubsec:function_spaces_unsteady.3},  for every $v\in \mathcal{W}(I)$, we find that
    \begin{align}\label{prop:brezis_ekeland_nayroles.1}
        \mathcal{E}(v)=\smash{\bigl(I_{E}(v)+\langle \partial_t v,v\rangle_{L^p(I;V)}+I_{E^*}(-\partial_t v)\bigr)}+\tfrac{1}{2}\|u_0\|_H^2+\chi_{\{u_0\}}(v(0))\,.
    \end{align}
    Therefore, for a function $u\in \mathcal{W}(I)$, we have that $\mathcal{E}(u)=\tfrac{1}{2}\|u_0\|_H^2$ if and only~if 
    \begin{subequations} \label{prop:brezis_ekeland_nayroles.2}
    \begin{alignat}{2}
        E(t,u(t))+\langle \partial_t u(t),u(t)\rangle_V+E^*(t,-\partial_t u(t))&=0&&\quad\text{ for a.e.\ }t\in I\,,\label{prop:brezis_ekeland_nayroles.2.1}\\
        u(0)&=u_0&&\quad\text{ in }H\,,\label{prop:brezis_ekeland_nayroles.2.2}
    \end{alignat} 
    \end{subequations}
    where, by the equality condition in the Fenchel--Young inequality (\textit{cf}.\ \cite[Prop.\ 51.2]{Zeidler1985III}), \eqref{prop:brezis_ekeland_nayroles.2.1}~is equivalent to \eqref{eq:subgradient_flow.1}.
\end{proof}\newpage

Under the following additional assumption, a (Fenchel) dual problem (in the sense of \cite[Rem.~4.2, p.~60/61]{EkelandTemam1999}) to the minimization of \eqref{eq:primal} is also given via the maximization~of~a~time~integral~\mbox{functional}.\vspace{-0.5mm}

\begin{assumption}\label{ass:sufficient_for_conjugation}
    The energy density $\psi\colon Q\times\mathbb{R}^{\ell}\to \mathbb{R}\cup\{+\infty\}$ is  such that the integral functional $I_F\colon \hspace{-0.1em}L^p(I;V)\hspace{-0.1em}\to\hspace{-0.1em} \mathbb{R}$ (associated with the convex normal integrand $F\colon \hspace{-0.1em} I\times V\hspace{-0.1em}\to\hspace{-0.1em}\mathbb{R}$)~is~\mbox{continuous}~and~bounded (\textit{i.e.}, maps bounded sets in $L^p(I;V)$ into bounded~sets~in~$\mathbb{R}$).
\end{assumption}

If  Assumptions \ref{ass:energy_densities} and \ref{ass:sufficient_for_conjugation} are satisfied, the (Fenchel) dual problem to the minimization of \eqref{eq:primal} consists in the maximization of the \emph{unsteady dual energy functional}
 $\mathcal{D}\colon L^{p'}(I;Y^*) \times \mathcal{W}(I) \to\mathbb{R} \cup \{- \infty\}$,
 for every $(y,\lambda)\in L^{p'}(I;Y^*) \times \mathcal{W}(I)$ defined by\vspace{-0.5mm}
\begin{align}\label{eq:dual}
\begin{aligned}
    \mathcal{D}(y,\lambda)\coloneqq -I_{G^*}(y)-I_{F^*}(-L^* y-\partial_t\lambda)-I_{E}(\lambda)
           -\tfrac{1}{2}\|\lambda(t_{\mathtt{fin}})\|_H^2+(\lambda(0),u_0)_H\,,
           \end{aligned}\\[-6mm]\notag
\end{align}
which \hspace{-0.1mm}is \hspace{-0.1mm}established \hspace{-0.1mm}along \hspace{-0.1mm}with \hspace{-0.1mm}the \hspace{-0.1mm}existence \hspace{-0.1mm}of \hspace{-0.1mm}a \hspace{-0.1mm}maximizer \hspace{-0.1mm}$(z,\mu)\hspace{-0.15em}\in\hspace{-0.15em} \smash{L^{p'}}(I;Y^*)\hspace{-0.1em}\times\hspace{-0.1em} \mathcal{W}(I)$, \hspace{-0.1mm}called~\hspace{-0.1mm}\emph{dual~\hspace{-0.1mm}solution}, a strong duality relation, and  optimality inclusions in the following theorem.\vspace{-0.5mm}

\begin{theorem}[Identification of dual problem, strong duality, and  optimality inclusions]\label{thm:duality} Let Assumptions \ref{ass:energy_densities} and \ref{ass:sufficient_for_conjugation} be satisfied. Then, the following statements apply:
		\begin{itemize}[noitemsep,topsep=2pt,leftmargin=!,labelwidth=\widthof{(iii)}]
			\item[(i)]\hypertarget{thm:duality.i}{}  A (Fenchel) 
            dual problem to the minimization of \eqref{eq:primal} is given via the \mbox{maximization}~of~\eqref{eq:dual};  
			\item[(ii)]\hypertarget{thm:duality.ii}{}  If  there exists $u^\dagger\in \mathcal{W}(I)$ with $u^\dagger(0)=u_0$ in $H$ such that            
            $((y,\lambda)\mapsto I_{G}(y)+I_{
            E^*}(\lambda))\colon L^{p}(I;Y)\times L^{p'}(I;V^*)\to \mathbb{R}\cup\{+\infty\}$ is continuous at $(\nabla u^\dagger,-\partial_t u^\dagger)$,~then~a~dual~solution $(z,\mu)\in \operatorname{dom}(-\mathcal{D})$, \textit{i.e.}, a maximizer of \eqref{eq:dual}, exists and a \emph{strong duality relation} applies, \textit{i.e.}, we have that\vspace{-0.5mm}
            \begin{align}\label{thm:duality.1}
                \mathcal{E}(u)=\tfrac{1}{2}\|u_0\|_H^2=\mathcal{D}(z,\mu)\,,\\[-6mm]\notag
            \end{align}
            which is equivalent to the \emph{optimality inclusions}\vspace{-0.5mm}
            \begin{subequations}\label{thm:duality.2} 
            \begin{alignat}{3}\label{thm:duality.2.1} 
            z(t)&\in \partial_y G(t,\nabla u(t))&&\quad\text{ in }Y^*&&\quad\text{ for a.e.\ }t\in I\,,\\\label{thm:duality.2.2} 
            -\partial_t u(t)&\in  \partial_v E(t,\mu(t))&&\quad\text{ in }V^*&&\quad\text{ for a.e.\ }t\in I\,,\\\label{thm:duality.2.3} 
            -\partial_t\mu(t)+\operatorname{div}z(t)&\in \partial_vF(t,u(t))&&\quad\text{ in }V^*&&\quad\text{ for a.e.\ }t\in I\,,\\
            u(t_{\mathtt{fin}})&=\mu(t_{\mathtt{fin}})&&\quad\text{ in }H\,.\label{thm:duality.2.4} \\[-6mm]\notag
            \end{alignat}
            \end{subequations}
		\end{itemize}
\end{theorem}

\begin{remark}[on Theorem \ref{thm:duality}]\label{rem:duality} 

The second component $\mu$ of a dual solution $(z,\mu)$ is not an
additional physical state variable. Instead, it is the Fenchel variable
conjugate to the second component of the parabolic differential operator
$\mathcal{L}\coloneqq
(v\mapsto(\nabla v,-\partial_t v))
\in\mathscr{L}(\mathcal{W}(I);
L^p(I;Y)\times L^{p'}(I;V^*))$
(\textit{cf}.~\eqref{thm:duality.4} below).
Thus, the terminal condition~\eqref{thm:duality.2.4} is the natural
boundary condition arising from integration by parts in time on
$\mathcal{W}(I)$. In this sense, $\mu$ has an adjoint-like character,
although it is not an adjoint state arising from a separate
optimal-control objective. The optimality inclusion~\eqref{thm:duality.2.2},~together~with~\eqref{eq:subgradient_flow.1}, shows \hspace{-0.1mm}that \hspace{-0.1mm}$u(t),\mu(t)\hspace{-0.15em}\in\hspace{-0.15em}\partial_{v^*} E^*(t,-\partial_tu(t))$
\hspace{-0.1mm}for \hspace{-0.1mm}a.e.\ \hspace{-0.1mm}$t\hspace{-0.15em}\in\hspace{-0.15em} I$.
\hspace{-0.1mm}In \hspace{-0.1mm}particular,~\hspace{-0.1mm}$\mu\hspace{-0.15em}=\hspace{-0.15em}u$~\hspace{-0.1mm}\mbox{whenever}~\hspace{-0.1mm}this~\hspace{-0.1mm}set~\hspace{-0.1mm}is~\hspace{-0.1mm}a~\hspace{-0.1mm}\mbox{singleton}.

This interpretation is related to the anti-self-duality of the Br\'ezis--Ekeland--Nayroles~\mbox{Lagrangian} $\mathfrak{L}\colon L^p(I;V)\times L^{p'}(I;V^*)\to \mathbb{R}\cup\{+\infty\}$, defined by $\mathfrak{L}(v,v^*)\coloneqq I_E(v)+I_{E^*}(-v^*)$ for all ${(v,v^*)\in L^p(I;V)}\times L^{p'}(I;V^*)$. \hspace{-0.1mm}More \hspace{-0.1mm}precisely, \hspace{-0.1mm}from \hspace{-0.1mm}Lemma \hspace{-0.1mm}\ref{lem:conjugate_of_integral_functional}, \hspace{-0.1mm}it \hspace{-0.1mm}follows \hspace{-0.1mm}that \hspace{-0.1mm}${\mathfrak{L}^*(w^*,w)\hspace{-0.15em}=\hspace{-0.15em}I_{E^*}(w^*)\hspace{-0.15em}+\hspace{-0.15em}I_{E}(-w)\hspace{-0.15em}=\hspace{-0.15em}\mathfrak{L}(-w,-w^*)}$ for all $(w^*,w)\in L^{p'}(I;V^*)\times L^p(I;V)$, 
\textit{i.e.}, $\mathfrak{L}$ is anti-self-dual on the path space. The dual formulation
of Theorem~\ref{thm:duality} can be viewed as the Fenchel-dual realization of
this structure after the substitution~$v^*=\partial_t v$, the splitting
$I_{E}=I_{G}\circ\nabla +I_{F}$, and the inclusion of the temporal boundary
terms.
\end{remark}

To prove Theorem \ref{thm:duality}, we first establish the following auxiliary density result.\enlargethispage{5mm}\vspace{-0.5mm}

\begin{lemma}\label{lem:dense_range}
    The mapping $\Pi\coloneqq(v\mapsto (v,v(t_{\mathtt{fin}})))\colon \mathcal{W}_0(I)\coloneqq \{v\in \mathcal{W}(I)\mid v(0)=0\}\to L^p(I;V)\times H$ has a dense range.\vspace{-0.5mm}
\end{lemma}

\begin{proof}
    Let $(v,h)\in L^p(I;V)\times H$ be fixed, but arbitrary. Since $C_{\mathrm{c}}^1(I;V)$ is dense in $L^p(I;V)$ (\textit{cf}.\ \cite[Prop.\ 23.2(c)]{Zeidler1990IIA}), there exists a sequence $\{\varphi_n\}_{n\in\mathbb{N}}\subseteq C_{\mathrm{c}}^1(I;V)\subseteq\mathcal{W}_0(I)$ such that $\varphi_n\to v$ in $L^p(I;V)$ $(n\to\infty)$. Moreover, since $V$ is dense in $H$, there exists a sequence $\{v_n\}_{n\in\mathbb{N}}\subseteq V$ such that $v_n\to h$ in $H$ $(n\to\infty)$.\linebreak
    Then, for every $n\in\mathbb{N}$, choose $m_n\in\mathbb{N}$ sufficiently large such that
    $\frac{t_{\mathtt{fin}}}{m_n}\|v_n\|_V^p\leq\frac{1}{n^p}$,
    and define $\psi_n\in W^{1,\infty}(I)$ by
    $\psi_n(t)\coloneqq
    \chi_{\smash{[t_{\mathtt{fin}}(\frac{m_n-1}{m_n}),t_{\mathtt{fin}}]}}(t)
    (\frac{m_n}{t_{\mathtt{fin}}}t-(m_n-1))$
    for all $t\in I$. Then, $\psi_n(0)=0$, $\psi_n(t_{\mathtt{fin}})=1$, and
    $\|\psi_n v_n\|_{L^p(I;V)}^p
   \hspace{-0.1em} \leq\hspace{-0.1em}
    \frac{t_{\mathtt{fin}}}{m_n}\|v_n\|_V^p
    \hspace{-0.1em}\leq\hspace{-0.1em}\frac{1}{n^p}$ for all $n\hspace{-0.1em}\in\hspace{-0.1em} \mathbb{N}$.
    Then, the sequence
    ${\{\widetilde{v}_n\}_{n\in\mathbb{N}}
  \hspace{-0.1em}  \coloneqq\hspace{-0.1em}
    \{\varphi_n+\psi_n v_n\}_{n\in\mathbb{N}}
   \hspace{-0.1em} \subseteq\hspace{-0.1em}\mathcal{W}_0(I)}$
    satisfies $\widetilde{v}_n\to v$ in $L^p(I;V)$ $(n\to\infty)$ and
    $\widetilde{v}_n(t_{\mathtt{fin}})=v_n\to h$ in $H$ $(n\to\infty)$. 
\end{proof}
\newpage
\begin{proof}[Proof (of Theorem \ref{thm:duality}).]
    \emph{ad \hspace{-0.1mm}(\hyperlink{thm:duality.i}{i}).}
   \hspace{-0.1mm}To \hspace{-0.1mm}begin \hspace{-0.1mm}with, \hspace{-0.1mm}we \hspace{-0.1mm}introduce \hspace{-0.1mm}the \hspace{-0.1mm}functionals \hspace{-0.1mm}${\mathcal{G}\hspace{-0.175em}\in\hspace{-0.175em}  \Gamma_0(L^p(I;Y)\hspace{-0.175em}\times\hspace{-0.175em} L^{p'}(I;V^*))}$ and $\mathcal{F}\in  \Gamma_0(\mathcal{W}(I))$, for every $(y,\lambda)\in L^p(I;Y)\times L^{p'}(I;V^*)$ and $v\in \mathcal{W}(I)$, respectively, defined by\enlargethispage{2mm}
    \begin{subequations}\label{thm:duality.3}
    \begin{align}\label{thm:duality.3.1}
			\mathcal{G}(y,\lambda)&\coloneqq I_{G}(y)+I_{
            E^*}(\lambda)\,,\\
			\mathcal{F}(v)&\coloneqq I_{F}(v)+\tfrac{1}{2}\|v(t_{\mathtt{fin}})\|_H^2+\chi_{\{u_0\}}(v(0))\,,\label{thm:duality.3.2}
	\end{align}
    \end{subequations}
    and the parabolic differential operator $\mathcal{L}\hspace{-0.1em}\in\hspace{-0.1em} \mathscr{L}(\mathcal{W}(I); L^p(I;Y)\times L^{p'}(I;V^*))$, for every $v\hspace{-0.1em}\in \hspace{-0.1em}\mathcal{W}(I)$~\mbox{defined}~by
    \begin{align}\label{thm:duality.4}
        \mathcal{L}v\coloneqq (\nabla v,-\partial_tv)\quad\text{ in }L^p(I;Y)\times L^{p'}(I;V^*)\,,
    \end{align}
    so that, for every  $v\in \mathcal{W}(I)$, we have that
    \begin{align*}
		\mathcal{E}(v)= \mathcal{G}(\mathcal{L}v)+\mathcal{F}(v)\,.
	\end{align*}
		Then, according to \cite[Rem.\ 4.2, p.\ 60/61]{EkelandTemam1999}, the (Fenchel) dual problem to the minimization of \eqref{eq:primal} is given via the maximization of the dual energy functional $\mathcal{D}\colon  L^{p'}(I;Y^*)\times L^{p}(I;V)\to \mathbb{R}\cup \{-\infty\}$,\linebreak for every $(y,\lambda)\in L^{p'}(I;Y^*)\times L^{p}(I;V)$ defined by 
		\begin{align}\label{prop:duality.6}
			\mathcal{D}(y,\lambda)\coloneqq -\mathcal{G}^*(y,\lambda)- \mathcal{F}^*(-\mathcal{L}^*(y,\lambda))\,,
		\end{align}
        where $\mathcal{L}^*\in \mathscr{L}(\smash{L^{p'}}(I;Y^*)\times L^{p}(I;V);(\mathcal{W}(I))^*)$ is the adjoint operator of 
        \eqref{thm:duality.4}.\enlargethispage{2.5mm}
        
        Therefore, it is only left to establish that the claimed representation \eqref{eq:dual} applies:
        \begin{itemize}[noitemsep,topsep=2pt,leftmargin=!,labelwidth=\widthof{$\bullet$}]
            \item[$\bullet$] Since, according to the reasoning in Remark \ref{rem:well-posedness}, the spatial integral reduction $G\colon I\times Y\to \mathbb{R}\cup\{+\infty\}$ and $E\colon I\times V\to \mathbb{R}\cup\{+\infty\}$ are convex normal integrands and, due to  Assumption \ref{ass:energy_densities}(\hyperlink{ass:energy_densities.ii}{ii}),~we~have~that $\operatorname{dom}(I_G)\cap L^p(I;Y)\neq \emptyset$ and $\operatorname{dom}(I_E)\cap L^p(I;V)\neq \emptyset$,  the convex conjugation formula for integral functionals (\textit{cf}.\ Lemma \ref{lem:conjugate_of_integral_functional}) is applicable 
            to  $I_G\colon L^p(I;Y)\to \mathbb{R}\cup\{+\infty\}$ and $I_E\colon L^p(I;V)\to \mathbb{R}\cup\{+\infty\}$, and,  for every $y\in L^{p'}(I;Y^*)$ and $\lambda \in L^{p}(I;V)$, due to $E^{**}=E$ (\textit{cf}.\ \cite[Prop.\ 4.1, p.\ 18]{EkelandTemam1999}), yields
            \begin{align}\label{prop:duality.7}
            \mathcal{G}^*(y,\lambda)=I_{G^*}(y)+I_{E}(\lambda)\,.
            \end{align}

            \item[$\bullet$] Let $y\in \smash{L^{p'}}(I;Y^*)$ and $\lambda \in L^{p}(I;V)$ be fixed, but arbitrary. Then, denoting by $\widehat{u}_0\in \mathcal{W}(I)$ a trace lift of the initial datum $u_0\in H$, \textit{i.e.}, there holds $\widehat{u}_0(0)=u_0$ in $H$,  
            we find that
            \begin{align} \label{prop:duality.8}
                \mathcal{F}^*(-\mathcal{L}^*(y,\lambda))&=\sup_{v\in \mathcal{W}(I)}\bigl\{-\langle y, \nabla v\rangle_{L^p(I;Y)}+\langle \partial_tv,\lambda\rangle_{L^p(I;V)}-I_{F}(v)-\tfrac{1}{2}\|v(t_{\mathtt{fin}})\|_H^2-\chi_{\{u_0\}}(v(0))\bigr\}\notag
                \\[-0.5mm]&=\sup_{v\in \mathcal{W}_0(I)}\bigl\{-\langle y, \nabla (v+\widehat{u}_0)\rangle_{L^p(I;Y)}+\langle \partial_t(v+\widehat{u}_0)(t),\lambda(t)\rangle_{L^p(I;V)}\\[-2.5mm]&\qquad\qquad\quad-I_{F}(v+\widehat{u}_0)-\tfrac{1}{2}\|(v+\widehat{u}_0)(t_{\mathtt{fin}})\|_H^2\bigr\}
           \eqqcolon \texttt{SUP}\,,\notag
            \end{align}
            Apparently, we have that\vspace{-0.5mm} 
            \begin{align}\label{prop:duality.9}
                \begin{aligned}
                    &\sup_{\varphi\in C_{\mathrm{c}}^1(I;V)}\bigl\{-\langle y, \nabla \varphi\rangle_{L^p(I;Y)}+(\lambda,\partial_t\varphi)_{Q}-I_{F}(\varphi+\widehat{u}_0)\bigr\}\\[-0.5mm]&\qquad\leq \texttt{SUP}+\tfrac{1}{2}\|\widehat{u}_0(t_{\mathtt{fin}})\|_H^2+
                    \langle y, \nabla\widehat{u}_0\rangle_{L^p(I;Y)} -\langle\partial_t \widehat{u}_0,\lambda\rangle_{L^p(I;V)}\,.
                     \end{aligned}
            \end{align}
            Next, if $\texttt{SUP}<+\infty$, 
             we distinguish the cases $\lambda\in \mathcal{W}(I)$ and $\lambda\notin  \mathcal{W}(I)$:

        \quad$\bullet$ \emph{Case $\lambda\notin  \mathcal{W}(I)$.} In this case, there exists a  sequence $\{\varphi_n\}_{n\in \mathbb{N}}\subseteq C_{\mathrm{c}}^1(I;V)$ bounded in $L^p(I;V)$ such that\vspace{-0.5mm} 
        \begin{align}\label{prop:duality.10}
            (\lambda,\partial_t\varphi_n)_{Q}\to +\infty\qquad (n\to \infty)\,.
        \end{align}
        Then, since, by Assumption \ref{ass:sufficient_for_conjugation},  the  functional $I_{F}\colon L^p(I;V)\to \mathbb{R}$ is bounded, from the boundedness of the sequence $\{\varphi_n\}_{n\in \mathbb{N}}\subseteq C_{\mathrm{c}}^1(I;V)$ bounded in $L^p(I;V)$ and \eqref{prop:duality.10}, it follows that  
        \begin{align*}
        \begin{aligned} 
            &\sup_{\varphi\in C_{\mathrm{c}}^1(I;V)}\bigl\{-(y, \nabla \varphi)_{Q}+(\lambda,\partial_t\varphi)_{Q}-I_{F}(\varphi+\widehat{u}_0)\bigr\}
            \\&\quad\ge -(y, \nabla \varphi_n)_{Q}+(\lambda,\partial_t\varphi_n)_{Q}-I_{F}(\varphi_n+\widehat{u}_0)
            \\&\quad\ge (\lambda,\partial_t\varphi_n)_{Q}-\sup_{n\in \mathbb{N}}{\bigl\{
            \|y\|_{L^{p'}(I;Y^*)}\|\varphi_n\|_{L^p(I;V)} +I_{F}(\varphi_n+\widehat{u}_0)\bigr\}}
            \to +\infty\quad (n\to \infty)\,,
            \end{aligned}
        \end{align*}
        \textit{i.e.}, due to \eqref{prop:duality.9}, a contradiction to the assumption $\texttt{SUP}<+\infty$.

        \quad$\bullet$ \emph{Case $\lambda\in  \mathcal{W}(I)$.} In this case, using the integration-by-parts formula in time \eqref{subsubsec:function_spaces_unsteady.3}, from \eqref{prop:duality.8}, it follows that
        \begin{align}\label{prop:duality.12}
        \hspace{-2.5mm}\begin{aligned} 
            \mathcal{F}^*(-\mathcal{L}^*(y,\lambda))&=\sup_{v\in \mathcal{W}_0(I)}\bigl\{-\langle y, \nabla (v+\widehat{u}_0)\rangle_{L^p(I;Y)}+\langle \partial_t(v+\widehat{u}_0),\lambda\rangle_{L^p(I;V)}\\[-2mm]&\qquad\qquad\quad-I_{F}(v+\widehat{u}_0)-\tfrac{1}{2}\|(v+\widehat{u}_0)(t_{\mathtt{fin}})\|_H^2\bigr\}
            \\&=\sup_{v\in \mathcal{W}_0(I)}\bigl\{\langle-L^* y-\partial_t\lambda,v+\widehat{u}_0\rangle_{L^p(I;V)}-I_{F}(v+\widehat{u}_0)\\[-2mm]&\qquad\qquad\quad+((v+\widehat{u}_0)(t_{\mathtt{fin}}),\lambda(t_{\mathtt{fin}}))_H-\tfrac{1}{2}\|(v+\widehat{u}_0)(t_{\mathtt{fin}})\|_H^2\bigr\}-(u_0,\lambda(0))_H\,.
            \end{aligned}\hspace{-7.5mm}
        \end{align}
        Since  $R(\Pi)$ is dense in $L^p(I;V)\times H$ (\textit{cf}.\ Lemma \ref{lem:dense_range}) and both $I_{F}\colon L^p(I;V)\to \mathbb{R}$ (\textit{cf}.\ Assumption~\ref{ass:sufficient_for_conjugation}) and $\tfrac{1}{2}\|\cdot\|_H^2\colon H\to \mathbb{R}$ are continuous, from \eqref{prop:duality.12}, we infer that 
        \begin{align}\label{prop:duality.13}
            \begin{aligned}
            \mathcal{F}^*(-\mathcal{L}^*(y,\lambda))&=\sup_{v\in L^p(I;V)}\bigl\{\langle-\partial_t\lambda-L^*y,v+\widehat{u}_0\rangle_{L^p(I;V)}-I_{F}(v+\widehat{u}_0)\bigr\}
            \\&\quad\,+\sup_{h \in H}\bigl\{(h+\widehat{u}_0(t_{\mathtt{fin}}),\lambda(t_{\mathtt{fin}}))_H-\tfrac{1}{2}\|h+\widehat{u}_0(t_{\mathtt{fin}})\|_H^2\bigr\}
            -(\lambda(0),u_0)_H
            \\&=I_{F^*}(-L^* y-\partial_t\lambda)+\tfrac{1}{2}\|\lambda(t_{\mathtt{fin}})\|_H^2-(\lambda(0),u_0)_H\,.
            \end{aligned}\hspace{-2.5mm}
        \end{align}
        \end{itemize}

        Putting everything together, \textit{i.e.}, the cases $\mathtt{SUP}=+\infty$ and $\mathtt{SUP}<+\infty$ (including the subcases $\lambda\notin \mathcal{W}(I)$ and $\lambda\in \mathcal{W}(I)$; using \eqref{prop:duality.7} and \eqref{prop:duality.13} in \eqref{prop:duality.6} in the latter case), we find that
        \begin{align*}
            \mathcal{D}(y,\lambda)&=
            \begin{cases}
            -I_{G^*}(y)-I_{F^*}(-L^* y-\partial_t\lambda)-I_{E}(\lambda)
           -\tfrac{1}{2}\|\lambda(t_{\mathtt{fin}})\|_H^2+(\lambda(0),u_0)_H&\quad \text{ if }\lambda\in \mathcal{W}(I)\,,\\
                -\infty&\quad\text{ else}\,.
            \end{cases}
        \end{align*}
        In particular, since $\operatorname{dom}(-\mathcal{D})\subseteq L^{p'}(I;Y^*)\times \mathcal{W}(I)$, it is sufficient to consider the restricted unsteady dual energy functional $\mathcal{D}\colon L^{p'}(I;Y^*)\times \mathcal{W}(I)\to \mathbb{R}\cup\{-\infty\}$, which, eventually, confirms the claimed representation \eqref{eq:dual}.

        \emph{ad \hspace{-0.1mm}(\hyperlink{thm:duality.ii}{ii}).} \hspace{-0.1mm}By \hspace{-0.1mm}Assumption \hspace{-0.1mm}\ref{ass:sufficient_for_conjugation}, \hspace{-0.1mm}there \hspace{-0.1mm}holds \hspace{-0.1mm}$\mathcal{F}(u^{\dagger})\hspace{-0.175em}<\hspace{-0.175em}+\infty$ \hspace{-0.1mm}and, \hspace{-0.1mm}by \hspace{-0.1mm}the \hspace{-0.1mm}additional \hspace{-0.1mm}assumption~\hspace{-0.1mm}of~\hspace{-0.1mm}this~\hspace{-0.1mm}\mbox{theorem}, $\mathcal{G}\colon \hspace{-0.15em}L^p(I;Y)\times L^{p'}(I;V^*)\hspace{-0.15em}\to\hspace{-0.15em} \mathbb{R}\cup\{+\infty\}$ is continuous at $\mathcal{L}u^{\dagger}=(\nabla u^{\dagger},-\partial_t u^{\dagger})$ (in particular, $\mathcal{G}(\mathcal{L}u^{\dagger})\hspace{-0.15em}<\hspace{-0.15em}+\infty$). Therefore,~the~Fenchel~\mbox{duality}~theorem (\textit{cf}.\ \cite[Rem.~4.1,~eqs.~(4.21), p.\ 61]{EkelandTemam1999}) yields the existence of a dual solution $(z,\mu)\in \textup{dom}(-\mathcal{D})\subseteq L^{p'}(I;Y^*)\times \mathcal{W}(I)$ and that a strong duality relation~\eqref{thm:duality.1}~applies. Moreover, 
        according to \cite[Rem.\ 4.1, eqs.\ (4.22),(4.23), p.\ 61]{EkelandTemam1999}, the strong duality relation~is~\mbox{equivalent}~to\vspace{-4.5mm}
        \begin{subequations}\label{prop:duality.14}
        \begin{align}\label{prop:duality.14.1}
            \mathcal{G}(\nabla u,-\partial_t u)-\langle (z,\mu), \mathcal{L}u\rangle_{L^p(I;Y)\times L^{p'}(I;V^*)}+\mathcal{G}^*(z,\mu)&=0\,,\\
            \mathcal{F}(u)-\langle -\mathcal{L}^*(z,\mu), u\rangle_{\mathcal{W}(I)}+\mathcal{F}^*(-\mathcal{L}^*(z,\mu))&=0\,,\label{prop:duality.14.2}
        \end{align}
        \end{subequations}
        where, by the definitions of the integral functionals in  \eqref{thm:duality.3}, the equations in \eqref{prop:duality.14} each equivalently can be expressed as
        \begin{subequations}\label{prop:duality.15}
        \begin{align}\label{prop:duality.15.1}
         &
         \begin{aligned} 
         0&=\bigl(I_{G^*}(z)-\langle z,\nabla u\rangle_{L^p(I;Y)}+I_{G}(\nabla u)\bigr)+
         \bigl(I_{E^*}(-\partial_t u)-\langle -\partial_t u,\mu\rangle_{L^p(I;V)}+I_{E}(\mu)\bigr)\,, 
            \end{aligned}\\
        &\begin{aligned} 0&=\bigl(I_{F^*}(-\nabla^* z-\partial_t\mu)-\langle-\nabla^* z-\partial_t \mu,u\rangle_{L^p(I;V)}+I_{F}(u)\bigr)+
        \tfrac{1}{2}\|(u-\mu)(t_{\mathtt{fin}})\|_H^2\,.
        \end{aligned}\label{prop:duality.15.2}
        \end{align}  
        \end{subequations}
        Since, by the Fenchel--Young inequality (\textit{cf}.\ \cite[Prop.\ 51.2]{Zeidler1985III}), the four integrands in \eqref{prop:duality.15} are point-wise non-negative (a.e.), we infer that 
        \begin{alignat*}{2}
            G^*(t,z(t))-\langle z(t),\nabla u(t)\rangle_{Y}+G(t,\nabla u(t))&=0&&\quad\text{ for a.e.\ }t\in I\,,\\
            E^*(t,-\partial_tu(t))-\langle -\partial_t u(t),\mu(t)\rangle_{V}+E(t,\mu (t))&=0&&\quad\text{ for a.e.\ }t\in I\,,\\
            F^*(t,\operatorname{div}z(t)-\partial_t \mu (t))-\langle \operatorname{div}z(t)-\partial_t \mu(t),u(t)\rangle_V+F(t,u(t))&=0&&\quad\text{ for a.e.\ }t\in I\,,\\
            u(t_{\mathtt{fin}})&=\mu(t_{\mathtt{fin}})&&\quad\text{ in }H\,,
        \end{alignat*}
        which,   by the equality condition in the Fenchel--Young inequality (\textit{cf}.\ \cite[Prop.\ 51.2]{Zeidler1985III}), in turn, is equivalent to the optimality inclusions \eqref{thm:duality.2}.
\end{proof}

\section{Duality-based \textit{a posteriori} error control}\label{sec:duality_based_a_posteriori_error_control}

\hspace{5mm}In this section, we introduce an \textit{a posteriori} error control
framework that is based on the Br\'ezis--Ekeland--Nayroles principle and the corresponding
Fenchel duality framework (\textit{cf}.\ Subsection~\ref{subsec:primal_problem}).
The approach is inspired by the contributions
\cite{BartelsGudiKaltenbach2025,AntilBartelsKaltenbachKhandelwal2025,
AntilKaltenbachKirk2026}, where so-called
\textit{primal-dual gap \emph{a posteriori} identities} for
classes of steady convex minimization problems were derived based on a Fenchel duality framework.

Following these contributions, the error measure on the left-hand side of such an
\emph{a posteriori} error identity is the sum of the \emph{primal optimal strong convexity measure}
$\rho_{\mathcal{E}}^2\colon\mathcal{W}(I)\to[0,+\infty]$ at a primal
solution $u\in\mathcal{W}(I)$, for every $v\in\mathcal{W}(I)$ defined by 
\begin{align}\label{def:primal_optimal_strong_convexity_measure}
    \smash{\rho_{\mathcal{E}}^2(v)\coloneqq \mathcal{E}(v)-\mathcal{E}(u)\,,}
\end{align}
and the \emph{dual optimal strong convexity measure} $\rho_{-\mathcal{D}}^2\colon \smash{L^{p'}}(I;Y^*)\times \mathcal{W}(I)\to [0,+\infty]$ at a dual solution $(z,\mu)\in \smash{L^{p'}}(I;Y^*)\times \mathcal{W}(I)$, for every $(y,\lambda)\in \smash{L^{p'}}(I;Y^*)\times \mathcal{W}(I)$ defined by
\begin{align}\label{def:dual_optimal_strong_convexity_measure}
    \smash{\rho_{-\mathcal{D}}^2(y,\lambda)\coloneqq -\mathcal{D}(y,\lambda)+\mathcal{D}(z,\mu)\,,}
\end{align}
respectively. The corresponding \textit{a posteriori} error estimator on the right-hand side is the 
\emph{primal-dual~gap estimator}
$\eta_{\mathcal{E}-\mathcal{D}}^2\colon\hspace{-0.1em}
\mathcal{W}(I)\times (L^{p'}(I;Y^*)\times \mathcal{W}(I))\hspace{-0.1em}\to\hspace{-0.1em}
[0,+\infty]$, for every $(v,(y,\lambda))\hspace{-0.1em}\in\hspace{-0.1em}\mathcal{W}(I)\times (\smash{L^{p'}}(I;Y^*)\times \mathcal{W}(I))$ defined by
\begin{align*}
    \smash{\eta_{\mathcal{E}-\mathcal{D}}^2(v,(y,\lambda))
    \coloneqq \mathcal{E}(v)-\mathcal{D}(y,\lambda)\,.}
\end{align*} 
Then, by virtue of the strong duality relation~\eqref{thm:duality.1}, for every $(v,(y,\lambda))\in\mathcal{W}(I)\times (\smash{L^{p'}}(I;Y^*)\times \mathcal{W}(I))$, there holds the
\emph{primal-dual gap identity}\vspace{-0.5mm}
\begin{align*}
    \rho_{\mathcal{E}}^2(v)
    + \rho_{-\mathcal{D}}^2(y,\lambda)
    =
    \eta_{\mathcal{E}-\mathcal{D}}^2(v,(y,\lambda))\,.
\end{align*} 

However, due to the anti-self-duality of the Br\'ezis--Ekeland--Nayroles principle (\textit{cf}.~\mbox{Remark}~\ref{rem:duality})\linebreak and the fact that the
optimal primal and dual values are independent of particular primal and
dual solutions (\textit{cf}.\ \eqref{thm:duality.1}), the above approach can be refined further. In what follows,
we derive separate primal and dual gap identities. The primal gap identities
are closer to weak residual-type identities: admissible approximations are
comparatively easy to generate, whereas the evaluation of the right-hand side may be more
involved. The dual gap identities, by contrast, can be viewed as
direct unsteady counterparts of steady primal-dual gap identities: the evaluation of the
right-hand side is typically~easier, but the construction of
admissible dual approximations may be more involved.

\begin{remark}[on optimal strong convexity measures]
\label{rem:optimal_strong_convexity_measure}
Let $(X,\|\cdot\|_X)$ be a real Banach space and let
$f\colon X\to \mathbb{R}\cup\{+\infty\}$ be proper, convex, and lower
semi-continuous. Following the terminology~in~\cite{BartelsGudiKaltenbach2025,AntilBartelsKaltenbachKhandelwal2025,AntilKaltenbachKirk2026},
if $\bar{x}\hspace{-0.1em}\in\hspace{-0.1em}\operatorname{argmin}_{y\in X}{\{f(y)\}}$, 
we introduce the \emph{optimal strong
convexity measure (of $f$ at $\bar{x}$)} $\rho_f^2\colon \hspace{-0.1em}X\hspace{-0.1em}\to\hspace{-0.1em} [0,+\infty]$, for every $y\in X$ defined by\vspace{-0.5mm}
\begin{align}\label{def:optimal_strong_convexity_measure}
    \smash{\rho_f^2(y) \coloneqq f(y)-f(\bar{x})\ge 0\, .}
\end{align}
The minimizer $\bar{x}\in X$ is not included in the notation $\rho_f^2$, since the value
$f(\bar{x})=\min_{y\in X}{\{f(y)\}}$ is independent of the particular choice of
$\bar{x}\in\operatorname*{argmin}_{y\in X}{\{f(y)\}}$. 

\begin{itemize}[noitemsep,topsep=2pt,leftmargin=!,labelwidth=\widthof{(ii)}]
\item[(i)]
Definition \eqref{def:optimal_strong_convexity_measure} can be interpreted as a special case of the generalized
Bregman~divergence~(see~\mbox{\cite{Bregman1967}}\linebreak and,
for the underlying subdifferential calculus,
\cite{EkelandTemam1999,RockafellarWets1998}). Indeed,~for~each~${x\hspace{-0.15em}\in\hspace{-0.15em}\operatorname{dom}(f)}$~and~${x^*\hspace{-0.15em}\in\hspace{-0.15em}\partial f(x)}$, the generalized
Bregman divergence (of $f$ at $x$ with respect to $x^*$) $\mathcal{D}_f^{x^*}(\cdot,x)\colon X\to [0,+\infty]$, for every $y\in X$, is defined by
\begin{align}\label{def:bregman_divergence}
    \smash{\mathcal{D}_f^{x^*}(y,x)
    \coloneqq
    f(y)-f(x)-\langle x^*,y-x\rangle_X\ge 0\,.}
\end{align} 
If, in particular, $x\in\operatorname{argmin}_{y\in X}{\{f(y)\}}$, then
$0^*\in\partial f(x)$ and, thus, for every $y\in X$, we have that
\begin{align*}
    \smash{\rho_f^2(y)
    =
    \mathcal{D}_f^{0}(y,x)\,.}
\end{align*}
In other words, the optimal strong convexity measure is precisely the generalized
Bregman divergence with the choice $x^*=0$;\newpage 
\item[(ii)]
If $f$ is G\^ateaux differentiable at $x\in X$, then the Bregman divergence
with respect to the G\^ateaux derivative $\mathrm{D}f(x)$ at $x$, for every $y\in X$, reads
\begin{align*}
    \smash{\mathcal{D}_f^{\mathrm{D}f(x)}(y,x)
    =
    f(y)-f(x)-\langle \mathrm{D}f(x),y-x\rangle_X}\, .
\end{align*}
If, in addition, $f$ is twice G\^ateaux differentiable along the segment
$[x,y]\coloneqq \{sy +(1-s)x\mid s\in [0,1]\}$ for some $y\in X$ and Taylor's formula is applicable
along this segment, then
\begin{align*}
  \mathcal{D}_f^{\mathrm{D}f(x)}(y,x)
    =
    \int_0^1{(1-s)
        \langle
            \mathrm{D}^2 f(sy +(1-s)x)(y-x),
            y-x
        \rangle_X
    \,\mathrm{d}s}\,;
\end{align*} 

\end{itemize}
\end{remark}

The primal and dual optimal strong convexity measures defined in
\eqref{def:primal_optimal_strong_convexity_measure} and
\eqref{def:dual_optimal_strong_convexity_measure} are global~quantities:\linebreak
they are expressed in terms of the optimal primal and dual values and do not
yet reflect the~Fenchel~struct\-ure of the Br\'ezis--Ekeland--Nayroles energy
functional \eqref{eq:primal}. For the later derivation~of~\textit{a~posteriori}~identities,
however, it is useful to rewrite these energy differences in terms of
generalized Bregman divergences. The required subgradients are provided by the
optimality inclusions \eqref{thm:duality.2}.

To begin with, we obtain such a representation for the primal optimal strong convexity
measure~\eqref{def:primal_optimal_strong_convexity_measure}. It separates the contribution of the integral functional $I_E\colon L^p(I;V)\to\mathbb{R}\cup\{+\infty\}$,
the contribution of the conjugate integral functional $I_{E^*}\colon L^{p'}(I;V^*)\to\mathbb{R}\cup\{+\infty\}$, and the
terminal contribution induced by the Br\'ezis--Ekeland--Nayroles energy  functional \eqref{eq:primal}.

\begin{lemma}[Bregman-type representation of the primal optimal strong convexity measure]
\label{lem:bregman_representation_primal_error_measure}
Let Assumption \ref{ass:energy_densities} be
satisfied. Then, for every
$v\in\operatorname{dom}(\mathcal{E})$, there holds 
\begin{align}
    \smash{\rho_{\mathcal{E}}^2(v)
    =\mathcal{D}_{I_E}^{-\partial_t u}
        (v,u)+\mathcal{D}_{I_{E^*}}^{u}
        (-\partial_t v,-\partial_t u) 
    +
    \tfrac{1}{2}
    \|(v-u)(t_{\mathtt{fin}})\|_H^2\,,}
    \label{lem:bregman_representation_primal_error_measure.0}
\end{align}
where, in the second Bregman divergence, $u\in L^p(I;V)$ 
is identified with its canonical~image~in~$L^p(I;V^{**})$.
\end{lemma}

\begin{proof}
Let $v\in\operatorname{dom}(\mathcal{E})$ be fixed, but arbitrary. Then, there holds 
\begin{align*}
\begin{aligned} 
    \rho_{\mathcal{E}}^2(v)&=I_E(v)-I_E(u)
\\&\quad+
I_{E^*}(-\partial_t v)-I_{E^*}(-\partial_t u)
\\&\quad+
\tfrac{1}{2}\|v(t_{\mathtt{fin}})\|_H^2
-
\tfrac{1}{2}\|u(t_{\mathtt{fin}})\|_H^2\,, 
\end{aligned}
\end{align*}
where, by a binomial formula, $v(0)=u(0)=u_0$ in $H$, and the integration-by-parts formula~in~time~\eqref{subsubsec:function_spaces_unsteady.3}, we have that
\begin{align*}
    \begin{aligned}
        \tfrac{1}{2}\|v(t_{\mathtt{fin}})\|_H^2
\hspace{-0.1em}-\hspace{-0.1em}
\tfrac{1}{2}\|u(t_{\mathtt{fin}})\|_H^2&=(u(t_{\mathtt{fin}}),(v\hspace{-0.1em}-\hspace{-0.1em}u)(t_{\mathtt{fin}}))_H \hspace{-0.1em}+\hspace{-0.1em}\tfrac{1}{2}\|(v\hspace{-0.1em}-\hspace{-0.1em}u)(t_{\mathtt{fin}})\|_H^2 
\\&=(u(t_{\mathtt{fin}}),(v\hspace{-0.1em}-\hspace{-0.1em}u)(t_{\mathtt{fin}}))_H  \hspace{-0.1em} -\hspace{-0.1em}(u(0),(v\hspace{-0.1em}-\hspace{-0.1em}u)(0))_H
\hspace{-0.1em}+\hspace{-0.1em}\tfrac{1}{2}\|(v\hspace{-0.1em}-\hspace{-0.1em}u)(t_{\mathtt{fin}})\|_H^2 \\
&=
    \langle \partial_t u,v\hspace{-0.1em}-\hspace{-0.1em}u\rangle_{L^p(I;V)}
\hspace{-0.1em}+\hspace{-0.1em} 
    \langle u,\partial_t(v\hspace{-0.1em}-\hspace{-0.1em}u)\rangle_{L^{p'}(I;V^*)}\hspace{-0.1em}+\hspace{-0.1em}\tfrac{1}{2}\|(v\hspace{-0.1em}-\hspace{-0.1em}u)(t_{\mathtt{fin}})\|_H^2 \,.
    \end{aligned}
\end{align*}
Due to \eqref{eq:subgradient_flow.1} and $(\partial_v E)^{-1}(t,\cdot)=\partial_{v^*} E^*(t,\cdot)$ on $V^*$ for a.e.\ $t\in I$ (\textit{cf}.\ \mbox{\cite[Cor.\ 5.2,~p.~22]{EkelandTemam1999}}), we have that 
\begin{alignat*}{3}
    -\partial_t u(t)&\in \partial_v E(t,u(t))
    &&\quad\text{ in }V^*
    &&\quad\text{ for a.e.\ }t\in I \,,\\
    u(t)&\in \partial_{v^*} E^*(t,-\partial_t u(t))
    &&\quad\text{ in }V^{**}
    &&\quad\text{ for a.e.\ }t\in I \,.
\end{alignat*}
Therefore, by the definition of the Bregman divergence \eqref{def:bregman_divergence}, we arrive at
\begin{align*}
    \begin{aligned}
        \rho_{\mathcal{E}}^2(v)&=I_E(v)-I_E(u)+
    \langle \partial_t u,v-u\rangle_{L^p(I;V)}
\\&\quad+
I_{E^*}(-\partial_t v)-I_{E^*}(-\partial_t u)+\langle u,\partial_t(v-u)\rangle_{L^{p'}(I;V^*)}
\\&\quad+\tfrac{1}{2}\|(v-u)(t_{\mathtt{fin}})\|_H^2 
\\&=\mathcal{D}_{I_E}^{-\partial_t u}
        (v,u)+\mathcal{D}_{I_{E^*}}^{u}
        (-\partial_t v,-\partial_t u) 
    +
    \tfrac{1}{2}
    \|(v-u)(t_{\mathtt{fin}})\|_H^2\,,
    \end{aligned}
\end{align*}
which is the claimed Bregman-type representation \eqref{lem:bregman_representation_primal_error_measure.0} of the primal optimal strong convexity measure \eqref{def:primal_optimal_strong_convexity_measure}.
\end{proof}

The dual counterpart of Lemma \ref{lem:bregman_representation_primal_error_measure} is obtained in the same way, but its structure is more
involved. This is because the dual energy functional \eqref{eq:dual} is written in terms of the three
energy functionals $I_{G^*}$, $I_{F^*}$, and $I_E$ and depends on the
additional dual variable $\lambda$. Therefore, the dual optimal strong
convexity measure \eqref{def:dual_optimal_strong_convexity_measure} decomposes into three Bregman divergences, corresponding~to~these~three~contributions, together with the terminal contribution for the
additional dual variable $\lambda$.

\begin{lemma}[Bregman-type representation of the dual optimal strong convexity measure]
\label{lem:bregman_representation_dual_error_measure}
Let Assumptions \ref{ass:energy_densities} and \ref{ass:sufficient_for_conjugation} be satisfied. Moreover, suppose that there exists a dual solution $(z,\mu)\in \operatorname{dom}(-\mathcal{D})$ and that the strong duality relation \eqref{thm:duality.1} and the optimality inclusions \eqref{thm:duality.2} apply.
Then, for every
$(y,\lambda)\in\operatorname{dom}(-\mathcal{D})$, there holds\vspace{-0.5mm}
\begin{align}\label{lem:bregman_representation_dual_error_measure.0}
    \rho_{-\mathcal{D}}^2(y,\lambda)
    =\mathcal{D}_{I_{G^*}}^{\nabla u}(y,z)
        +\mathcal{D}_{I_{F^*}}^{u}(-\partial_t\lambda-L^*y,-L^*z-\partial_t\mu)
        +\mathcal{D}_{I_{E}}^{-\partial_t u}(\lambda,\mu)
    +
    \tfrac{1}{2}
    \|(\lambda-\mu)(t_{\mathtt{fin}})\|_H^2 \,,\\[-6mm]\notag
\end{align}
where, in the first and second Bregman divergence, $\nabla u\in L^p(I;Y)$ and $u\in L^p(I;V)$ 
are identified with their canonical images in $L^p(I;Y^{**})$ and $L^p(I;V^{**})$, respectively. 
\end{lemma}

\begin{proof}
Let $(y,\lambda)\in\operatorname{dom}(-\mathcal{D})$ be fixed, but arbitrary. Then, there holds\vspace{-0.5mm}
\begin{align*}
    \begin{aligned}
    \rho_{-\mathcal{D}}^2(y,\lambda)&=I_{G^*}(y)-I_{G^*}(z)\\&\quad+I_{F^*}(-L^* y-\partial_t\lambda)-I_{F^*}(-L^* z-\partial_t\mu)\\&\quad+I_{E}(\lambda)-I_{E}(\mu)
          \\&\quad+ \tfrac{1}{2}\|\lambda(t_{\mathtt{fin}})\|_H^2-\tfrac{1}{2}\|\mu(t_{\mathtt{fin}})\|_H^2-((\lambda-\mu)(0),u_0)_H\,,
    \end{aligned}\\[-6mm]\notag
\end{align*}
where, by a binomial formula, $\mu(t_{\mathtt{fin}})=u(t_{\mathtt{fin}})$ in $H$ (\textit{cf}.\ \eqref{thm:duality.2.4}), $u(0)=u_0$ in $H$ (\textit{cf}.\ \eqref{eq:subgradient_flow.2}), and the integration-by-parts formula in time \eqref{subsubsec:function_spaces_unsteady.3}, we have that\vspace{-0.5mm}
\begin{align*}
    \begin{aligned}
        \tfrac{1}{2}\|\lambda(t_{\mathtt{fin}})\|_H^2-\tfrac{1}{2}\|\mu(t_{\mathtt{fin}})\|_H^2&-((\lambda-\mu)(0),u_0)_H\\&=(\mu(t_{\mathtt{fin}}),(\lambda-\mu)(t_{\mathtt{fin}}))_H   -(u_0,(\lambda-\mu)(0))_H+\tfrac{1}{2}\|(\lambda-\mu)(t_{\mathtt{fin}})\|_H^2
\\&=(u(t_{\mathtt{fin}}),(\lambda-\mu)(t_{\mathtt{fin}}))_H   -(u(0),(\lambda-\mu)(0))_H
+\tfrac{1}{2}\|(\lambda-\mu)(t_{\mathtt{fin}})\|_H^2 \\
&=
    \langle \partial_t u,\lambda-\mu\rangle_{L^p(I;V)}
+ 
    \langle u,\partial_t (\lambda-\mu)\rangle_{L^{p'}(I;V^*)}+\tfrac{1}{2}\|(\lambda-\mu)(t_{\mathtt{fin}})\|_H^2 \,.
    \end{aligned}\\[-6mm]\notag
\end{align*}
Due to the optimality inclusions \eqref{thm:duality.2.1}--\eqref{thm:duality.2.3} as well as $(\partial_y G)^{-1}(t,\cdot)=\partial_{y^*} G^*(t,\cdot)$~on~$Y^*$ for a.e.\ $t\in I$ and  $(\partial_v F)^{-1}(t,\cdot)=\partial_{v^*} F^*(t,\cdot)$ on $V^*$ for a.e.\ $t\in I$ (\textit{cf}.\ \cite[Cor.\ 5.2(5.8), p.\ 22]{EkelandTemam1999}), we~have~that\vspace{-0.5mm}
\begin{alignat*}{3}
            \nabla u(t)&\in \partial_{y^*} G^*(t,z(t))&&\quad\text{ in }Y^{**}&&\quad\text{ for a.e.\ }t\in I\,,\\
            -\partial_t u(t)&\in  \partial_v E(t,\mu(t))&&\quad\text{ in }V^*&&\quad\text{ for a.e.\ }t\in I\,,\\
            u(t) &\in \partial_{v^*} F^*(t,-\partial_t\mu(t)+\operatorname{div}z(t))&&\quad\text{ in }V^{**}&&\quad\text{ for a.e.\ }t\in I\,.
\end{alignat*}\\[-4.5mm]\notag
Therefore, by  $-\langle \nabla u,y-z\rangle_{L^{p'}(I;Y^*)}+\langle u,L^*(y-z)\rangle_{L^{p'}(I;V^*)}=0$ and the definition of the Bregman divergence~\eqref{def:bregman_divergence}, we arrive at\vspace{-0.5mm}
\begin{align*}
   \begin{aligned}
    \rho_{-\mathcal{D}}^2(y,\lambda)&=I_{G^*}(y)-I_{G^*}(z)-\langle \nabla u,y-z\rangle_{L^{p'}(I;Y^*)}
    \\&\quad+I_{F^*}(-L^* y-\partial_t\lambda)-I_{F^*}(-L^* z-\partial_t\mu)-\langle u,(-L^* y-\partial_t\lambda) -(-L^*z-\partial_t \mu)\rangle_{L^{p'}(I;V^*)}\\&\quad+I_{E}(\lambda)-I_{E}(\mu)+  \langle \partial_t u,\lambda-\mu\rangle_{L^p(I;V)}
          \\&\quad+\tfrac{1}{2}\|(\lambda-\mu)(t_{\mathtt{fin}})\|_H^2
          \\&=\mathcal{D}_{I_{G^*}}^{\nabla u}(y,z)
        +\mathcal{D}_{I_{F^*}}^{u}(-\partial_t\lambda-L^*y,-L^*z-\partial_t\mu)
        +\mathcal{D}_{I_{E}}^{-\partial_t u}(\lambda,\mu) 
    +
    \tfrac{1}{2}
    \|(\lambda-\mu)(t_{\mathtt{fin}})\|_H^2\,,
    \end{aligned}\\[-6mm]\notag
\end{align*}
which is the claimed Bregman-type representation \eqref{lem:bregman_representation_dual_error_measure.0} of the dual optimal strong convexity measure~\eqref{def:dual_optimal_strong_convexity_measure}.
\end{proof}

\begin{remark}[Alternative representation of Bregman divergences]\label{rem:alternative_representation_bregman_divergence}
Note that, by the  equality condition in the Fenchel--Young inequality (\textit{cf}.\ \cite[Prop.\ 51.2]{Zeidler1985III}), if $(X,\|\cdot\|_X)$ is a real Banach space, for proper and convex functional $f\colon X\to \mathbb{R}\cup\{+\infty\}$, for every $y\in X$, $x\in \operatorname{dom}(f)$, and $x^*\in \partial f(x)$,~we~have~that\vspace{-0.5mm}
\begin{align}\label{eq:special_bregman_representation}
    \smash{\mathcal{D}_f^{x^*}(y,x)=f^*(x^*)-\langle x^*,y\rangle_X+f(y)\,.}
\end{align}\newpage

\noindent This allows us to derive the following representations of Bregman divergences in Lemmas \ref{lem:bregman_representation_primal_error_measure} and \ref{lem:bregman_representation_dual_error_measure}:\enlargethispage{1.5mm}\vspace{-1mm}

\begin{itemize}[noitemsep,topsep=2pt,leftmargin=!,labelwidth=\widthof{$\bullet$}]
    \item[$\bullet$] If there exists $u^\sharp\in \operatorname{dom}(I_G\circ \nabla)\cap \operatorname{dom}(I_F)$ such that $I_G\circ \nabla$ or $I_F$ is continuous at $u^\sharp$, then, for every $v\in \mathcal{W}(I)$, using \eqref{eq:special_bregman_representation}, which is applicable since, due to \eqref{eq:subgradient_flow.1}, we have that 
\begin{align}\label{rem:alternative_representation_bregman_divergence.1}
    -\partial_t u\in \partial I_{E}(u)\,, 
\end{align} 
the definition of the steady primal energy functional \eqref{eq:primal_steady}, \eqref{eq:E_prime_steady}, 
and $\langle y,\nabla v\rangle_{L^p(I;Y)}+\langle -L^*y,v\rangle_{L^p(I;V)}$ $=0$, there holds\vspace{-0.5mm}
\begin{align}\label{rem:alternative_representation_bregman_divergence.2}
    \begin{aligned} 
    \mathcal{D}_{I_{E}}^{-\partial_t u}
    (v,u)
    &=
   I_{E^*}(-\partial_t u)-
    \langle -\partial_tu,v\rangle_{L^p(I;V)}+ I_{E}(v)
    \\&=\inf_{y\in L^{p'}(I;Y^*)}
    \bigl\{
        I_{G^*}(y)-\langle y,\nabla v\rangle_{L^p(I;Y)}+I_{G}(\nabla v)
        \\[-2.5mm]&\qquad\qquad\qquad+I_{F^*}(-\partial_t u-L^*y)-\langle -\partial_t u-L^*y,v\rangle_{L^p(I;V)}+I_{F}(v)
    \bigr\}  \,;
    \end{aligned}\\[-6mm]\notag
\end{align}
 
 \item[$\bullet$] If $\mu = u$ in $\mathcal{W}(I)$, then, using the definition of the Bregman divergence \eqref{def:bregman_divergence}, the definition of the steady primal energy  functional \eqref{eq:primal_steady}, that $\langle z,\nabla (v-u)\rangle_{L^p(I;Y)}+\langle -L^*z,v-u\rangle_{L^p(I;V)}=0$, and \eqref{eq:special_bregman_representation},
    which is applicable since, due to \eqref{thm:duality.2.1} and \eqref{thm:duality.2.3}, we have that\vspace{-0.5mm}
    \begin{subequations}\label{rem:alternative_representation_bregman_divergence.3}
    \begin{align}\label{rem:alternative_representation_bregman_divergence.3.1}
        z&\in \partial I_{G}(\nabla u)\,,\\
        -\partial_t u-L^* z&\in \partial I_{F}(u)\,,\label{rem:alternative_representation_bregman_divergence.3.2}\\[-6mm]\notag
    \end{align}
    \end{subequations}
    there holds\vspace{-0.5mm}
\begin{align}\label{rem:alternative_representation_bregman_divergence.4}
    \begin{aligned} 
    \mathcal{D}_{I_{E}}^{-\partial_t u}
    (v,u)
    &=
    I_{E}(v)
    -
    I_{E}(u)
    -
    \langle -\partial_tu,v-u\rangle_{L^p(I;V)}
    \\
     &=
    I_{G}(\nabla v)-
    I_{G}(\nabla u)-\langle z,\nabla (v-u)\rangle_{L^p(I;Y)}
    \\&\quad+I_{F}(v)-I_{F}(u)
    -
    \langle -\partial_tu-L^*z,v-u\rangle_{L^p(I;V)}
      \\
     &=
     \mathcal{D}_G^z(\nabla v,\nabla u)
    +\mathcal{D}_F^{-\partial_tu-L^*z}(v,u) 
    \\
     &=
    I_{G^*}(z)-\langle z,\nabla v\rangle_{L^p(I;Y)}+I_{G}(\nabla v)\\&\quad+I_{F^*}(-\partial_t u-L^* z)-\langle-\partial_t u-L^* z,v\rangle_{L^p(I;V)}+I_{F}(v) 
    \,;
    \end{aligned}\\[-6mm]\notag
\end{align}

\item[$\bullet$] If there exists $u^\sharp\in \operatorname{dom}(I_G\circ \nabla)\cap \operatorname{dom}(I_F)$ such that $I_G\circ \nabla$ or $I_F$ is continuous at $u^\sharp$, then, for every $v\in \mathcal{W}(I)$,
  using \eqref{eq:special_bregman_representation}, which is applicable since, due to \eqref{rem:alternative_representation_bregman_divergence.1} and ${(\partial I_{E})^{-1}\hspace{-0.1em}=\hspace{-0.1em}\partial (I_{E})^*\hspace{-0.1em}=\hspace{-0.1em}\partial I_{E^*}}$ (\textit{cf}.\ \cite[Cor.\ 5.2(5.8), p.\ 22]{EkelandTemam1999} and Lemma \ref{lem:conjugate_of_integral_functional}), we have that\vspace{-0.5mm}\begin{align}\label{rem:alternative_representation_bregman_divergence.5}
    u\in \partial I_{E^*}(-\partial_t u)\,,\\[-6mm]\notag
\end{align}
the definition of the steady primal energy functional \eqref{eq:primal_steady}, \eqref{eq:E_prime_steady}, 
and $\langle y,\nabla v\rangle_{L^p(I;Y)}+\langle -L^*y,v\rangle_{L^p(I;V)}$ $=0$, there holds\vspace{-0.5mm}
\begin{align}\label{rem:alternative_representation_bregman_divergence.6}
    \begin{aligned} 
    \mathcal{D}_{I_{E^*}}^{u}
    (-\partial_t v,-\partial_t u)
    &= 
    I_E(u)-\langle u,-\partial_t v\rangle_{L^p(I;V)}+I_{E^*}(-\partial_t v)
   \\
    &=
    \inf_{y\in L^{p'}(I;Y^*)}
    \bigl\{
        I_{G^*}(y)-\langle y,\nabla u\rangle_{L^p(I;Y)}+I_{G}(\nabla u)
        \\[-2.5mm]&\qquad\qquad\qquad+I_{F^*}(-\partial_t v-L^*y)-\langle -\partial_t v-L^*y,u \rangle_{L^p(I;V)}+I_F(u) \bigr\}\,;
    \end{aligned}\\[-6mm]\notag
\end{align}

    \item[$\bullet$] For every $y\in L^{p'}(I;Y^*)$, using \eqref{eq:special_bregman_representation}, which is applicable since, due to \eqref{rem:alternative_representation_bregman_divergence.3.1} and $(\partial I_{G})^{-1}=\partial (I_{G})^*=\partial I_{G^*}$ (\textit{cf}.\ \cite[Cor.\ 5.2(5.8), p.\ 22]{EkelandTemam1999} and Lemma \ref{lem:conjugate_of_integral_functional}), we have that\vspace{-0.5mm}
    \begin{align}\label{rem:alternative_representation_bregman_divergence.7}
        \nabla u\in \partial I_{G^*}(z)\,,\\[-6mm]\notag
    \end{align}
    there holds\vspace{-0.5mm}
\begin{align}\label{rem:alternative_representation_bregman_divergence.8}
    \mathcal{D}_{I_{G^*}}^{\nabla u}
    (y,z)
    =
    I_{G^*}(y)-\langle y,\nabla u\rangle_{L^p(I;Y)} +I_{G}(\nabla u)
    \,;\\[-6mm]\notag
\end{align}

\item[$\bullet$] For every $(y,\lambda)\in L^{p'}(I;Y^*)\times\mathcal{W}(I)$, using \eqref{eq:special_bregman_representation}, which is applicable since, due to \eqref{thm:duality.2.3} and $(\partial I_{F})^{-1}=\partial (I_{F})^*=\partial I_{F^*}$ (\textit{cf}.\ \cite[Cor.\ 5.2(5.8), p.\ 22]{EkelandTemam1999} and Lemma \ref{lem:conjugate_of_integral_functional}), we have that\vspace{-0.5mm}
\begin{align}\label{rem:alternative_representation_bregman_divergence.9}
    u\in \partial I_{F^*}(-\partial_t\mu - L^*z)\,,\\[-6mm]\notag
\end{align}
there holds\vspace{-0.5mm}
\begin{align}\label{rem:alternative_representation_bregman_divergence.10}
    \mathcal{D}_{I_{F^*}}^{u}
    (-\partial_t \lambda -L^*y,-\partial_t\mu - L^*z)=
    I_{F^*}(-\partial_t \lambda -L^*y)-\langle -\partial_t \lambda -L^*y,u\rangle_{L^p(I;V)} +I_{F}(u)
    \,.\\[-6mm]\notag
\end{align}
\end{itemize}
\end{remark}\pagebreak

\subsection{Primal gap identities}\label{subsec:primal_gap_identities}

\hspace{5mm}We first derive \textit{a posteriori} identities that are based directly on
the Br\'ezis--Ekeland--Nayroles principle
(\textit{cf}.\ Subsection~\ref{subsec:primal_problem}). In contrast to the
primal-dual gap identity, these identities only involve~admissible primal
approximations. 
The error measure on the left-hand side of these \textit{a posteriori} error identities
is, again, the primal optimal strong convexity measure~\eqref{def:primal_optimal_strong_convexity_measure}.
The corresponding \textit{a posteriori} error estimator on the right-hand side is the
\emph{primal gap estimator}
$\eta_{\mathcal{E}}^2\colon \mathcal{W}(I)\to
[0,+\infty]$, for every $v\in \mathcal{W}(I)$ defined by
\begin{align}\label{def:primal_gap_estimator}
    \eta_{\mathcal{E}}^2(v)
    \coloneqq
    \mathcal{E}(v)-\tfrac{1}{2}\|u_0\|_H^2\,.
\end{align}

\begin{theorem}[Primal gap identity]\label{thm:primal_gap_identity}
Let Assumption~\ref{ass:energy_densities} be satisfied. Then,  
    for every $v\in \mathcal{W}(I)$, there holds
    \begin{align}\label{thm:primal_gap_identity.0}
        \rho_{\mathcal{E}}^2(v)=\eta_{\mathcal{E}}^2(v)\,.
    \end{align}
\end{theorem} 

\begin{proof}
 \hspace{-1mm}Let \hspace{-0.15mm}$v\hspace{-0.175em}\in\hspace{-0.175em}\mathcal{W}(I)$ \hspace{-0.15mm}be \hspace{-0.15mm}fixed, \hspace{-0.15mm}but \hspace{-0.15mm}arbitrary. \hspace{-0.15mm}By \hspace{-0.1mm}the
\hspace{-0.15mm}Br\'ezis--Ekeland--Nayroles \hspace{-0.15mm}principle \hspace{-0.15mm}(\textit{cf}.\ \hspace{-0.15mm}\mbox{Proposition}~\hspace{-0.15mm}\ref{prop:brezis_ekeland_nayroles}), 
the primal solution $u\in \mathcal{W}(I)$ satisfies $\mathcal{E}(u)=\tfrac{1}{2}\|u_0\|_H^2$. 
Therefore, using the definitions
\eqref{def:primal_optimal_strong_convexity_measure} and
\eqref{def:primal_gap_estimator}, we arrive at
\begin{align*}
     \rho_{\mathcal{E}}^2(v)
    =
    \mathcal{E}(v)-\mathcal{E}(u)
    =
    \mathcal{E}(v)-\tfrac{1}{2}\|u_0\|_H^2
    =
    \eta_{\mathcal{E}}^2(v)\,,
\end{align*}
which is the claimed primal gap identity \eqref{thm:primal_gap_identity.0}.
\end{proof}

The primal gap identity \eqref{thm:primal_gap_identity.0} is exact but still abstract,
since (initially) it is expressed entirely in terms of the Br\'ezis--Ekeland--Nayroles energy functional \eqref{eq:primal}. While the primal optimal
strong convexity measure \eqref{def:primal_optimal_strong_convexity_measure} admits a meaningful Bregman-type representation (\textit{cf}.\ Lemma \ref{lem:bregman_representation_primal_error_measure}),~in~the~following~lemma, we 
derive a corresponding representation for the primal
gap estimator \eqref{def:primal_gap_estimator} that exposes its hidden Fenchel duality structure.

\begin{lemma}[Representation of the primal gap estimator]\label{lem:representation_of_primal_gap_estimator}
Let Assumption~\ref{ass:energy_densities} be satisfied and assume there \hspace{-0.1mm}exists \hspace{-0.1mm}$u^\sharp\hspace{-0.15em}\in\hspace{-0.15em} \operatorname{dom}(I_G\circ \nabla)\cap \operatorname{dom}(I_F)$ \hspace{-0.1mm}such \hspace{-0.1mm}that \hspace{-0.1mm}$I_G\circ \nabla$ \hspace{-0.1mm}or \hspace{-0.1mm}$I_F$ \hspace{-0.1mm}is \hspace{-0.1mm}continuous~\hspace{-0.1mm}at~\hspace{-0.1mm}$u^\sharp$.~\hspace{-0.1mm}Then,~\hspace{-0.1mm}for~\hspace{-0.1mm}\mbox{every}~\hspace{-0.1mm}${v\hspace{-0.15em}\in\hspace{-0.15em}\mathcal{W}(I)}$ with $v(0)=u_0$ in $H$, we have that
    \begin{align}\label{lem:representation_of_primal_gap_estimator.0}
        \begin{aligned} 
        \eta_{\mathcal{E}}^2(v)&=\inf_{y\in L^{p'}(I;Y^*)}\bigl\{\bigl(I_{G^*}(y)-\langle y,\nabla v\rangle_{L^p(I;Y)}+I_{G}(\nabla v)\bigr)\\[-2mm]&\qquad\qquad\qquad+\bigl(I_{F^*}(-L^* y-\partial_t v)-\langle -L^* y-\partial_t v, v\rangle_{L^p(I;V)}+I_{F}(v)\bigr)
        \bigr\}\,.
    \end{aligned}
    \end{align}
\end{lemma}  

\begin{proof}
    Let $v\in \mathcal{W}(I)$  with $v(0)=u_0$ in $H$ be fixed, but arbitrary.
   According to  \eqref{prop:brezis_ekeland_nayroles.1} and Lemma \ref{lem:conjugate_of_integral_functional}, we have that
    \begin{align*}
        \begin{aligned} 
        \eta_{\mathcal{E}}^2(v)&=I_{E}(v)+\langle \partial_t v,v\rangle_{L^p(I;V)}+I_{E^*}(-\partial_t v)
        \\&=I_{G}(\nabla v)+I_{F}(v)+\langle \partial_t v,v\rangle_{L^p(I;V)}
        +(I_{E})^*(-\partial_t v)\,.
        \end{aligned}
    \end{align*}
    Then, using \eqref{eq:E_prime_steady} and $\langle y,\nabla v\rangle_{L^p(I;Y)}+\langle -L^*y, v\rangle_{L^p(I;V)}=0$,
    we arrive at
    \begin{align*}
        \begin{aligned} 
        \eta_{\mathcal{E}}^2(v)
         &=\inf_{y\in L^{p'}(I;Y^*)}\bigl\{I_{G}(\nabla v)+I_{F}(v)-\langle y,\nabla v\rangle_{L^p(I;Y)}-\langle-L^*y- \partial_t v,v\rangle_{L^p(I;V)}
        \\[-2mm]&\qquad\qquad\qquad+ I_{G^*}(y)+I_{F^*}(-L^*y-\partial_t v)\bigr\}
        \\&=\inf_{y\in L^{p'}(I;Y^*)}\bigl\{\bigl(I_{G^*}(y)-\langle y,\nabla v\rangle_{L^p(I;Y)}+I_{G}(\nabla v)\bigr)\\[-2mm]&\qquad\qquad\qquad+\bigl(I_{F^*}(-L^* y-\partial_t v)-\langle -L^* y-\partial_t v, v\rangle_{L^p(I;V)}+I_{F}(v)\bigr)
        \bigr\}\,,
        \end{aligned}
    \end{align*}
    which is the claimed representation \eqref{lem:representation_of_primal_gap_estimator.0} of the primal gap estimator \eqref{def:primal_gap_estimator}. 
\end{proof}
 
The representation \eqref{lem:representation_of_primal_gap_estimator.0} of
the primal gap estimator \eqref{def:primal_gap_estimator} is still
of limited practical~use, as it contains (time) integral functionals of Fenchel
conjugates of (space) integral functionals defined~on~$V$, the evaluation of which, in general, is non-trivial.

In practice, however, admissible primal approximations
$v\in\operatorname{dom}(\mathcal{E})$ typically do not merely possess a
distributional time derivative; rather, the relevant time derivatives, and
after reconstruction likewise the corresponding divergences, are often
represented by integrable functions (\textit{i.e.}, weak derivatives). Resorting to this additional
regularity, it is possible to derive a quasi time-space integral
representation of the primal gap estimator
\eqref{def:primal_gap_estimator}, where the prefix \emph{quasi} refers to the
fact that an additional infimum over admissible flux reconstructions remains
to be evaluated. 

At the basis of the quasi time-space integral
representation of the  primal gap estimator
\eqref{def:primal_gap_estimator} is the additional Assumption \ref{ass:convex_conjugation} on the validity of a convex conjugation formula (in space), which allows us to restrict the infimum in \eqref{lem:representation_of_primal_gap_estimator.0} formed with respect to $Y^*$ to its subspace $Y^*(\operatorname{div})$.

\begin{lemma}\label{lem:integral_type_representation_of_E_prime}
    Let Assumptions \ref{ass:energy_densities} and  \ref{ass:convex_conjugation} be satisfied, and suppose that $f=\iota_p^* f_0\in L^{p'}(I;V^*)$ for some $f_0\in L^{p'}(Q;\mathbb{R}^{\ell})$ and there exists $u^\sharp\in \operatorname{dom}(I_G\circ \nabla)\cap \operatorname{dom}(I_F)$ such that $I_G\circ \nabla$ or $I_F$ is continuous~at~$u^\sharp$. Then,  for every $v^*\in L^{p'}(I;V^*)$ with $v^*=\iota_p^* v_0^*$ in $L^{p'}(I;V^*)$ for some $v_0^*\in L^{p'}(Q;\mathbb{R}^{\ell})$, we have that
\begin{align}\label{lem:integral_type_representation_of_E_prime.0}
    I_{E^*}(v^*)=\inf_{\substack{y \in L^{p'}(I;Y^*)\\ y(t)\in Y^*(\operatorname{div})\text{ for a.e.\ }t\in I}}{
\bigl\{
I_{\phi^*}(y)
+
I_{\psi^*}(v_0^* + \operatorname{div}y+f_0) 
\bigr\}}\,,
\end{align}
where
$I_{\phi^*}\colon L^{p'}(Q;\mathbb{R}^{\ell\times d})
\to\mathbb{R}\cup\{+\infty\}$ and
$I_{\psi^*}\colon L^0(Q;\mathbb{R}^{\ell})
\to\mathbb{R}\cup\{\pm\infty\}$
denote the integral functionals associated with $\phi^*$ and
$\psi^*$, respectively.
Note that, for every admissible $y \in L^{p'}(I;Y^*)$ with $y(t)\in Y^*(\operatorname{div})$ for a.e.\ $t\in I$, we have that
$\operatorname{div} y\in L^0(Q;\mathbb{R}^{\ell})$ by the distributional definition
of $\operatorname{div}$, a density argument, and Pettis' measurability theorem.
\end{lemma}

\begin{proof} 
    To begin with,  
    using $(I_G)^*=I_{G^*}=I_{\phi^*}$ (\textit{cf}.\ Lemma \ref{lem:conjugate_of_integral_functional}) together with \eqref{eq:G_prime_steady} in \eqref{eq:E_prime_steady}~as~well~as $(I_F)^*=I_{F^*}$ (\textit{cf}.\ Lemma \ref{lem:conjugate_of_integral_functional}), for every $v^*\in  L^{p'}(I;V^*)$, we find that
    \begin{align}\label{lem:integral_type_representation_of_E_prime.3}
    I_{E^*}(v^*)=\inf_{y \in L^{p'}(I;Y^*)}{
    \bigl\{
    I_{\phi^*}(y)
    +
    I_{F^*}(v^*-L^*y) 
    \bigr\}}\,.
    \end{align}
    Next, by Assumption~\ref{ass:convex_conjugation} and the assumption that $f=\iota_p^*f_0\in L^{p'}(I;V^*)$ for some $f_0\in L^{p'}(Q;\mathbb{R}^{\ell})$, for a.e. $t\in I$, every $v^*\in V^*$ with $v^*=\iota_p^* v_0^*$ in $V^*$ for some $v_0^*\in L^{p'}(\Omega;\mathbb{R}^{\ell})$, and every $y\in Y^*$,
applying~\eqref{eq:convex_conjugation_formula} to $v^*-L^*y\in V^*$, analogously~to~\eqref{eq:integral_representation_F_prime},
we have that
    \begin{align}\label{lem:integral_type_representation_of_E_prime.4}
      F^*(t,v^*-L^*y) = 
      \begin{cases}
         I_{\psi^*}^{\Omega}(t, v^*_0 + \operatorname{div}y+f_0(t))&\text{ if }y\in Y^*(\operatorname{div})\,,\\
        +\infty &\text{ else}\,.
      \end{cases}
    \end{align}
    Eventually, using \eqref{lem:integral_type_representation_of_E_prime.4} in \eqref{lem:integral_type_representation_of_E_prime.3}, we conclude that the claimed representation \eqref{lem:integral_type_representation_of_E_prime.0} applies. 
\end{proof}

The preceding representation \eqref{lem:integral_type_representation_of_E_prime.0} can now be inserted into the primal gap
estimator \eqref{def:primal_gap_estimator} with $v^*=\iota_p^*(-\partial_t v)\in L^{p'}(I;V^*)$ if $\partial_t v\in L^{p'}(Q;\mathbb{R}^{\ell})$, yielding the announced quasi
time-space integral representation.

\begin{corollary}[Quasi
time-space integral representation of the primal gap estimator]\label{thm:integral_representation_of_primal_gap_estimator}
Let the assumptions of Lemma \ref{lem:integral_type_representation_of_E_prime} be satisfied. 
Then,  
    for every $v\in \mathcal{W}(I)$  with $v(0)=u_0$ in $H$ and  $\partial_t v\in L^{p'}(Q;\mathbb{R}^{\ell})$, we have that
    \begin{align}\label{thm:integral_representation_of_primal_gap_estimator.0} 
        \eta_{\mathcal{E}}^2(v)&=\inf_{\substack{y \in L^{p'}(I;Y^*)\\ y(t)\in Y^*(\operatorname{div})\text{ for a.e.\ }t\in I}}\biggl\{\int_{Q}{\bigl(\phi^*(\cdot,\cdot,y)-y:\nabla v+\phi(\cdot,\cdot,\nabla v)\bigr)\,\mathrm{d}t\mathrm{d}x}\\[-2.5mm]&\qquad\qquad\qquad\qquad\quad\;+\int_{Q}{\bigl(\psi^*(\cdot,\cdot,\operatorname{div} y+f_0-\partial_t v)-(\operatorname{div} y+f_0-\partial_t v)\cdot \, v+\psi(\cdot,\cdot,v)\bigr)\,\mathrm{d}t\mathrm{d}x}
        \biggr\}\,.\notag
    \end{align}
\end{corollary}

\begin{proof}
    The assertion follows along the lines of the proof of Lemma \ref{lem:representation_of_primal_gap_estimator} using \eqref{lem:integral_type_representation_of_E_prime.0} instead of \eqref{eq:E_prime_steady}.
\end{proof}\pagebreak

\subsection{Dual gap identity}\label{subsec:dual_gap_identities}\enlargethispage{5mm}\vspace{-0.5mm} 

\hspace{5mm}We next derive the dual counterparts of the primal gap identities (\textit{cf}.\ Subsection \ref{subsec:primal_gap_identities}).~More~precisely, these \textit{a posteriori} error identities are based on the derived Fenchel duality framework associated with the Br\'ezis--Ekeland--Nayroles principle (\textit{cf}.\ Subsection~\ref{subsec:primal_problem}). 
The error measure on the left-hand side of these \textit{a posteriori} error identities
is, again, the dual optimal strong convexity measure~\eqref{def:dual_optimal_strong_convexity_measure}.
The corresponding \textit{a posteriori} error estimator on the right-hand side is the
\emph{dual gap estimator} $\eta_{-\mathcal{D}}^2\colon L^{p'}(I;Y^*)\times \mathcal{W}(I)\to [0,+\infty]$, for every~${(y,\lambda)\in L^{p'}(I;Y^*)\times \mathcal{W}(I)}$~defined~by\vspace{-0.5mm} 
\begin{align}\label{def:dual_gap_estimator} 
\smash{\eta_{-\mathcal{D}}^2(y,\lambda) \coloneqq -\mathcal{D}(y,\lambda) +\tfrac{1}{2}\|u_0\|_H^2\,.}\\[-6mm]\notag
\end{align}

\begin{theorem}[Dual gap identity]\label{thm:dual_gap_identity}
Let Assumptions~\ref{ass:energy_densities} and \ref{ass:sufficient_for_conjugation} be satisfied.  Moreover, suppose that there exists a dual solution $(z,\mu)\in \operatorname{dom}(-\mathcal{D})$ and that the strong duality relation \eqref{thm:duality.1} and the optimality inclusions \eqref{thm:duality.2} apply. Then, for every~$(y,\lambda)\in L^{p'}(I;Y^*)\times \mathcal{W}(I)$, we have that\vspace{-1mm} 
    \begin{align}\label{thm:dual_gap_identity.0}
        \smash{\rho_{-\mathcal{D}}^2(y,\lambda)=\eta_{-\mathcal{D}}^2(y,\lambda)}\,.\\[-6mm]\notag
    \end{align}
\end{theorem}

 \begin{proof}
Let $(y,\lambda)\in L^{p'}(I;Y^*)\times \mathcal{W}(I)$ be fixed, but arbitrary. By the strong duality relation
\eqref{thm:duality.1}, the dual solution satisfies~${\mathcal{D}(z,\mu)=\tfrac{1}{2}\|u_0\|_H^2}$. 
Hence, using the definitions
\eqref{def:dual_optimal_strong_convexity_measure} and
\eqref{def:dual_gap_estimator}, we arrive at\vspace{-0.5mm} 
\begin{align*}
    \smash{ \rho_{-\mathcal{D}}^2(y,\lambda)
    =
    -\mathcal{D}(y,\lambda)+\mathcal{D}(z,\mu)
    =
    -\mathcal{D}(y,\lambda)+\tfrac{1}{2}\|u_0\|_H^2
    =
    \eta_{-\mathcal{D}}^2(y,\lambda)\,,}\\[-6mm]\notag
\end{align*}
which is the claimed dual gap identity \eqref{thm:dual_gap_identity.0}.
\end{proof} 
 
As in the primal case, the dual gap identity \eqref{thm:dual_gap_identity.0} is exact
but still abstract, since (initially)~it~is~expressed entirely in terms of the dual energy functional \eqref{eq:dual}. While the dual optimal strong convexity measure \eqref{def:dual_optimal_strong_convexity_measure} admits a meaningful Bregman-type representation (\textit{cf}.\ Lemma \ref{lem:bregman_representation_dual_error_measure}),~in~the~\mbox{following}~\mbox{lemma}, we derive a corresponding representation for the dual gap estimator \eqref{def:dual_gap_estimator} that  likewise exposes its Fenchel duality structure.\vspace{-0.5mm}

\begin{lemma}[Representation of the dual gap estimator]\label{lem:representation_of_dual_gap_estimator}
Let Assumptions~\ref{ass:energy_densities} and \ref{ass:sufficient_for_conjugation} be satisfied. Then, for every $(y,\lambda)\in L^{p'}(I;Y^*)\times \mathcal{W}(I)$, we have that\vspace{-0.5mm} 
    \begin{align}\label{lem:representation_of_dual_gap_estimator.0}
    \begin{aligned} 
        \eta_{-\mathcal{D}}^2(y,\lambda)&=\bigl(I_{G^*}(y)-\langle y,\nabla\lambda\rangle_{L^p(I;Y)}+I_{G}(\nabla \lambda)\bigr)
        \\&\quad+\bigl(I_{F^*}(-L^* y-\partial_t\lambda)-\langle-L^* y-\partial_t\lambda,\lambda\rangle_{L^p(I;V)}+I_{F}(\lambda)\bigr) 
         +\tfrac{1}{2}\|\lambda(0)-u_0\|_H^2\,.
    \end{aligned}\\[-6mm]\notag
    \end{align}
\end{lemma}

\begin{proof}
    Let $(y,\lambda)\in L^{p'}(I;Y^*)\times \mathcal{W}(I)$ be fixed, but arbitrary. Then, using \eqref{eq:dual}, \eqref{eq:primal_steady}, that $\langle y,\nabla \lambda\rangle_{L^p(I;Y)}+\langle -L^*y,\lambda\rangle_{L^p(I;V)}=0$, and the integration-by-parts formula in time \eqref{subsubsec:function_spaces_unsteady.3}, we find that\vspace{-0.5mm}
    \begin{align*}
    \begin{aligned} 
        \eta_{-\mathcal{D}}^2(y,\lambda)
           &=
        I_{G^*}(y)-\langle y,\nabla \lambda\rangle_{L^p(I;Y)}+I_{G}(\nabla\lambda)
        \\&\quad+I_{F^*}(-L^* y-\partial_t\lambda)-\langle -L^* y-\partial_t\lambda, \lambda\rangle_{L^p(I;V)} +I_{F}(\lambda)\\&\quad
         -\langle \partial_t\lambda, \lambda\rangle_{L^p(I;V)}  +\tfrac{1}{2}\|\lambda(t_{\mathtt{fin}})\|_H^2-(\lambda(0),u_0)_H+\tfrac{1}{2}\|u_0\|_H^2
         \\&=
        I_{G^*}(y)-\langle y,\nabla \lambda\rangle_{L^p(I;Y)}+I_{G}(\nabla\lambda)
        \\&\quad+I_{F^*}(-L^* y-\partial_t\lambda)-\langle -L^* y-\partial_t\lambda, \lambda\rangle_{L^p(I;V)} +I_{F}(\lambda)
         +\tfrac{1}{2}\|\lambda(0)-u_0\|_H^2\,,
    \end{aligned}\\[-6mm]\notag
    \end{align*}
    which is the claimed representation \eqref{lem:representation_of_dual_gap_estimator.0} of the dual gap estimator \eqref{def:dual_gap_estimator}.
\end{proof}

 By \hspace{-0.1mm}analogy \hspace{-0.1mm}with \hspace{-0.1mm}Corollary \hspace{-0.1mm}\ref{thm:integral_representation_of_primal_gap_estimator}, \hspace{-0.1mm}under \hspace{-0.1mm}the \hspace{-0.1mm}additional \hspace{-0.1mm}Assumption \hspace{-0.1mm}\ref{ass:convex_conjugation} \hspace{-0.1mm}on \hspace{-0.1mm}the \hspace{-0.1mm}validity \hspace{-0.1mm}of \hspace{-0.1mm}a \hspace{-0.1mm}convex~\hspace{-0.1mm}conjuga\-tion \hspace{-0.1mm}formula \hspace{-0.1mm}(in \hspace{-0.1mm}space), \hspace{-0.1mm}we \hspace{-0.1mm}can \hspace{-0.1mm}derive \hspace{-0.1mm}a \hspace{-0.1mm}time-space \hspace{-0.1mm}integral
\hspace{-0.1mm}representation \hspace{-0.1mm}for~\hspace{-0.1mm}the~\hspace{-0.1mm}dual~\hspace{-0.1mm}gap~\hspace{-0.1mm}\mbox{estimator}~\hspace{-0.1mm}\eqref{def:dual_gap_estimator}.\vspace{-0.5mm}

\begin{corollary}[Time-space integral functional representation of the dual gap estimator]\label{thm:time_integral_representation_of_gap_estimator}
Let Assumption \ref{ass:sufficient_for_conjugation} and
 the assumptions of Lemma \ref{lem:integral_type_representation_of_E_prime}
be satisfied. Then, for every $(y,\lambda)\in \operatorname{dom}(-\mathcal{D})$ with $\operatorname{div}y,\partial_t \lambda\in  L^{p'}(Q;\mathbb{R}^{\ell})$, we have that\vspace{-0.5mm} 
    \begin{align*}
    \begin{aligned} 
        \eta_{-\mathcal{D}}^2(y,\lambda)&=\int_Q{\bigl(\phi^*(\cdot,\cdot,y)-y:\nabla \lambda +\phi(\cdot,\cdot,\nabla \lambda)\bigr)\,\mathrm{d}t\mathrm{d}x}  
        \\&\quad+\int_Q{\bigl(\psi^*(\cdot,\cdot,\operatorname{div}y+f_0-\partial_t\lambda)-(\operatorname{div}y+f_0-\partial_t\lambda)\cdot\lambda+\psi(\cdot,\cdot,\lambda)\bigr)\,\mathrm{d}t\mathrm{d}x}   +\tfrac{1}{2}\|\lambda(0)-u_0\|_H^2\,.
    \end{aligned}\\[-6mm]\notag
    \end{align*}
\end{corollary}

\begin{proof}
    The assertion follows along the lines of the proof of Lemma \ref{lem:representation_of_dual_gap_estimator} using \eqref{lem:integral_type_representation_of_E_prime.4}.
\end{proof}

 \newpage 
 \section{Applications}\label{sec:applications}

\hspace{5mm}In this section, we apply the Fenchel duality framework from Section
\ref{sec:fenchel_duality_framework} and derive corresponding primal and dual gap identities from
Section \ref{sec:duality_based_a_posteriori_error_control}. For each model problem, we
proceed in the same way: we specify the energy densities, identify the corresponding steady
and unsteady primal~and~dual~functionals, and specialize the abstract subgradient flow and
the primal and dual gap identities~to~the~concrete~setting.\linebreak As model problems serve the unsteady heat equation, the unsteady Stokes equations, the unsteady Navier--Lam\'e equations,  the unsteady Bingham flow through a pipe, the unsteady obstacle problem, and the unsteady elasto-plastic~torsion~problem.
 
 \subsection{The unsteady heat equation}\label{subsec:heat_equation}

\hspace{5mm}In this subsection, we consider the \emph{unsteady heat equation}, which was first formulated by J.-B.-J.\ Fourier in 1807 and later developed in his \emph{Théorie analytique de la chaleur} (1822) (\textit{cf}.\ \cite{Fourier1822,Narasimhan2009}) and models the temporal evolution of the temperature in a heat-conducting body due to thermal diffusion and distributed heat sources.

Let $\ell=1$ and $p=2$, \textit{i.e.},
$V=W^{1,2}_D(\Omega;\mathbb{R}^1)$,
$Y=L^2(\Omega;\mathbb{R}^d)$, and
$H=L^2(\Omega;\mathbb{R}^1)$, where $V$ is equipped with the gradient
norm $\|\cdot\|_V\coloneqq\|\nabla(\cdot)\|_\Omega$ in $V$.
Moreover, let the energy densities
$\phi\colon\mathbb{R}^d\to\mathbb{R}$~and
$\psi\colon Q\times\mathbb{R}\to\mathbb{R}$, for a.e.\
$(t,x)\in Q$, every $a\in\mathbb{R}^d$, and $b\in\mathbb{R}$,
respectively, be defined by
\begin{subequations}\label{subsec:heat_equation.1}
\begin{align}\label{subsec:heat_equation.1.1}
    \phi(a)
    &\coloneqq
    \tfrac{1}{2}\vert a\vert^2\,,
    \\\label{subsec:heat_equation.1.2}
    \psi(t,x,b)
    &\coloneqq 0\,,
\end{align}
\end{subequations}
so that $F\colon \hspace{-0.1em}I\times V\hspace{-0.1em}\to\hspace{-0.1em} \mathbb{R}$, for some $f\hspace{-0.1em}\in\hspace{-0.1em} L^2(I;V^*)$, is given via $F(t,v)\hspace{-0.1em}\coloneqq \hspace{-0.1em}-\langle f(t),v\rangle_V$ for a.e.\ $t\hspace{-0.1em}\in\hspace{-0.1em} I$~and~all~$v\hspace{-0.1em}\in\hspace{-0.1em} V$.  For the choice
\eqref{subsec:heat_equation.1},
Assumptions~\ref{ass:energy_densities},
\ref{ass:convex_conjugation}, and
\ref{ass:sufficient_for_conjugation}~are~\mbox{readily}~\mbox{verified}.

For a.e.\ $t\hspace{-0.1em}\in\hspace{-0.1em} I$, the steady primal energy functional
$E(t,\cdot)\colon \hspace{-0.1em}V\hspace{-0.1em}\to\hspace{-0.1em}\mathbb{R}\cup\{+\infty\}$, for every
$v\hspace{-0.1em}\in\hspace{-0.1em} V$,~is~given~via
\begin{align}\label{subsec:heat_equation.2}
    E(t,v)
    =
    \tfrac{1}{2}\|\nabla v\|_\Omega^2
    -
    \langle f(t),v\rangle_V\,,
\end{align}
and, due to \eqref{eq:E_prime_steady}, the Fenchel conjugate functional
$E^*(t,\cdot)\colon \hspace{-0.1em}V^*\hspace{-0.1em}\to\hspace{-0.1em}\mathbb{R}\cup\{+\infty\}$ of
\eqref{subsec:heat_equation.2},~for~every~$v^*\hspace{-0.1em}\in\hspace{-0.1em} V^*$, is given via
\begin{align}\label{subsec:heat_equation.3}
\begin{aligned}
    E^*(t,v^*)
    =
    \inf_{\substack{
        y\in Y^*\\
        L^*y=f(t)+v^*\text{ in }V^*
    }}
    \bigl\{
        \tfrac{1}{2}\|y\|_\Omega^2
    \bigr\}
    =
    \tfrac{1}{2}\|f(t)+v^*\|_{V^*}^2\,.
\end{aligned}
\end{align}
Here, we used the standard
representation of the $V^*$-norm induced by the gradient inner product: $\|v^*\|_{V^*}
   \hspace{-0.1em} =\hspace{-0.1em}
    \inf_{y\in Y^*: L^*y=v^*}
    \|y\|_{\Omega}$ for all $v^*\in V^*$, which, in turn, is based on~that~$\|\cdot\|_{V}=\|\nabla (\cdot)\|_{\Omega}$~on~$V$.

If the Laplace operator $-\Delta\coloneqq L^*\circ \nabla \colon V\to V^*$, for every  $v,w\in V$, is defined by 
\begin{align*}
    \langle(-\Delta)v,w\rangle_V
    \coloneqq
    (\nabla v,\nabla w)_\Omega\,,
\end{align*}
then, given an arbitrary initial datum $u_0\in H$, the abstract
subgradient-flow problem \eqref{eq:subgradient_flow} is given via the 
unsteady heat equation: find $u\in\mathcal{W}(I)$ such that
\begin{subequations}\label{subsec:heat_equation.4}
\begin{alignat}{2}\label{subsec:heat_equation.4.1}
    \partial_tu(t)+(-\Delta)u(t)
    &=
    f(t)
    &&\quad\text{in }V^*\quad\text{for a.e.\ }t\in I\,,
    \\\label{subsec:heat_equation.4.2}
    u(0)
    &=
    u_0
    &&\quad\text{in }H\,.
\end{alignat}
\end{subequations}
By the standard existence theory for linear parabolic problems in an
evolution triple (\textit{cf}.\
\cite[Thm.~23.A]{Zeidler1990IIA}), the  unsteady heat equation
\eqref{subsec:heat_equation.4} admits a unique solution
$u\in\mathcal{W}(I)$. According to the
Br\'ezis--Ekeland--Nayroles principle
(\textit{cf}.\ Proposition~\ref{prop:brezis_ekeland_nayroles}), this
solution is equivalently characterized~as~the~\mbox{primal}~solution,
\textit{i.e.}, as a minimizer of the unsteady primal energy functional
$\mathcal{E}\colon\mathcal{W}(I)\to\mathbb{R}\cup\{+\infty\}$, for
every $v\in\mathcal{W}(I)$ defined by
\begin{align}\label{subsec:heat_equation.5}
    \mathcal{E}(v)
    \coloneqq
    \tfrac{1}{2}\|\nabla v\|_Q^2
    -
    \langle f,v\rangle_{L^2(I;V)}
    +
    \tfrac{1}{2}\|f-\partial_tv\|_{L^2(I;V^*)}^2
    +
    \tfrac{1}{2}\|v(t_{\mathtt{fin}})\|_H^2
    +
    \chi_{\{u_0\}}(v(0))\,.
\end{align}\pagebreak

\noindent According to Theorem~\ref{thm:duality}(\hyperlink{thm:duality.i}{i}),
the unsteady dual energy functional
$\mathcal{D}\colon L^2(I;Y^*)\times\mathcal{W}(I)
\to\mathbb{R}\cup\{-\infty\}$, for every
$(y,\lambda)\in L^2(I;Y^*)\times\mathcal{W}(I)$, is given via
\begin{align}\label{subsec:heat_equation.6}
\begin{aligned}
    \mathcal{D}(y,\lambda)
    \coloneqq
    &-
    \tfrac{1}{2}\|y\|_Q^2
    -
    \chi_{\{-f\}}(-\partial_t\lambda-L^*y)
    -
    \tfrac{1}{2}\|\nabla\lambda\|_Q^2
    \\
    &+
    \langle f,\lambda\rangle_{L^2(I;V)}
    -
    \tfrac{1}{2}\|\lambda(t_{\mathtt{fin}})\|_H^2
    +
    (\lambda(0),u_0)_H\,,
\end{aligned}
\end{align}
where the indicator functional $\chi_{\{-f\}}\colon L^2(I;V^*)\to \mathbb{R}\cup\{+\infty\}$, for every $v^*\in L^2(I;V^*)$, is defined by
\begin{align*}
    \chi_{\{-f\}}(v^*)\coloneqq \begin{cases}
        0&\text{ if }v^*=-f\text{ in }L^2(I;V^*)\,,\\
        +\infty&\text{ else}\,.
    \end{cases}
\end{align*}
Since the continuity condition in
Theorem~\ref{thm:duality}(\hyperlink{thm:duality.ii}{ii}) is satisfied
in the present quadratic setting, there exists a dual solution
$(z,\mu)\in L^2(I;Y^*)\times\mathcal{W}(I)$ and the corresponding
strong duality relation~\eqref{thm:duality.1}~applies. 
In particular, the 
optimality inclusions \eqref{thm:duality.2} together with the strict
convexity of the steady primal energy functional \eqref{subsec:heat_equation.2} imply that the dual
solution is unique and given via
\begin{subequations}\label{subsec:heat_equation.opt}
\begin{alignat}{2}\label{subsec:heat_equation.opt.1}
    z
    &=
    \nabla u
    &&\quad\text{ in }L^2(I;Y^*)\,,
    \\\label{subsec:heat_equation.opt.2}
    \partial_tu+L^*z
    &=
    f
    &&\quad\text{in }L^2(I;V^*)\,,
    \\\label{subsec:heat_equation.opt.3}
    \mu
    &=
    u
    &&\quad\text{in }\mathcal{W}(I)\,.
\end{alignat}
\end{subequations}

\hspace{-1mm}Next, \hspace{-0.1mm}we \hspace{-0.1mm}record \hspace{-0.1mm}the \hspace{-0.1mm}corresponding \hspace{-0.1mm}primal \hspace{-0.1mm}and \hspace{-0.1mm}dual \hspace{-0.1mm}gap \hspace{-0.1mm}identities
\hspace{-0.1mm}(\textit{cf}.\ \hspace{-0.1mm}Theorems~\ref{thm:primal_gap_identity} \hspace{-0.1mm}and
\hspace{-0.1mm}\ref{thm:dual_gap_identity}).~\hspace{-0.1mm}To~\hspace{-0.1mm}this~\hspace{-0.1mm}end, we note that, for every
$v,w\in \mathcal{W}(I)$ and  $y\in  L^2(I;Y^*)$ with $ L^*y=f-\partial_tv$ in $L^2(I;V^*)$, due to $I_{G^*}=I_G=\frac{1}{2}\|\cdot\|_{Q}^2$, $I_{F}=-\langle f,\cdot\rangle_{L^2(I;V)}$, $I_{F^*}=\chi_{\{-f\}}$, and a binomial formula, there holds
\begin{align}\label{subsec:heat_equation.7}
\begin{aligned}
    &I_{G^*}(y)
    -
    \langle y,\nabla w\rangle_{L^2(I;Y)}
    +
    I_G(\nabla w)
    \\
    &\quad+
    \chi_{\{-f\}}(-\partial_tv-L^*y)
    -
    \langle-\partial_tv-L^*y,w\rangle_{L^2(I;V)}
    -
    \langle f,w\rangle_{L^2(I;V)}
    \\
    &=
    \tfrac{1}{2}\|y-\nabla w\|_Q^2\,.
\end{aligned}
\end{align}

First, using repeatedly the elementary identity \eqref{subsec:heat_equation.7}, we derive the corresponding primal~gap~identity.

\begin{lemma}[Primal gap identity for the unsteady heat equation]
\label{lem:primal_gap_identity_heat_equation}
For every $v\in\mathcal{W}(I)$ with $v(0)=u_0$ in $H$, there holds
\begin{align}\label{lem:primal_gap_identity_heat_equation.0}
\begin{aligned}
    &\tfrac{1}{2}\|\nabla(v-u)\|_Q^2
    +
    \tfrac{1}{2}\|\partial_t(v-u)\|_{L^2(I;V^*)}^2
    +
    \tfrac{1}{2}\|(v-u)(t_{\mathtt{fin}})\|_H^2
    \\&=
    \inf_{\substack{
        y\in L^2(I;Y^*)\\
        L^*y=f-\partial_tv\text{ in }L^2(I;V^*)
    }}
    \bigl\{
        \tfrac{1}{2}\|y-\nabla v\|_Q^2
    \bigr\}
    \\&=
    \tfrac{1}{2}
    \|\partial_tv+(-\Delta)v-f\|_{L^2(I;V^*)}^2\,.
\end{aligned}
\end{align}
\end{lemma}

\begin{proof}
Let $v\in\mathcal{W}(I)$ with $v(0)=u_0$ in $H$ be fixed, but
arbitrary.

\emph{$\bullet$ Optimal strong convexity measure.}
Due to the alternative Bregman divergence representation
\eqref{rem:alternative_representation_bregman_divergence.4} (which is
applicable since $\mu=u$),
\eqref{subsec:heat_equation.7} (applied with
$(v,w,y)=(u,v,z)$), and \eqref{subsec:heat_equation.opt},~we~have~that
\begin{align*}
\begin{aligned}
    \smash{\mathcal{D}_{I_E}^{-\partial_tu}(v,u)}
    &=
    \tfrac{1}{2}\|z-\nabla v\|_Q^2
    \\&=
    \tfrac{1}{2}\|\nabla(v-u)\|_Q^2\,.
\end{aligned}
\end{align*} 
Moreover, due to the alternative Bregman divergence representation
\eqref{rem:alternative_representation_bregman_divergence.6} (which is applicable since both $I_G\circ\nabla$ and $I_F$ are continuous),
\eqref{subsec:heat_equation.7} (applied with
$(v,w,y)=(v,u,y)$), and \eqref{subsec:heat_equation.3},~we~have~that
\begin{align*}
\begin{aligned}
    \mathcal{D}_{I_{E^*}}^u
    (-\partial_tv,-\partial_tu)
    &=
    \inf_{\substack{
        y\in L^2(I;Y^*)\\
        L^*y=f-\partial_tv\text{ in }L^2(I;V^*)
    }}
    \bigl\{
        \tfrac{1}{2}\|y-\nabla u\|_Q^2
    \bigr\}
    \\&=
    \tfrac{1}{2}
    \|\partial_t(v-u)\|_{L^2(I;V^*)}^2\,.
\end{aligned}
\end{align*}
Therefore, the Bregman-type representation of the primal error measure
\eqref{lem:bregman_representation_primal_error_measure.0} yields
\begin{align}\label{lem:primal_gap_identity_heat_equation.3}
\begin{aligned}
    \rho_{\mathcal{E}}^2(v)
    =
    \tfrac{1}{2}\|\nabla(v-u)\|_Q^2
    +
    \tfrac{1}{2}\|\partial_t(v-u)\|_{L^2(I;V^*)}^2 +
    \tfrac{1}{2}\|(v-u)(t_{\mathtt{fin}})\|_H^2\,.
\end{aligned}
\end{align}

\emph{$\bullet$ Primal gap estimator.}
Using the representation
\eqref{lem:representation_of_primal_gap_estimator.0}  of the primal gap
estimator \eqref{def:primal_gap_estimator} (which is applicable since both $I_G\circ\nabla$ and $I_F$ are continuous),
\eqref{subsec:heat_equation.7} (applied with
$(v,w,y)=(v,v,y)$),~and~\eqref{subsec:heat_equation.3}, we find that
\begin{align}\label{lem:primal_gap_identity_heat_equation.4}
\begin{aligned}
    \eta_{\mathcal{E}}^2(v)
    &=
    \inf_{\substack{
        y\in L^2(I;Y^*)\\
        L^*y=f-\partial_tv\text{ in }L^2(I;V^*)
    }}
    \bigl\{
        \tfrac{1}{2}\|y-\nabla v\|_Q^2
    \bigr\}
    \\&=
    \tfrac{1}{2}
    \|\partial_tv+(-\Delta)v-f\|_{L^2(I;V^*)}^2\,.
\end{aligned}
\end{align}

Finally, using the representations
\eqref{lem:primal_gap_identity_heat_equation.3} and
\eqref{lem:primal_gap_identity_heat_equation.4} and the general primal
gap identity \eqref{thm:primal_gap_identity.0}, we arrive at the claimed
representation \eqref{lem:primal_gap_identity_heat_equation.0} of the
primal gap identity for the unsteady~heat~equation~\eqref{subsec:heat_equation.4}.
\end{proof}

Next, we derive the corresponding dual~gap~identity.

\begin{lemma}[Dual gap identity for the unsteady heat equation]
\label{lem:dual_gap_identity_heat_equation}
For every
$(y,\lambda)\in L^2(I;Y^*)\times\mathcal{W}(I)$ with
\begin{align}\label{lem:dual_gap_identity_heat_equation.-1}
    \operatorname{div}y
    =
    \partial_t\lambda-f
    \quad\text{in }L^2(I;V^*)\,,
\end{align}
there holds
\begin{align}\label{lem:dual_gap_identity_heat_equation.0}
\begin{aligned}
    \tfrac{1}{2}\|y-z\|_Q^2
    +
    \tfrac{1}{2}\|\nabla(\lambda-u)\|_Q^2
    +
    \tfrac{1}{2}\|(\lambda-u)(t_{\mathtt{fin}})\|_H^2
    =
    \tfrac{1}{2}\|y-\nabla\lambda\|_Q^2
    +
    \tfrac{1}{2}\|\lambda(0)-u_0\|_H^2\,.
\end{aligned}
\end{align}
\end{lemma}

\begin{proof}
Let $(y,\lambda)\in L^2(I;Y^*)\times\mathcal{W}(I)$ with
\eqref{lem:dual_gap_identity_heat_equation.-1} be fixed, but arbitrary.

\emph{$\bullet$ Optimal strong convexity measure.}
Due to the alternative Bregman divergence representation
\eqref{rem:alternative_representation_bregman_divergence.8}, $I_{G^*}=I_{G}=\tfrac{1}{2}\|\cdot\|_{Q}^2$, a
binomial formula, and \eqref{subsec:heat_equation.opt.1}, we have that
\begin{align*}
    \mathcal{D}_{I_{G^*}}^{\nabla u}(y,z)
    &= \tfrac{1}{2}\|y-\nabla u\|_Q^2
   \\& =
    \tfrac{1}{2}\|y-z\|_Q^2\,,
\end{align*}
due to the alternative Bregman divergence representation
\eqref{rem:alternative_representation_bregman_divergence.10}, $I_F= -\langle f,\cdot\rangle_{L^2(I;V)}$,
$I_{F^*}=\chi_{\{-f\}}$,
\eqref{lem:dual_gap_identity_heat_equation.-1}, and
\eqref{subsec:heat_equation.4.1}, we have that
\begin{align*}
    \mathcal{D}_{I_{F^*}}^u
    (-\partial_t\lambda-L^*y,-\partial_tu-L^*z)
    &=\chi_{\{-f\}}(-\partial_t\lambda-L^*y)-\langle f-\partial_t\lambda-L^*y,u\rangle_{L^2(I;V)}
    \\&=0\,,
\end{align*}
and, due to the alternative Bregman divergence representation
\eqref{rem:alternative_representation_bregman_divergence.4} (which is
applicable since $\mu=u$ in $\mathcal{W}(I)$),
\eqref{subsec:heat_equation.7} (applied with
$(v,w,y)=(u,\lambda,z)$), and \eqref{subsec:heat_equation.opt}, we have
that
\begin{align*}
\begin{aligned}
    \mathcal{D}_{I_E}^{-\partial_tu}(\lambda,u)
    &=
    \tfrac{1}{2}\|z-\nabla\lambda\|_Q^2
    \\&=
    \tfrac{1}{2}\|\nabla(\lambda-u)\|_Q^2\,.
\end{aligned}
\end{align*}
Therefore, the Bregman-type representation of the dual error measure
\eqref{lem:bregman_representation_dual_error_measure.0} yields
\begin{align}\label{lem:dual_gap_identity_heat_equation.4}
\begin{aligned}
    \rho_{-\mathcal{D}}^2(y,\lambda)
    =
    \tfrac{1}{2}\|y-z\|_Q^2
    +
    \tfrac{1}{2}\|\nabla(\lambda-u)\|_Q^2
    +
    \tfrac{1}{2}\|(\lambda-u)(t_{\mathtt{fin}})\|_H^2\,.
\end{aligned}
\end{align}

\emph{$\bullet$ Dual gap estimator.}
Using the representation
\eqref{lem:representation_of_dual_gap_estimator.0} of the dual gap
estimator \eqref{def:dual_gap_estimator},
\eqref{subsec:heat_equation.7} (applied with
$(v,w,y)=(\lambda,\lambda,y)$), and a binomial formula, we find that
\begin{align}\label{lem:dual_gap_identity_heat_equation.5}
    \eta_{-\mathcal{D}}^2(y,\lambda)
    =
    \tfrac{1}{2}\|y-\nabla\lambda\|_Q^2
    +
    \tfrac{1}{2}\|\lambda(0)-u_0\|_H^2\,.
\end{align}

Finally, using the representations
\eqref{lem:dual_gap_identity_heat_equation.4} and
\eqref{lem:dual_gap_identity_heat_equation.5} and the general dual gap
identity \eqref{thm:dual_gap_identity.0}, we arrive at the claimed
representation \eqref{lem:dual_gap_identity_heat_equation.0} of the
dual gap identity for the unsteady~heat~equation~\eqref{subsec:heat_equation.4}.
\end{proof}\pagebreak

\subsection{The unsteady Stokes equations}\label{subsec:stokes_equations}\enlargethispage{1.5mm}

\hspace{5mm}In this subsection, we consider the \emph{unsteady Stokes equations}, which are named after G.~G.~Stokes, who systematically derived the governing equations for viscous fluid motion in his 1845 work on the internal friction of fluids (\textit{cf}.\ \cite{Stokes2009}), and which describe the motion of incompressible viscous Newtonian fluids in the regime of low Reynolds numbers.

Let $\ell\hspace{-0.1em}=\hspace{-0.1em}d$ and $p\hspace{-0.1em}=\hspace{-0.1em}2$, \textit{i.e.},
$V\hspace{-0.1em}=\hspace{-0.1em}W^{1,2}_D(\Omega;\mathbb{R}^d)$,
$Y\hspace{-0.1em}=\hspace{-0.1em}L^2(\Omega;\mathbb{R}^{d\times d})$, and
$H\hspace{-0.1em}=\hspace{-0.1em}L^2(\Omega;\mathbb{R}^d)$, where $V$~is~\mbox{equipped} with the gradient
norm $\|\cdot\|_V \coloneqq \|\nabla(\cdot)\|_\Omega$ in $V$.
For simplicity, we assume~that~${\Gamma_D=\partial\Omega}$~and~define
\begin{align*}
    V_{\sigma}
    \coloneqq
    \smash{\bigl\{
        v\in V
        \mid
        \operatorname{div}v=0\text{ a.e.\ in }\Omega
    \bigr\}}
    \quad\text{and}\quad
    H_{\sigma}
    \coloneqq
    \operatorname{cl}_H V_{\sigma}\,.
\end{align*}
Moreover, let the energy densities
$\phi\colon\smash{\mathbb{R}^{d\times d}}\to\mathbb{R}\cup\{+\infty\}$ and
$\psi\colon Q\times\smash{\mathbb{R}^d}\to\mathbb{R}$, for a.e.\
$(t,x)\in Q$, every $A\in\mathbb{R}^{d\times d}$ and
$a\in\mathbb{R}^d$, respectively, be defined by\vspace{-0.5mm}
\begin{subequations}\label{subsec:stokes_equations.1}
\begin{align}\label{subsec:stokes_equations.1.1}
    \phi(A)
    &\coloneqq
    \tfrac{1}{2}|A|^2
    +
    \chi_{\{0\}}(\operatorname{tr}A)\,,
    \\\label{subsec:stokes_equations.1.2}
    \psi(t,x,a)
    &\coloneqq
    0\,,\\[-6mm]\notag
\end{align}
\end{subequations}
so that $G\colon Y\to \mathbb{R}\cup\{+\infty\}$ is given via $G(y)\coloneqq\frac{1}{2}\|y\|_{\Omega}^2+\chi_{\{0\}}(\operatorname{tr}y)$ for all $y\in Y$,  where $\chi_{\{0\}}\colon L^2(\Omega;\mathbb{R}^1)\to \mathbb{R}\cup\{+\infty\}$, for every $v\in L^2(\Omega;\mathbb{R}^1)$, is given via\vspace{-0.5mm}
\begin{align}\label{subsec:stokes_equations.1.3}
    \chi_{\{0\}}(v)\coloneqq \begin{cases}
        0&\text{ if }v=0\text{ a.e.\ in }\Omega\,,\\
        +\infty&\text{ else}\,,
    \end{cases}\\[-6mm]\notag
\end{align}
and $F\colon \hspace{-0.1em}I\times V\hspace{-0.1em}\to\hspace{-0.1em} \mathbb{R}$, for some $f\hspace{-0.1em}\in\hspace{-0.1em} L^2(I;V^*)$, is given via $F(t,v)\hspace{-0.1em}\coloneqq \hspace{-0.1em}-\langle f(t),v\rangle_V$ for a.e.\ $t\hspace{-0.1em}\in\hspace{-0.1em} I$~and~all~$v\hspace{-0.1em}\in\hspace{-0.1em} V$. 
For the choice
\eqref{subsec:stokes_equations.1}, Assumptions~\ref{ass:energy_densities},
\ref{ass:convex_conjugation}, and
\ref{ass:sufficient_for_conjugation} are readily verified.

For a.e.\ $t\hspace{-0.1em}\in\hspace{-0.1em} I$, the steady primal energy functional
$E(t,\cdot)\colon \hspace{-0.1em}V\hspace{-0.1em}\to\hspace{-0.1em}\mathbb{R}\cup\{+\infty\}$, for every
$v\hspace{-0.1em}\in\hspace{-0.1em} V$,~is~given~via\vspace{-0.5mm}
\begin{align}\label{subsec:stokes_equations.2}
    E(t,v)
    =
    \tfrac{1}{2}\|\nabla v\|_\Omega^2
    +
    \chi_{\{0\}}(\operatorname{tr}\nabla v)
    -
    \langle f(t),v\rangle_V\,,\\[-6mm]\notag
\end{align}
and, \hspace{-0.1mm}due \hspace{-0.1mm}to \hspace{-0.1mm}\eqref{eq:E_prime_steady} \hspace{-0.1mm}and \hspace{-0.1mm}since \hspace{-0.1mm}the \hspace{-0.1mm}Fenchel \hspace{-0.1mm}conjugate \hspace{-0.1mm}$\phi^*\hspace{-0.1em}\colon \hspace{-0.175em}\mathbb{R}^{d\times d}\hspace{-0.175em}\to\hspace{-0.175em} \mathbb{R}$ \hspace{-0.1mm}of \hspace{-0.1mm}\eqref{subsec:stokes_equations.1.1}, \hspace{-0.1mm}for \hspace{-0.1mm}every \hspace{-0.1mm}$A^*\hspace{-0.175em}\in \hspace{-0.175em}\mathbb{R}^{d\times d}$,~\hspace{-0.1mm}is~\hspace{-0.1mm}given~\hspace{-0.1mm}via\vspace{-0.5mm}
\begin{align}\label{subsec:stokes_equations.3}
     \phi^*(A^*)
    =
    \tfrac{1}{2}|\operatorname{dev}A^*|^2\,,\\[-6mm]\notag
\end{align}
where $ \operatorname{dev}A^*
   \hspace{-0.15em} \coloneqq\hspace{-0.15em}
    A^*\hspace{-0.05em}-\hspace{-0.05em}\tfrac{1}{d}\operatorname{tr}(A^*)\mathrm{I}_d\hspace{-0.15em}\in\hspace{-0.15em}\mathbb{R}^{d\times d}$ denotes the deviatoric part of $A^*$, the Fenchel conjugate~functional
$E^*(t,\cdot)\colon V^*\to\mathbb{R}\cup\{+\infty\}$ of
\eqref{subsec:stokes_equations.2}, for every $v^*\in V^*$, is given via\vspace{-0.5mm}
\begin{align}\label{subsec:stokes_equations.4}
    \begin{aligned} 
    E^*(t,v^*)
    &=
    \inf_{\substack{
        y\in Y^*\\
        L^*y=f(t)+v^*\text{ in }V^*
    }}
    \bigl\{
        \tfrac{1}{2}\|\operatorname{dev}y\|_\Omega^2
    \bigr\}=\tfrac{1}{2}\|(f(t)+v^*)|_{V_\sigma}\|_{V_\sigma^*}^2\,.
    \end{aligned}\\[-6mm]\notag
\end{align}
Here, we used the standard representation of the $V_\sigma^*$-norm induced by the gradient inner product:
$\|v^*_\sigma\|_{V_\sigma^*}\hspace{-0.15em}=\hspace{-0.15em}\inf_{y\in Y^*:L^*y|_{V_\sigma}=v^*_\sigma}{\{\|y\|_{\Omega}\}}$ \hspace{-0.1mm}for \hspace{-0.1mm}all \hspace{-0.1mm}$v^*_\sigma\hspace{-0.15em}\in\hspace{-0.15em} V^*_\sigma$, \hspace{-0.1mm}which, \hspace{-0.1mm}in \hspace{-0.1mm}turn, \hspace{-0.1mm}is \hspace{-0.1mm}based \hspace{-0.1mm}on \hspace{-0.1mm}that~\hspace{-0.1mm}${\|\cdot\|_{V_\sigma}\hspace{-0.15em}=\hspace{-0.15em}\|\nabla(\cdot)\|_{\Omega}}$~\hspace{-0.1mm}on~\hspace{-0.1mm}$V_\sigma$.

If the vector Laplace operator $-\Delta\colon V\to V^*$ is
defined by $\langle (-\Delta)v,w\rangle_V
    \coloneqq
    (\nabla v,\nabla w)_\Omega$ for all $v,w\in V$ and the (distributional) gradient operator $\nabla\colon L_0^2(\Omega)\to V^*$, where $L_0^2(\Omega)\coloneqq \{q\in L^2(\Omega)\mid (q,1)_{\Omega}=0\}$, is defined by $\langle\nabla q,v\rangle_V
    \coloneqq
    -(q,\operatorname{div}v)_\Omega$ for all $q\in L_0^2(\Omega)$ and $v\in V$,
    then,
given~an~initial~datum~$u_0\in H_{\sigma}$, the abstract
subgradient-flow problem \eqref{eq:subgradient_flow} is given via the  
unsteady Stokes equations: find
$(u,p)\in\mathcal{W}(I)\times L^2(I;L_0^2(\Omega))$ such that\vspace{-0.5mm}
\begin{subequations}\label{subsec:stokes_equations.5}
\begin{alignat}{3}\label{subsec:stokes_equations.5.1}
    \partial_tu(t)+(-\Delta)u(t)+\nabla p(t)
    &=f(t)
    &&\quad\text{ in }V^*&&\quad\text{ for a.e.\ }t\in I\,,
    \\\label{subsec:stokes_equations.5.2}
    \operatorname{div}u(t)
    &=0
    &&\quad\text{ a.e.\ in }\Omega&&\quad\text{ for a.e.\ }t\in I\,,
    \\\label{subsec:stokes_equations.5.3}
    u(0)
    &=u_0
    &&\quad\text{ in }H\,.\\[-6mm]\notag
\end{alignat}
\end{subequations} 
By the standard existence theory for the unsteady Stokes equations (\textit{cf}.\ \cite[Ch.\ III, §1, Thm.\ 1.1]{Temam1977}),
\eqref{subsec:stokes_equations.5} admits a unique velocity field
$u\in\mathcal{W}(I)$ and a unique kinematic pressure
$p\in L^2(I;L_0^2(\Omega))$.  
According to the Br\'ezis--Ekeland--Nayroles principle
(\textit{cf}.\ Proposition~\ref{prop:brezis_ekeland_nayroles}), the
velocity field $u\in \mathcal{W}(I)$ is equivalently characterized as the primal solution,
\textit{i.e.}, as a minimizer of the unsteady primal energy functional
$\mathcal{E}\colon\mathcal{W}(I)\to\mathbb{R}\cup\{+\infty\}$, for
every $v\in\mathcal{W}(I)$ defined by\vspace{-0.5mm}
\begin{align}\label{subsec:stokes_equations.6}
\begin{aligned}
    \mathcal{E}(v)
    &\coloneqq
    \tfrac{1}{2}\|\nabla v\|_Q^2
    +
    \chi_{\{0\}}(\operatorname{tr}\nabla v)
    -
    \langle f,v\rangle_{L^2(I;V)}
    \\
    &\quad+ 
    \tfrac{1}{2}\|(f-\partial_t v)|_{L^2(I;V_\sigma)}\|_{L^2(I;V_\sigma^*)}^2 
    +
    \tfrac{1}{2}\|v(t_{\mathtt{fin}})\|_H^2
    +
    \chi_{\{u_0\}}(v(0))\,.
\end{aligned}\\[-6mm]\notag
\end{align}
According to Theorem~\ref{thm:duality}(\hyperlink{thm:duality.i}{i}),
the unsteady dual energy functional
$\mathcal{D}\colon L^2(I;Y^*)\times\mathcal{W}(I)
\to\mathbb{R}\cup\{-\infty\}$, for every
$(y,\lambda)\in L^2(I;Y^*)\times\mathcal{W}(I)$, is given via\vspace{-0.5mm}
\begin{align}\label{subsec:stokes_equations.7}
\begin{aligned}
    \mathcal{D}(y,\lambda)
    \coloneqq
    &-
    \tfrac{1}{2}\|\operatorname{dev}y\|_Q^2
    -
    \chi_{\{-f\}}(-\partial_t\lambda-L^*y)
    -
    \tfrac{1}{2}\|\nabla\lambda\|_Q^2
    \\
    &-
    \chi_{\{0\}}(\operatorname{tr}\nabla\lambda)
    +
    \langle f,\lambda\rangle_{L^2(I;V)}
    -
    \tfrac{1}{2}\|\lambda(t_{\mathtt{fin}})\|_H^2
    +
    (\lambda(0),u_0)_H\,.
\end{aligned}\\[-6mm]\notag
\end{align}
The continuity condition in
Theorem~\ref{thm:duality}(\hyperlink{thm:duality.ii}{ii}) is not
satisfied in the present setting, since the effective domain of the functional
$I_G\coloneqq\frac{1}{2}\|\cdot\|_{Q}^2+\chi_{\{0\}}(\operatorname{tr}(\cdot))\colon L^2(I;Y)\to\mathbb{R}\cup\{+\infty\}$ is the proper closed
subspace $ 
    \operatorname{dom}(I_G)\hspace{-0.1em}=\hspace{-0.1em}\{
        y\hspace{-0.1em}\in \hspace{-0.1em}L^2(I;Y)
        \mid
        \operatorname{tr}y\hspace{-0.1em}=\hspace{-0.1em}0\text{ a.e.\ in }Q
    \}$.
Nevertheless, a dual solution $(z,\mu)\hspace{-0.1em}\in \hspace{-0.1em}L^2(I;Y^*)\times \mathcal{W}(I)$~can\linebreak be constructed directly from the
velocity  field and the kinematic pressure.  More precisely,~if~we~define\vspace{-0.5mm}
\begin{align}\label{subsec:stokes_equations.8}
    (z,\mu)
    \coloneqq
    (\nabla u-p\mathrm{I}_d,u)
    \quad\text{ a.e.\ in }Q\,,\\[-6mm]\notag
\end{align} 
from $\operatorname{dev}\nabla u=\nabla u$  a.e.\ in $Q$ (since $\operatorname{tr}\nabla u=\operatorname{div}u=0$ a.e.\ in $Q$),  $\operatorname{dev}\mathrm{I}_d=0$, and
\eqref{subsec:stokes_equations.5}, it follows that\vspace{-0.5mm}
\begin{subequations}\label{subsec:stokes_equations.opt}
\begin{alignat}{2}\label{subsec:stokes_equations.opt.1}
    \operatorname{dev}z
    &=\nabla u
    &&\quad\text{a.e.\ in }Q\,,
    \\\label{subsec:stokes_equations.opt.2}
    \partial_tu+L^*z
    &=f
    &&\quad\text{in }L^2(I;V^*)\,.\\[-6mm]\notag
\end{alignat}
\end{subequations} 
Therefore, 
inserting \eqref{subsec:stokes_equations.8} into \eqref{subsec:stokes_equations.7}, using the optimality relations \eqref{subsec:stokes_equations.opt} and the integration-by-parts
formula in time \eqref{subsubsec:function_spaces_unsteady.3},~we~find~that\vspace{-0.5mm}
\begin{align*}
    \mathcal{D}(z,\mu)
=
\tfrac{1}{2}\|u_0\|_H^2
=
\mathcal{E}(u)\,,\\[-6mm]\notag
\end{align*} 
so that, by weak duality, \eqref{subsec:stokes_equations.8} is a maximizer of  \eqref{subsec:stokes_equations.7} and the
strong duality~relation~holds.

\hspace{-1mm}Next, \hspace{-0.1mm}we \hspace{-0.1mm}record \hspace{-0.1mm}the \hspace{-0.1mm}corresponding \hspace{-0.1mm}primal \hspace{-0.1mm}and \hspace{-0.1mm}dual \hspace{-0.1mm}gap \hspace{-0.1mm}identities
\hspace{-0.1mm}(\textit{cf}.\ \hspace{-0.1mm}Theorems~\ref{thm:primal_gap_identity} \hspace{-0.1mm}and
\hspace{-0.1mm}\ref{thm:dual_gap_identity}):~\hspace{-0.1mm}To~\hspace{-0.1mm}this~\hspace{-0.1mm}end, we note that, for every
$v,w\in\mathcal{W}(I)$ with
$w\in L^2(I;V_{\sigma})$ and every
$y\in L^2(I;Y^*)$ with
$L^*y=f-\partial_tv$ in $L^2(I;V^*)$, using $I_{G}=\frac{1}{2}\|\cdot\|_{Q}^2+\chi_{\{0\}}(\operatorname{tr}(\cdot))$, $I_{G^*}=I_{\smash{I_{\phi^*}^{\Omega}}}$ with 
\eqref{subsec:stokes_equations.3}, $I_{F}=-\langle f,\cdot\rangle_{L^2(I;V)}$, $I_{F^*}=\chi_{\{-f\}}$, and a binomial formula, there holds\vspace{-0.5mm}
\begin{align}\label{subsec:stokes_equations.9}
\begin{aligned}
    &I_{G^*}(y)
    -
    \langle y,\nabla w\rangle_{L^2(I;Y)}
    +
    I_G(\nabla w)
    \\
    &\quad+
    \chi_{\{-f\}}(-\partial_tv-L^*y)
    -
    \langle-\partial_tv-L^*y,w\rangle_{L^2(I;V)}
    -
    \langle f,w\rangle_{L^2(I;V)}
    \\
    &=
    \tfrac{1}{2}\|\operatorname{dev}y-\nabla w\|_Q^2\,.
\end{aligned}\\[-6mm]\notag
\end{align}

First, using repeatedly the elementary identity \eqref{subsec:stokes_equations.9}, we derive the corresponding primal~gap~\mbox{identity}.

\begin{lemma}[Primal gap identity for the unsteady Stokes equations]
\label{lem:primal_gap_identity_stokes_equations}
For every $v\in\mathcal{W}(I)$ with $v(0)=u_0$ in $H$ and
$v\in L^2(I;V_{\sigma})$, there holds\vspace{-0.5mm}
\begin{align}\label{lem:primal_gap_identity_stokes_equations.0}
\begin{aligned}
    &\tfrac{1}{2}\|\nabla(v-u)\|_Q^2
    +
    \tfrac{1}{2}\|\partial_t(v-u)|_{L^2(I;V_\sigma)}\|_{L^2(I;V_\sigma^*)}^2
    +
    \tfrac{1}{2}\|(v-u)(t_{\mathtt{fin}})\|_H^2
    \\&=
    \inf_{\substack{
        y\in L^2(I;Y^*)\\
        L^*y=f-\partial_tv\text{ in }L^2(I;V^*)
    }}
    \bigl\{
        \tfrac{1}{2}\|\operatorname{dev}y-\nabla v\|_Q^2
    \bigr\} 
    \\&=\tfrac{1}{2}\|(\partial_t v+(-\Delta)v-f)|_{L^2(I;V_\sigma)}\|_{L^2(I;V_\sigma^*)}^2\,.
\end{aligned}\\[-6mm]\notag
\end{align}
\end{lemma}

\begin{proof}

Let $v\in\mathcal{W}(I)$ with $v(0)=u_0$ in $H$ and
$v\in L^2(I;V_{\sigma})$ be fixed, but arbitrary.

\emph{$\bullet$ Optimal strong convexity measure.}
Due to the alternative Bregman divergence representation
\eqref{rem:alternative_representation_bregman_divergence.4} (which is
applicable since $\mu =u$ in $\mathcal{W}(I)$), 
\eqref{subsec:stokes_equations.9} (applied with
$(v,w,y)=(u,v,z)$), and
the optimality conditions \eqref{subsec:stokes_equations.opt}, we have that\vspace{-0.5mm}
\begin{align*}
\begin{aligned}
    \mathcal{D}_{I_E}^{-\partial_tu}(v,u)
    &=
    \tfrac{1}{2}\|\operatorname{dev}z-\nabla v\|_Q^2
    \\
    &=
    \tfrac{1}{2}\|\nabla(v-u)\|_Q^2\,.
\end{aligned}\\[-6mm]\notag
\end{align*}
Moreover, due to the alternative Bregman divergence representation
\eqref{rem:alternative_representation_bregman_divergence.6} (which is applicable since $I_G\circ \nabla$ is proper and $I_F$ is continuous),
\eqref{subsec:stokes_equations.9} (applied with
$(v,w,y)=(v,u,y)$), and \eqref{subsec:stokes_equations.4},~we~have~that\vspace{-0.5mm}
\begin{align*}
\begin{aligned}
    \mathcal{D}_{I_{E^*}}^u
    (-\partial_tv,-\partial_tu)
    &=
    \inf_{\substack{
        y\in L^2(I;Y^*)\\
        L^*y=f-\partial_tv\text{ in }L^2(I;V^*)
    }}
    \bigl\{
        \tfrac{1}{2}\|\operatorname{dev}y-\nabla u\|_Q^2
    \bigr\}
    \\&=\tfrac{1}{2}\|\partial_t(v-u)|_{L^2(I;V_\sigma)}\|_{L^2(I;V_\sigma^*)}^2\,.
\end{aligned}\\[-6mm]\notag
\end{align*}
Therefore, the Bregman-type representation of the primal error measure
\eqref{lem:bregman_representation_primal_error_measure.0} yields
\begin{align}\label{lem:primal_gap_identity_stokes_equations.3}
\begin{aligned}
    \rho_{\mathcal{E}}^2(v)
    =
    \tfrac{1}{2}\|\nabla(v-u)\|_Q^2
    +
    \tfrac{1}{2}\|\partial_t(v-u)|_{L^2(I;V_\sigma)}\|_{L^2(I;V_\sigma^*)}^2
    +
    \tfrac{1}{2}\|(v-u)(t_{\mathtt{fin}})\|_H^2\,.
\end{aligned}
\end{align}

\emph{$\bullet$ Primal gap estimator.}
Using the representation
\eqref{lem:representation_of_primal_gap_estimator.0} of the primal gap 
estimator \eqref{def:primal_gap_estimator} (which is applicable since $I_G\circ \nabla$ is proper and $I_F$ is continuous), \eqref{subsec:stokes_equations.9} (applied with
$(v,w,y)=(v,v,y)$), and \eqref{subsec:stokes_equations.4}, we find that
\begin{align}\label{lem:primal_gap_identity_stokes_equations.4}
    \begin{aligned} 
    \eta_{\mathcal{E}}^2(v)
    &=
    \inf_{\substack{
        y\in L^2(I;Y^*)\\
        L^*y=f-\partial_tv\text{ in }L^2(I;V^*)
    }}
    \bigl\{
        \tfrac{1}{2}\|\operatorname{dev}y-\nabla v\|_Q^2
    \bigr\} 
    \\&=\tfrac{1}{2}\|(\partial_t v+(-\Delta)v-f)|_{L^2(I;V_\sigma)}\|_{L^2(I;V_\sigma^*)}^2\,.
    \end{aligned}
\end{align}

Finally, using the representations
\eqref{lem:primal_gap_identity_stokes_equations.3} and
\eqref{lem:primal_gap_identity_stokes_equations.4} and the general
primal gap identity \eqref{thm:primal_gap_identity.0}, we arrive at
the claimed representation \eqref{lem:primal_gap_identity_stokes_equations.0} of the primal gap identity for the unsteady Stokes~equations \eqref{subsec:stokes_equations.5}.
\end{proof}

Next, we derive the corresponding dual~gap~identity.

\begin{lemma}[Dual gap identity for the unsteady Stokes equations]
\label{lem:dual_gap_identity_stokes_equations}
For every $(y,\lambda)\in L^2(I;Y^*)\times\mathcal{W}(I)$ with
$\lambda\in L^2(I;V_{\sigma})$ and
\begin{align}\label{lem:dual_gap_identity_stokes_equations.-1}
    \operatorname{div}y
    =
    \partial_t\lambda-f
    \quad\text{in }L^2(I;V^*)\,,
\end{align}
there holds
\begin{align}\label{lem:dual_gap_identity_stokes_equations.0}
\begin{aligned}
    \tfrac{1}{2}\|\operatorname{dev}(y-z)\|_Q^2
    +
    \tfrac{1}{2}\|\nabla(\lambda-u)\|_Q^2
    +
    \tfrac{1}{2}\|(\lambda-u)(t_{\mathtt{fin}})\|_H^2
    =
    \tfrac{1}{2}\|\operatorname{dev}y-\nabla\lambda\|_Q^2
    +
    \tfrac{1}{2}\|\lambda(0)-u_0\|_H^2\,.
\end{aligned}
\end{align}
\end{lemma}

\begin{proof}
Let $(y,\lambda)\in L^2(I;Y^*)\times\mathcal{W}(I)$ with
$\lambda\in L^2(I;V_{\sigma})$ and
\eqref{lem:dual_gap_identity_stokes_equations.-1} be fixed, but
arbitrary.

\emph{$\bullet$ Optimal strong convexity measure.}
Due to the alternative Bregman divergence representation
\eqref{rem:alternative_representation_bregman_divergence.8}, \eqref{subsec:stokes_equations.3}, a binomial formula, and \eqref{subsec:stokes_equations.opt.1}, we have that
\begin{align*}
    \mathcal{D}_{I_{G^*}}^{\nabla u}(y,z)
    &=
    \tfrac{1}{2}\|\operatorname{dev}y-\nabla u\|_Q^2
    \\&= \tfrac{1}{2}\|\operatorname{dev}(y-z)\|_Q^2\,,
\end{align*}
due to the alternative Bregman divergence representation
\eqref{rem:alternative_representation_bregman_divergence.10},  $I_F= -\langle f,\cdot\rangle_{L^2(I;V)}$, $I_{F^*}=\chi_{\{-f\}}$, \eqref{lem:dual_gap_identity_stokes_equations.-1}, and \eqref{subsec:stokes_equations.opt.2}, we have that
\begin{align*}
    \mathcal{D}_{I_{F^*}}^u
    (-\partial_t\lambda-L^*y,-\partial_tu-L^*z)
    &=\chi_{\{-f\}}(-\partial_t\lambda-L^*y)-\langle f-\partial_t\lambda-L^*y,u\rangle_{L^2(I;V)}
    \\&=
    0\,, 
\end{align*}
and, due to the alternative Bregman divergence representation
\eqref{rem:alternative_representation_bregman_divergence.4} (which is
applicable since $\mu=u$ in $\mathcal{W}(I)$),
\eqref{subsec:stokes_equations.9} (applied with
$(v,w,y)=(u,\lambda,z)$), and
\eqref{subsec:stokes_equations.opt.1}, we have that
\begin{align*}
\begin{aligned}
    \mathcal{D}_{I_E}^{-\partial_tu}(\lambda,u)
     &=
    \tfrac{1}{2}\|\operatorname{dev}z-\nabla \lambda \|_Q^2
    \\&=
    \tfrac{1}{2}\|\nabla(\lambda-u)\|_Q^2\,.
\end{aligned}
\end{align*}
Therefore, the Bregman-type representation of the dual error measure
\eqref{lem:bregman_representation_dual_error_measure.0} yields
\begin{align}\label{lem:dual_gap_identity_stokes_equations.4}
\begin{aligned}
    \rho_{-\mathcal{D}}^2(y,\lambda)
    =
    \tfrac{1}{2}\|\operatorname{dev}y-\nabla u\|_Q^2
    +
    \tfrac{1}{2}\|\nabla(\lambda-u)\|_Q^2 +
    \tfrac{1}{2}\|(\lambda-u)(t_{\mathtt{fin}})\|_H^2\,.
\end{aligned}
\end{align}

\emph{$\bullet$ Dual gap estimator.}
Using the representation
\eqref{lem:representation_of_dual_gap_estimator.0} of the dual gap
estimator \eqref{def:dual_gap_estimator}, \eqref{subsec:stokes_equations.9} (applied with
$(v,w,y)=(\lambda,\lambda,y)$), and a binomial formula, we find that
\begin{align}\label{lem:dual_gap_identity_stokes_equations.5}
    \eta_{-\mathcal{D}}^2(y,\lambda)
    =
    \tfrac{1}{2}\|\operatorname{dev}y-\nabla\lambda\|_Q^2
    +
    \tfrac{1}{2}\|\lambda(0)-u_0\|_H^2\,.
\end{align}

Finally, using the representations
\eqref{lem:dual_gap_identity_stokes_equations.4} and
\eqref{lem:dual_gap_identity_stokes_equations.5} and the general dual
gap identity \eqref{thm:dual_gap_identity.0}, we arrive at
the claimed representation \eqref{lem:dual_gap_identity_stokes_equations.0} of the dual gap identity for the unsteady Stokes equations~\eqref{subsec:stokes_equations.5}.
\end{proof}\pagebreak

\subsection{The unsteady Navier--Lam\'e equations}\label{subsec:navier_lame_equations}

\hspace{5mm}In this subsection, we consider the \emph{unsteady Navier--Lam\'e equations}, which were first introduced by C.-L.~Navier in 1821 (\textit{cf}.\ \cite{Navier1821}) and later cast in their modern form by G.~Lam\'e in 1833 (\textit{cf}.\ \cite{Lame1833}) and which, in the present parabolic setting, describe the dissipative evolution of the displacement field in a homogeneous isotropic linearly elastic body under external loads, with the Lam\'e parameters characterizing its resistance to shear and volumetric deformation.

Let $\ell\hspace{-0.1em}=\hspace{-0.1em}d$ and $p=2$, \textit{i.e.},
$V\hspace{-0.1em}=\hspace{-0.1em}W^{1,2}_D(\Omega;\mathbb{R}^d)$,
$Y\hspace{-0.1em}=\hspace{-0.1em}L^2(\Omega;\mathbb{R}^{d\times d})$, and
$H\hspace{-0.1em}=\hspace{-0.1em}L^2(\Omega;\mathbb{R}^d)$, where~$V$~is~\mbox{equipped} with the gradient
norm $\|\cdot\|_V\coloneqq\|\nabla(\cdot)\|_\Omega$ in $V$.
For simplicity, we assume that $\Gamma_D=\partial\Omega$~and~define
\begin{align*}
    Y^*_{\mathrm{sym}}
    \coloneqq
    \smash{\bigl\{
        y\in Y^*
        \mid
        y=y^\top\text{ a.e.\ in }\Omega
    \bigr\}\,.}
\end{align*}
For $\mu_{\mathrm{L}}>0$ and
$\lambda_{\mathrm{L}}>-\frac{2}{d}\mu_{\mathrm{L}}$, let the elasticity tensor
$\mathbb{C}\colon\smash{\mathbb{R}_{\mathrm{sym}}^{d\times d}}
\to\smash{\mathbb{R}_{\mathrm{sym}}^{d\times d}}$, for every
$A\in\smash{\mathbb{R}_{\mathrm{sym}}^{d\times d}}$,~be~defined~by
\begin{align*}
    \smash{\mathbb{C}A
    \coloneqq
    2\mu_{\mathrm{L}}A
    +
    \lambda_{\mathrm{L}}\operatorname{tr}(A)\mathrm{I}_d\,,}
\end{align*}
Moreover, let the energy
densities
$\phi\colon\mathbb{R}^{d\times d}\to\mathbb{R}$ and
$\psi\colon Q\times\mathbb{R}^d\to\mathbb{R}$, for a.e.\
$(t,x)\in Q$, every $A\in\mathbb{R}^{d\times d}$ and
$a\in\mathbb{R}^d$, respectively, denoting by $\operatorname{sym}A
    \coloneqq \tfrac{1}{2}(A+A^\top)$ the symmetric part of $A$, be defined by
\begin{subequations}\label{subsec:navier_lame_equations.1}
\begin{align}\label{subsec:navier_lame_equations.1.1}
    \phi(A)
    &\coloneqq
    \tfrac{1}{2}
   \mathbb{C}\operatorname{sym}A:
   \operatorname{sym}A
   =\smash{\mu_{\mathrm{L}}\vert\operatorname{sym}A\vert^2
    +
    \tfrac{\lambda_{\mathrm{L}}}{2}\vert\operatorname{tr}A\vert^2}
   \,,
    \\\label{subsec:navier_lame_equations.1.2}
    \psi(t,x,a)
    &\coloneqq
    0\,,
\end{align}
\end{subequations}
so that $F\colon \hspace{-0.1em}I\times V\hspace{-0.1em}\to\hspace{-0.1em} \mathbb{R}$, for some $f\hspace{-0.1em}\in\hspace{-0.1em} L^2(I;V^*)$, is given via $F(t,v)\hspace{-0.1em}\coloneqq \hspace{-0.1em}-\langle f(t),v\rangle_V$ for a.e.\ $t\hspace{-0.1em}\in\hspace{-0.1em} I$~and~all~$v\hspace{-0.1em}\in\hspace{-0.1em} V$. 
For the choice
\eqref{subsec:navier_lame_equations.1},
Assumptions~\ref{ass:energy_densities},
\ref{ass:convex_conjugation}, and
\ref{ass:sufficient_for_conjugation}~are~\mbox{readily}~\mbox{verified}.

For a.e.\ $t\hspace{-0.1em}\in\hspace{-0.1em}I$, the steady primal
energy functional
$E(t,\cdot)\colon V\to\mathbb{R}\cup\{+\infty\}$, for every
$v\in V$,~is~given~via
\begin{align}\label{subsec:navier_lame_equations.3}
    \smash{E(t,v)
    =
    \tfrac{1}{2}
    \|
        \mathbb{C}^{\frac{1}{2}}
        \operatorname{sym}\nabla v
    \|_\Omega^2
    -
    \langle f(t),v\rangle_V\,.}
\end{align}
Due to \eqref{eq:E_prime_steady} and since 
the Fenchel conjugate functional 
$\phi^*\colon\mathbb{R}^{d\times d}
\to\mathbb{R}\cup\{+\infty\}$ of
\eqref{subsec:navier_lame_equations.1.1}, for every
$A^*\in\mathbb{R}^{d\times d}$, denoting by $\operatorname{skew} A^*\coloneqq A^*-\operatorname{sym}A^*$ the skew-symmetric part of $A^*$, is given via
\begin{align}\label{subsec:navier_lame_equations.4}
\begin{aligned}
    \phi^*(A^*)
    &=
    \tfrac{1}{2}
    \mathbb{C}^{-1}\operatorname{sym}A^*:\operatorname{sym}A^*
    +
    \chi_{\{0\}}(\vert\operatorname{skew}A^*\vert)
    \\
    &=
    \smash{\tfrac{1}{4\mu_{\mathrm{L}}}}
    \vert\operatorname{dev}\operatorname{sym}A^*\vert^2
    +
    \smash{\tfrac{1}{2d(2\mu_{\mathrm{L}}+d\lambda_{\mathrm{L}})}}
    \vert\operatorname{tr}\operatorname{sym}A^*\vert^2
    +
    \chi_{\{0\}}(\vert\operatorname{skew}A^*\vert)\,,
\end{aligned}
\end{align}
where $\chi_{\{0\}}\colon L^2(\Omega;\mathbb{R}^1)\to \mathbb{R}\cup\{+\infty\}$ is defined by \eqref{subsec:stokes_equations.1.3},
the Fenchel conjugate functional
$E^*(t,\cdot)\colon V^*\to\mathbb{R}\cup\{+\infty\}$ of
\eqref{subsec:navier_lame_equations.3}, for every $v^*\in V^*$, is
given via\vspace{-0.5mm}
\begin{align}\label{subsec:navier_lame_equations.5}
\begin{aligned}
    E^*(t,v^*)
    &=
    \inf_{\substack{
        y\in Y^*_{\mathrm{sym}}\\
        L^*y=f(t)+v^*\text{ in }V^*
    }}
    \bigl\{
        \tfrac{1}{2}
        \|
            \mathbb{C}^{-\frac{1}{2}}y
        \|_\Omega^2
    \bigr\}
    =
    \tfrac{1}{2}
    \|f(t)+v^*\|_{V_{\mathbb{C}}^*}^2\,,
\end{aligned}\\[-6mm]\notag
\end{align}
where $V_{\mathbb{C}}$ denotes the space $V$ equipped with the equivalent  elastic
norm $ \|\cdot\|_{V_{\mathbb{C}}}
    \coloneqq
    \|
        \smash{\mathbb{C}^{\frac{1}{2}}}
        \operatorname{sym}\nabla (\cdot)
    \|_\Omega$~on~$V_{\mathbb{C}}$.
Here, we used the standard minimal-stress representation of the
$V_{\mathbb{C}}^*$-norm induced by the elastic inner product, while
the equivalence of the elastic norm and the gradient norm follows from
Korn's inequality.

If the  Navier--Lam\'e operator
$-\operatorname{div}\mathbb{C}\operatorname{sym}\nabla\colon V\to V^*$ is defined by $\langle-\operatorname{div}\mathbb{C}\operatorname{sym}\nabla v,w\rangle_V
    \coloneqq
    (
        \mathbb{C}\operatorname{sym}\nabla v,
        \operatorname{sym}\nabla w
    )_\Omega$ for all $v,w\in V$,  
then, given an arbitrary initial datum $u_0\in H$, the abstract sub\-gradient-flow problem \eqref{eq:subgradient_flow} is given via the 
unsteady Navier--Lam\'e equations:~find~${u\in\mathcal{W}(I)}$~such~that\vspace{-4.5mm}
\begin{subequations}\label{subsec:navier_lame_equations.6}
\begin{alignat}{2}\label{subsec:navier_lame_equations.6.1}
    \partial_tu(t)+(-\operatorname{div}\mathbb{C}\operatorname{sym}\nabla)u(t)
    &=
    f(t)
    &&\quad\text{in }V^*\quad\text{for a.e.\ }t\in I\,,
    \\\label{subsec:navier_lame_equations.6.2}
    u(0)
    &=
    u_0
    &&\quad\text{in }H\,.
\end{alignat}
\end{subequations}
By the standard existence theory for linear parabolic problems in an
evolution triple (\textit{cf}.\
\cite[Thm.~23.A]{Zeidler1990IIA}),
\eqref{subsec:navier_lame_equations.6} admits a unique solution
$u\in\mathcal{W}(I)$. According to the
Br\'ezis--Ekeland--Nayroles principle
(\textit{cf}.\ Proposition~\ref{prop:brezis_ekeland_nayroles}), this
solution is equivalently characterized as the primal solution,
\textit{i.e.}, as a minimizer of the unsteady primal energy functional
$\mathcal{E}\colon\mathcal{W}(I)\to\mathbb{R}\cup\{+\infty\}$, for
every $v\in\mathcal{W}(I)$ defined by\vspace{-0.5mm}
\begin{align}\label{subsec:navier_lame_equations.8}
\begin{aligned}
    \mathcal{E}(v)
    &\coloneqq
    \tfrac{1}{2}
    \|
        \mathbb{C}^{\frac{1}{2}}
        \operatorname{sym}\nabla v
    \|_Q^2
    -
    \langle f,v\rangle_{L^2(I;V)}
    \\
    &\quad+
    \tfrac{1}{2}
    \|f-\partial_tv\|_{L^2(I;V_{\mathbb{C}}^*)}^2
    +
    \tfrac{1}{2}\|v(t_{\mathtt{fin}})\|_H^2
    +
    \chi_{\{u_0\}}(v(0))\,.
\end{aligned}
\end{align}
According to Theorem~\ref{thm:duality}(\hyperlink{thm:duality.i}{i}),
it is sufficient to consider the restricted unsteady dual energy
functional
$\mathcal{D}\colon L^2(I;Y^*)\times\mathcal{W}(I)
\to\mathbb{R}\cup\{-\infty\}$, for every
$(y,\lambda)\in L^2(I;Y^*)\times\mathcal{W}(I)$,
given via
\begin{align}\label{subsec:navier_lame_equations.9}
\begin{aligned}
    \mathcal{D}(y,\lambda)
    \coloneqq
    &-
    \tfrac{1}{2}
    \|
        \mathbb{C}^{-\frac{1}{2}}\operatorname{sym}y
    \|_Q^2
    -\chi_{\{0\}}(\vert\operatorname{skew}y\vert)-
    \chi_{\{-f\}}(-\partial_t\lambda-L^*y)
    \\
    &-
    \tfrac{1}{2}
    \|
        \mathbb{C}^{\frac{1}{2}}
        \operatorname{sym}\nabla\lambda
    \|_Q^2
    +
    \langle f,\lambda\rangle_{L^2(I;V)}
    -
    \tfrac{1}{2}\|\lambda(t_{\mathtt{fin}})\|_H^2
    +
    (\lambda(0),u_0)_H\,.
\end{aligned}
\end{align}
Since the continuity condition in
Theorem~\ref{thm:duality}(\hyperlink{thm:duality.ii}{ii}) is satisfied
in the present quadratic setting, there exists a dual solution
$(z,\mu)\in L^2(I;Y^*_{\mathrm{sym}})\times\mathcal{W}(I)$ and the corresponding
strong duality relation~\eqref{thm:duality.1}~applies. 
In particular, the 
optimality inclusions \eqref{thm:duality.2} together with the strict
convexity of the steady primal energy functional \eqref{subsec:navier_lame_equations.3} imply that the dual
solution is unique and given via 
\begin{subequations}\label{subsec:navier_lame_equations.opt}
\begin{alignat}{2}\label{subsec:navier_lame_equations.opt.1}
    z
    &=
    \mathbb{C}\operatorname{sym}\nabla u
    &&\quad\text{a.e.\ in }Q\,,
    \\\label{subsec:navier_lame_equations.opt.2}
    \partial_tu+L^*z
    &=
    f
    &&\quad\text{in }L^2(I;V^*)\,,
    \\\label{subsec:navier_lame_equations.opt.3}
    \mu
    &=
    u
    &&\quad\text{in }\mathcal{W}(I)\,.
\end{alignat}
\end{subequations}

\hspace{-1mm}Next, \hspace{-0.1mm}we \hspace{-0.1mm}record \hspace{-0.1mm}the \hspace{-0.1mm}corresponding \hspace{-0.1mm}primal \hspace{-0.1mm}and \hspace{-0.1mm}dual \hspace{-0.1mm}gap \hspace{-0.1mm}identities
\hspace{-0.1mm}(\textit{cf}.\ \hspace{-0.1mm}Theorems~\ref{thm:primal_gap_identity} \hspace{-0.1mm}and
\hspace{-0.1mm}\ref{thm:dual_gap_identity}).~\hspace{-0.1mm}To~\hspace{-0.1mm}this~\hspace{-0.1mm}end, we note that, for every
$v,w\in \mathcal{W}(I)$ and  
$y\in  L^2(I;Y^*_{\mathrm{sym}})$ with
$L^*y= f-\partial_tv$ in $L^2(I;V^*)$, using $I_{G}=I_{I_\phi^\Omega}$ with \eqref{subsec:navier_lame_equations.1.1}, $I_{G^*}\hspace{-0.15em}=\hspace{-0.15em}I_{I_{\phi^*}^\Omega}$ with
\eqref{subsec:navier_lame_equations.4}, $I_{F}\hspace{-0.15em}=\hspace{-0.15em}-\langle f,\cdot\rangle_{L^2(I;V)}$, $I_{F^*}\hspace{-0.15em}=\hspace{-0.15em}\chi_{\{-f\}}$,~and~a~binomial~formula,~there~holds
\begin{align}\label{subsec:navier_lame_equations.10}
\begin{aligned}
    &I_{G^*}(y)
    -
    \langle y,\nabla w\rangle_{L^2(I;Y)}
    +
    I_G(\nabla w)
    \\
    &\quad+
    \chi_{\{-f\}}(-\partial_tv-L^*y)
    -
    \langle-\partial_tv-L^*y,w\rangle_{L^2(I;V)}
    -
    \langle f,w\rangle_{L^2(I;V)}
    \\
    &=
    \tfrac{1}{2}
    \|
        \mathbb{C}^{-\frac{1}{2}}
        (
            y-\mathbb{C}\operatorname{sym}\nabla w
        )
    \|_Q^2\,.
\end{aligned}
\end{align}

First, using repeatedly the elementary identity \eqref{subsec:navier_lame_equations.10}, we derive the corresponding primal~gap~\mbox{identity}.

\begin{lemma}[Primal gap identity for the unsteady Navier--Lam\'e equations]
\label{lem:primal_gap_identity_navier_lame_equations}
For every $v\in\mathcal{W}(I)$ with $v(0)=u_0$ in $H$, there holds
\begin{align}\label{lem:primal_gap_identity_navier_lame_equations.0}
\begin{aligned}
    &\tfrac{1}{2}
    \|
        \mathbb{C}^{\frac{1}{2}}
        \operatorname{sym}\nabla(v-u)
    \|_Q^2
    +
    \tfrac{1}{2}
    \|\partial_t(v-u)\|_{L^2(I;V_{\mathbb{C}}^*)}^2
    +
    \tfrac{1}{2}\|(v-u)(t_{\mathtt{fin}})\|_H^2
    \\
    &=
    \inf_{\substack{
        y\in L^2(I;Y^*_{\mathrm{sym}})\\
        L^*y=f-\partial_tv\text{ in }L^2(I;V^*)
    }}
    \bigl\{
        \tfrac{1}{2}
        \|
            \mathbb{C}^{-\frac{1}{2}}
            (
                y-\mathbb{C}\operatorname{sym}\nabla v
            )
        \|_Q^2
    \bigr\}
    \\
    &=
    \tfrac{1}{2}
    \|\partial_tv+(-\operatorname{div}\mathbb{C}\operatorname{sym}\nabla)v-f
    \|_{L^2(I;V_{\mathbb{C}}^*)}^2\,.
\end{aligned}
\end{align}
\end{lemma}

\begin{proof}
Let $v\in\mathcal{W}(I)$ with $v(0)=u_0$ in $H$ be fixed, but
arbitrary.

\emph{$\bullet$ Optimal strong convexity measure.}
Due to the alternative Bregman divergence representation
\eqref{rem:alternative_representation_bregman_divergence.4} (which is
applicable since $\mu=u$ in $\mathcal{W}(I)$),
\eqref{subsec:navier_lame_equations.10} (applied with
$(v,w,y)=(u,v,z)$), and the optimality conditions
\eqref{subsec:navier_lame_equations.opt}, we have that
\begin{align*}
\begin{aligned}
    \mathcal{D}_{I_E}^{-\partial_tu}(v,u)
    &=
    \tfrac{1}{2}
    \|
        \mathbb{C}^{-\frac{1}{2}}
        (
            z-\mathbb{C}\operatorname{sym}\nabla v
        )
    \|_Q^2
    \\
    &=
    \tfrac{1}{2}
    \|
        \mathbb{C}^{\frac{1}{2}}
        \operatorname{sym}\nabla(v-u)
    \|_Q^2\,.
\end{aligned}
\end{align*}
Moreover, due to the alternative Bregman divergence representation
\eqref{rem:alternative_representation_bregman_divergence.6} (which is applicable since both $I_G\circ \nabla$ and $I_F$ are continuous),
\eqref{subsec:navier_lame_equations.10} (applied with
$(v,w,y)=(v,u,y)$), and \eqref{subsec:navier_lame_equations.5}, we have
that
\begin{align*}
\begin{aligned}
    \mathcal{D}_{I_{E^*}}^u
    (-\partial_tv,-\partial_tu)
    &=
    \inf_{\substack{
        y\in L^2(I;Y^*_{\mathrm{sym}})\\
        L^*y=f-\partial_tv\text{ in }L^2(I;V^*)
    }}
    \bigl\{
        \tfrac{1}{2}
        \|
            \mathbb{C}^{-\frac{1}{2}}
            (
                y-\mathbb{C}\operatorname{sym}\nabla u
            )
        \|_Q^2
    \bigr\}
    \\
    &=
    \tfrac{1}{2}
    \|\partial_t(v-u)\|_{L^2(I;V_{\mathbb{C}}^*)}^2\,.
\end{aligned}
\end{align*}
Therefore, the Bregman-type representation of the primal error measure
\eqref{lem:bregman_representation_primal_error_measure.0} yields
\begin{align}\label{lem:primal_gap_identity_navier_lame_equations.3}
\begin{aligned}
    \rho_{\mathcal{E}}^2(v)
    &=
    \tfrac{1}{2}
    \|
        \mathbb{C}^{\frac{1}{2}}
        \operatorname{sym}\nabla(v-u)
    \|_Q^2
    +
    \tfrac{1}{2}
    \|\partial_t(v-u)\|_{L^2(I;V_{\mathbb{C}}^*)}^2
    +
    \tfrac{1}{2}\|(v-u)(t_{\mathtt{fin}})\|_H^2\,.
\end{aligned}
\end{align}

\emph{$\bullet$ Primal gap estimator.}
Using the representation
\eqref{lem:representation_of_primal_gap_estimator.0} of the primal gap
estimator \eqref{def:primal_gap_estimator} (which is applicable since both $I_G\circ \nabla$ and $I_F$ are continuous),
\eqref{subsec:navier_lame_equations.10} (applied with
$(v,w,y)=(v,v,y)$),~and~\eqref{subsec:navier_lame_equations.5}, we find
that
\begin{align}\label{lem:primal_gap_identity_navier_lame_equations.4}
\begin{aligned}
    \eta_{\mathcal{E}}^2(v)
    &=
    \inf_{\substack{
        y\in L^2(I;Y^*_{\mathrm{sym}})\\
        L^*y=f-\partial_tv\text{ in }L^2(I;V^*)
    }}
    \bigl\{
        \tfrac{1}{2}
        \|
            \mathbb{C}^{-\frac{1}{2}}
            (
                y-\mathbb{C}\operatorname{sym}\nabla v
            )
        \|_Q^2
    \bigr\}
    \\
    &=
    \tfrac{1}{2}
    \|\partial_tv+(-\operatorname{div}\mathbb{C}\operatorname{sym}\nabla)v-f
    \|_{L^2(I;V_{\mathbb{C}}^*)}^2\,.
\end{aligned}
\end{align}

Finally, using the representations
\eqref{lem:primal_gap_identity_navier_lame_equations.3} and
\eqref{lem:primal_gap_identity_navier_lame_equations.4} and the general
primal gap identity \eqref{thm:primal_gap_identity.0}, we arrive at the
claimed representation
\eqref{lem:primal_gap_identity_navier_lame_equations.0} of the primal
gap identity for the unsteady Navier--Lam\'e equations
\eqref{subsec:navier_lame_equations.6}.
\end{proof}

Next, we derive the corresponding dual~gap~identity.

\begin{lemma}[Dual gap identity for the unsteady Navier--Lam\'e equations]
\label{lem:dual_gap_identity_navier_lame_equations}
For every
$(y,\lambda)\in L^2(I;Y^*_{\mathrm{sym}})\times\mathcal{W}(I)$ with
\begin{align}\label{lem:dual_gap_identity_navier_lame_equations.-1}
    \operatorname{div}y
    =
    \partial_t\lambda-f
    \quad\text{in }L^2(I;V^*)\,,
\end{align}
there holds
\begin{align}\label{lem:dual_gap_identity_navier_lame_equations.0}
\begin{aligned}
    &\tfrac{1}{2}
    \|
        \mathbb{C}^{-\frac{1}{2}}(y-z)
    \|_Q^2
    +
    \tfrac{1}{2}
    \|
        \mathbb{C}^{\frac{1}{2}}
        \operatorname{sym}\nabla(\lambda-u)
    \|_Q^2
    +
    \tfrac{1}{2}\|(\lambda-u)(t_{\mathtt{fin}})\|_H^2
    \\
    &=
    \tfrac{1}{2}
    \|
        \mathbb{C}^{-\frac{1}{2}}
        (
            y-\mathbb{C}\operatorname{sym}\nabla\lambda
        )
    \|_Q^2
    +
    \tfrac{1}{2}\|\lambda(0)-u_0\|_H^2\,.
\end{aligned}
\end{align}
\end{lemma}

\begin{proof}
Let
$(y,\lambda)\in L^2(I;Y^*_{\mathrm{sym}})\times\mathcal{W}(I)$ with
\eqref{lem:dual_gap_identity_navier_lame_equations.-1} be fixed, but
arbitrary.

\emph{$\bullet$ Optimal strong convexity measure.}
Due to the alternative Bregman divergence representation
\eqref{rem:alternative_representation_bregman_divergence.8},
\eqref{subsec:navier_lame_equations.4}, a binomial formula, and \eqref{subsec:navier_lame_equations.opt.1}, we have
that
\begin{align*}
\begin{aligned}
    \mathcal{D}_{I_{G^*}}^{\nabla u}(y,z)
    &=\tfrac{1}{2}\|
        \mathbb{C}^{-\frac{1}{2}}(y-\mathbb{C}\operatorname{sym}\nabla u)
    \|_Q^2
    \\&=\tfrac{1}{2}
    \|
        \mathbb{C}^{-\frac{1}{2}}(y-z)
    \|_Q^2\,,
\end{aligned}
\end{align*}
due to the alternative Bregman divergence representation
\eqref{rem:alternative_representation_bregman_divergence.10}, $I_F= -\langle f,\cdot\rangle_{L^2(I;V)}$,
$I_{F^*}=\chi_{\{-f\}}$,
\eqref{lem:dual_gap_identity_navier_lame_equations.-1}, and
\eqref{subsec:navier_lame_equations.opt.2}, we have that
\begin{align*}
    \mathcal{D}_{I_{F^*}}^u
    (-\partial_t\lambda-L^*y,-\partial_tu-L^*z)
    &=\chi_{\{-f\}}(-\partial_t\lambda-L^*y)-\langle f-\partial_t\lambda-L^*y,u\rangle_{L^2(I;V)}
    \\&=
    0\,,
\end{align*}
and, due to the alternative Bregman divergence representation
\eqref{rem:alternative_representation_bregman_divergence.4} (which is
applicable since $\mu=u$ in $\mathcal{W}(I)$),
\eqref{subsec:navier_lame_equations.10} (applied with
$(v,w,y)=(u,\lambda,z)$), and
\eqref{subsec:navier_lame_equations.opt.1}, we have that
\begin{align*}
\begin{aligned}
    \mathcal{D}_{I_E}^{-\partial_tu}(\lambda,u)
    &=
    \tfrac{1}{2}
    \|
        \mathbb{C}^{-\frac{1}{2}}
        (
            z-\mathbb{C}\operatorname{sym}\nabla\lambda
        )
    \|_Q^2
    \\
    &=
    \tfrac{1}{2}
    \|
        \mathbb{C}^{\frac{1}{2}}
        \operatorname{sym}\nabla(\lambda-u)
    \|_Q^2\,.
\end{aligned}
\end{align*}
Therefore, the Bregman-type representation of the dual error measure
\eqref{lem:bregman_representation_dual_error_measure.0} yields
\begin{align}\label{lem:dual_gap_identity_navier_lame_equations.4}
\begin{aligned}
    \rho_{-\mathcal{D}}^2(y,\lambda)
    =
    \tfrac{1}{2}
    \|
        \mathbb{C}^{-\frac{1}{2}}(y-z)
    \|_Q^2
    +
    \tfrac{1}{2}
    \|
        \mathbb{C}^{\frac{1}{2}}
        \operatorname{sym}\nabla(\lambda-u)
    \|_Q^2
    +
    \tfrac{1}{2}\|(\lambda-u)(t_{\mathtt{fin}})\|_H^2\,.
\end{aligned}
\end{align}

\emph{$\bullet$ Dual gap estimator.}
Using the representation
\eqref{lem:representation_of_dual_gap_estimator.0} of the dual gap
estimator \eqref{def:dual_gap_estimator},
\eqref{subsec:navier_lame_equations.10} (applied with
$(v,w,y)=(\lambda,\lambda,y)$), and a binomial formula, we find that
\begin{align}\label{lem:dual_gap_identity_navier_lame_equations.5}
    \eta_{-\mathcal{D}}^2(y,\lambda)
    =
    \tfrac{1}{2}
    \|
        \mathbb{C}^{-\frac{1}{2}}
        (
            y-\mathbb{C}\operatorname{sym}\nabla\lambda
        )
    \|_Q^2
    +
    \tfrac{1}{2}\|\lambda(0)-u_0\|_H^2\,.
\end{align}

Finally, using the representations
\eqref{lem:dual_gap_identity_navier_lame_equations.4} and
\eqref{lem:dual_gap_identity_navier_lame_equations.5} and the general
dual gap identity \eqref{thm:dual_gap_identity.0}, we arrive at the
claimed representation
\eqref{lem:dual_gap_identity_navier_lame_equations.0} of the dual gap
identity for the unsteady Navier--Lam\'e equations
\eqref{subsec:navier_lame_equations.6}.
\end{proof}\pagebreak

\subsection{The unsteady Bingham flow through a pipe}\label{subsec:bingham_flow_pipe}\enlargethispage{1mm}

\hspace{5mm}In this subsection, we consider the unsteady flow of a Bingham fluid through a straight pipe~(\textit{cf}.~\cite{DuvautLions1972}), a classical model for viscoplastic fluids that remain unyielded below a critical shear stress and exhibit a viscous response once this threshold is exceeded. Assuming a fully developed unidirectional flow, the velocity has only an axial component, and the vector-valued incompressible flow problem reduces to a scalar evolution problem for the axial velocity on the cross-section $\Omega$ of the pipe.

Let $\ell=1$ and $p=2$, \textit{i.e.},
$V=W^{1,2}_D(\Omega;\mathbb{R}^1)$, $Y=L^2(\Omega;\mathbb{R}^d)$, and
$H=L^2(\Omega;\mathbb{R}^1)$, where $V$ is equipped with
the gradient norm $\|\cdot\|_V\coloneqq\|\nabla(\cdot)\|_\Omega$ in $V$.
Moreover, let $\nu,g>0$ and let the~energy~densities
$\phi\colon\mathbb{R}^d\to\mathbb{R}$ and
$\psi\colon Q\times\mathbb{R}\to\mathbb{R}$, for a.e.\
$(t,x)\in Q$, every $a\in\mathbb{R}^d$ and
$b\in\mathbb{R}$, respectively, be defined by 
\begin{subequations}\label{subsec:bingham_flow_pipe.1}
\begin{align}\label{subsec:bingham_flow_pipe.1.1}
    \phi(a)&\coloneqq \tfrac{\nu}{2}\vert a\vert^2+g\vert a\vert\,,\\\label{subsec:bingham_flow_pipe.1.2}
    \psi(t,x,b)&\coloneqq 0\,, 
\end{align}
\end{subequations}
so that $F\colon \hspace{-0.1em}I\times V\hspace{-0.1em}\to\hspace{-0.1em} \mathbb{R}$, for some $f\hspace{-0.1em}\in\hspace{-0.1em} L^2(I;V^*)$, is given via $F(t,v)\hspace{-0.1em}\coloneqq \hspace{-0.1em}-\langle f(t),v\rangle_V$ for a.e.\ $t\hspace{-0.1em}\in\hspace{-0.1em} I$~and~all~$v\hspace{-0.1em}\in\hspace{-0.1em} V$. 
For the choice \eqref{subsec:bingham_flow_pipe.1}, Assumptions~\ref{ass:energy_densities},
\ref{ass:convex_conjugation}, and
\ref{ass:sufficient_for_conjugation}~are~\mbox{readily}~\mbox{verified}.

For a.e.\ $t\hspace*{-0.1em}\in\hspace*{-0.1em} I$, the steady primal energy functional
$E(t,\cdot)\colon \hspace*{-0.1em}V\hspace*{-0.1em}\to\hspace*{-0.1em}\mathbb{R}\cup\{+\infty\}$,
for every $v\hspace*{-0.1em}\in\hspace*{-0.1em} V$,~is~given~via
\begin{align}\label{subsec:bingham_flow_pipe.2}
    E(t,v)
    =\tfrac{\nu}{2}\|\nabla v\|_\Omega^2+g\|\nabla v\|_{1,\Omega} 
    -
    \langle f(t),v\rangle_V\,.
\end{align}
Due to \eqref{eq:E_prime_steady} and since the Fenchel conjugate $\phi^*\colon \mathbb{R}^d\to \mathbb{R}$ of \eqref{subsec:bingham_flow_pipe.1.1}, for every $a^*\in \mathbb{R}^d$, is given via 
\begin{align}\label{subsec:bingham_flow_pipe.3}
    \phi^*(a^*)
   =
    \tfrac{1}{2\nu}(\vert a^*\vert-g)_+^2=\tfrac{1}{2\nu}\vert a^*-\Pi_g(a^*)\vert^2\,, 
\end{align}
where $\Pi_g\colon \mathbb{R}^d\to K_g^d(0)$ denotes the orthogonal projection onto $ K_g^d(0)$, 
the Fenchel conjugate functional 
$E^*(t,\cdot)\colon V^*\to\mathbb{R}\cup\{+\infty\}$ of \eqref{subsec:bingham_flow_pipe.2},
for every $v^*\in V^*$, is given via 
\begin{align}\label{subsec:bingham_flow_pipe.4}
\begin{aligned}
    E^*(t,v^*)
    =
    \inf_{\substack{y\in Y^*\\ L^*y=f(t)+v^*\text{ in }V^*}}
    \bigl\{
            \tfrac{1}{2\nu}\|(\vert y\vert-g)_+\|_{\Omega}^2
    \bigr\}
    =
    \inf_{\substack{y\in Y^*\\ L^*y=f(t)+v^*\text{ in }V^*}}
    \bigl\{
            \tfrac{1}{2\nu}\|y-\Pi_g(y)\|_{\Omega}^2
    \bigr\}\,.
    \end{aligned} 
\end{align}
 
Using the Laplace operator $-\Delta\colon V\to V^*$ introduced above, given an arbitrary initial datum $u_0\in H$,  the subgradient-flow problem \eqref{eq:subgradient_flow}   is given via the 
unsteady Bingham flow through~a~pipe:~find~$u\in\mathcal{W}(I)$ such that 
\begin{subequations}\label{subsec:bingham_flow_pipe.7}
\begin{alignat}{2}\label{subsec:bingham_flow_pipe.7.1}
    \partial_t u(t)+\nu(-\Delta) u+L^*(g\partial \|\vert \cdot\vert\|_{1,\Omega}(\nabla u))&\ni f(t)
    &&\quad\text{in }V^*\quad\text{for a.e.\ }t\in I\,,\\\label{subsec:bingham_flow_pipe.7.2}
    u(0)&=u_0
    &&\quad\text{in }H\,. 
\end{alignat}
\end{subequations}
If, in addition, $u_0\in V$, by the standard existence theory for Hilbert-space subgradient flows
induced by proper, convex, lower semi-continuous functionals (\textit{cf}.\ \cite[Cor.\ 4.1]{Barbu1984}), the 
unsteady Bingham flow \eqref{subsec:bingham_flow_pipe.7} admits a unique
solution $u\in\mathcal{W}(I)$. In what follows, without
imposing these additional regularity assumptions on $u_0\in H$ throughout,  we assume that there exists~a~solution~${u\in\mathcal{W}(I)}$.
According to the
Br\'ezis--Ekeland--Nayroles principle
(\textit{cf}.\ Proposition~\ref{prop:brezis_ekeland_nayroles}), this
solution is equivalently characterized as the primal solution, \textit{i.e.},
as a minimizer of the unsteady primal energy functional
$\mathcal{E}\colon\mathcal{W}(I)\to\mathbb{R}\cup\{+\infty\}$, for every
$v\in\mathcal{W}(I)$ defined by 
\begin{align}\label{subsec:bingham_flow_pipe.8}
    \begin{aligned} 
    \mathcal{E}(v)
    &\coloneqq\tfrac{\nu}{2}\|\nabla v\|_{Q}^2+g\|\nabla v\|_{1,Q}
    -
    \langle f,v\rangle_{L^2(I;V)}
    \\
    &\quad+
    \inf_{\substack{y\in L^2(I;Y^*)\\ L^*y=f-\partial_t v\text{ in }L^2(I;V^*)}}
    \bigl\{
        \tfrac{1}{2\nu}\|y-\Pi_g(y)\|_{Q}^2 
    \bigr\} +
    \tfrac{1}{2}\|v(t_{\mathtt{fin}})\|_H^2
    +
    \chi_{\{u_0\}}(v(0))\,.
    \end{aligned}
\end{align}
According to Theorem~\ref{thm:duality}(\hyperlink{thm:duality.i}{i}), the
unsteady dual energy functional
$\mathcal{D}\colon L^2(I;Y^*)\times\mathcal{W}(I)\to
\mathbb{R}\cup\{-\infty\}$, for every
$(y,\lambda)\in L^2(I;Y^*)\times\mathcal{W}(I)$, is given via
\begin{align}\label{subsec:bingham_flow_pipe.9}
    \begin{aligned} 
    \mathcal{D}(y,\lambda)
    \coloneqq
    &- 
        \tfrac{1}{2\nu}\|y-\Pi_g(y)\|_{Q}^2 
    -
    \chi_{\{-f\}}(-\partial_t\lambda-L^*y)
  \\
    &\quad - 
            \tfrac{\nu}{2}\|\nabla\lambda\|_{Q}^2
            -g\|\nabla\lambda\|_{1,Q}
    +
    \langle f,\lambda\rangle_{L^2(I;V)}
    -
    \tfrac{1}{2}\|\lambda(t_{\mathtt{fin}})\|_H^2
    +
    (\lambda(0),u_0)_H\,.
    \end{aligned}
\end{align} Since the continuity condition in
Theorem~\ref{thm:duality}(\hyperlink{thm:duality.ii}{ii}) is satisfied
in the present setting, there exists a dual solution
$(z,\mu)\in L^2(I;Y^*)\times\mathcal{W}(I)$ and the corresponding
strong duality relation~\eqref{thm:duality.1}~applies.\pagebreak 

\noindent In particular, the 
optimality inclusions \eqref{thm:duality.2}  together with the strict
convexity of the steady primal energy functional \eqref{subsec:bingham_flow_pipe.2} imply that the dual
solution is given via
\begin{subequations}\label{subsec:bingham_flow_pipe.opt}
\begin{alignat}{2}\label{subsec:bingham_flow_pipe.opt.1}
    z&\in\partial\phi(\nabla u)
    &&\quad\text{ a.e.\ in }Q\,,
    \\
    \label{subsec:bingham_flow_pipe.opt.2}
    \partial_tu+L^*z&=f
    &&\quad\text{ in }L^2(I;V^*)\,,\\
    \mu&=u&&\quad\text{ in }\mathcal{W}(I)\,.
\end{alignat}
\end{subequations}
Moreover, the optimality inclusion \eqref{thm:duality.2.1} is equivalent to
\begin{subequations}\label{subsec:bingham_flow_pipe.opt_explicit}
\begin{alignat}{2}\label{subsec:bingham_flow_pipe.opt_explicit.1}
    z-\Pi_g(z)&=\nu\nabla u
    &&\quad\text{a.e.\ in }Q\,, 
    \\
    \label{subsec:bingham_flow_pipe.opt_explicit.2}
    \Pi_g(z)\cdot\nabla u&=g\vert\nabla u\vert
    &&\quad\text{a.e.\ in }Q\,.
\end{alignat}
\end{subequations}

\hspace{-1mm}Next, \hspace{-0.1mm}we \hspace{-0.1mm}record \hspace{-0.1mm}the \hspace{-0.1mm}corresponding \hspace{-0.1mm}primal \hspace{-0.1mm}and \hspace{-0.1mm}dual \hspace{-0.1mm}gap \hspace{-0.1mm}identities
\hspace{-0.1mm}(\textit{cf}.\ \hspace{-0.1mm}Theorems~\ref{thm:primal_gap_identity} \hspace{-0.1mm}and
\hspace{-0.1mm}\ref{thm:dual_gap_identity}).~\hspace{-0.1mm}To~\hspace{-0.1mm}this~\hspace{-0.1mm}end, we note that for every $v,w\in \mathcal{W}(I)$ and $ y\in L^2(I;Y^*)$ with $L^*y=f-\partial_t v$~in~$L^2(I;V^*)$, due to $I_{G}\hspace{-0.15em}=\hspace{-0.15em}I_{\smash{I_{\phi}^{\Omega}}}$ with  \eqref{subsec:bingham_flow_pipe.1.1}, $I_{G^*}\hspace{-0.15em}=\hspace{-0.15em}I_{\smash{I_{\phi^*}^{\Omega}}}$ with \eqref{subsec:bingham_flow_pipe.3}, $I_F\hspace{-0.15em}=\hspace{-0.15em}-\langle f,\cdot\rangle_{L^2(I;V)}$, $I_{F^*}\hspace{-0.15em}=\hspace{-0.15em}\chi_{\{-f\}}$, and a binomial~formula,~there~holds
    \begin{align}\label{lem:primal_gap_identity_bingham_flow_pipe.0.5}
    \begin{aligned} 
        &I_{G^*}(y)-\langle y,\nabla w\rangle_{L^2(I;Y)}+I_{G}(\nabla w)
        \\&\quad+\chi_{\{-f\}}(-\partial_t v-L^*y)-\langle -\partial_t v-L^*y,w\rangle_{L^2(I;V)}-\langle f,w\rangle_{L^2(I;V)}
   \\
    &=
            \tfrac{1}{2\nu}\|y-\Pi_g(y)\|_{Q}^2-(y,\nabla w)_{Q}+\tfrac{\nu}{2}\|\nabla w\|^2_{Q}+g\|\nabla w\|_{1,Q}
    \\
    &=
            \tfrac{1}{2\nu}\| y-\Pi_g(y)-\nu \nabla w\|_{Q}^2+\bigl(g\|\nabla w\|_{1,Q}-(\Pi_g(y),\nabla w)_{Q}\bigr)\,.
    \end{aligned}
    \end{align}

First, using repeatedly the elementary identity \eqref{lem:primal_gap_identity_bingham_flow_pipe.0.5}, we derive the corresponding primal~gap~\mbox{identity}.
 
\begin{lemma}[Primal gap identity for the unsteady Bingham flow through a pipe]
\label{lem:primal_gap_identity_bingham_flow_pipe}
For every $v\in\mathcal{W}(I)$ with $v(0)=u_0$ in $H$, there holds
\begin{align}\label{lem:primal_gap_identity_bingham_flow_pipe.0}
\begin{aligned} 
    &\tfrac{\nu}{2}\|\nabla(v-u)\|_Q^2
    +\bigl(g\|\nabla v\|_{1,Q}-(\Pi_g(z),\nabla v)_{Q}\bigr)+
    \tfrac{1}{2}\|(v-u)(t_{\mathtt{fin}})\|_H^2 \\
    &\quad+
    \inf_{\substack{
        y\in L^2(I;Y^*)\\
        L^*y=f-\partial_t v\text{ in }L^2(I;V^*)
    }}
    \bigl\{
            \tfrac{1}{2\nu}\| y-\Pi_g(y)-\nu \nabla u\|_{Q}^2+\bigl(g\|\nabla u\|_{1,Q}-(\Pi_g(y),\nabla u)_{Q}\bigr)
    \bigr\} \\
    &=
   \inf_{\substack{
        y\in L^2(I;Y^*)\\
        L^*y=f-\partial_t v\text{ in }L^2(I;V^*)
    }}
    \bigl\{
            \tfrac{1}{2\nu}\| y-\Pi_g(y)-\nu \nabla v\|_{Q}^2+\bigl(g\|\nabla v\|_{1,Q}-(\Pi_g(y),\nabla v)_{Q}\bigr)
    \bigr\}\,.
    \end{aligned}
\end{align}
\end{lemma}

\begin{proof} 
    Let $v\in\mathcal{W}(I)$ with $v(0)=u_0$ in $H$ be fixed, but arbitrary.

\emph{$\bullet$ Optimal strong convexity measure.} Due to the alternative Bregman divergence representation
\eqref{rem:alternative_representation_bregman_divergence.4} (which is applicable since $\mu = u$), \eqref{lem:primal_gap_identity_bingham_flow_pipe.0.5} (applied with $(v,w,y)=(u,v,z)$), and the optimality equation \eqref{subsec:bingham_flow_pipe.opt} (or \eqref{subsec:bingham_flow_pipe.opt_explicit}, respectively), we have that
\begin{align*}
\begin{aligned} 
    \mathcal{D}_{I_E}^{-\partial_t u}(v,u) 
    &= \tfrac{1}{2\nu}\| z-\Pi_g(z)-\nu \nabla v\|_{Q}^2+\bigl(g\|\nabla v\|_{1,Q}-(\Pi_g(z),\nabla v)_{Q}\bigr)
    \\&=\tfrac{\nu}{2}\|\nabla (v-u)\|_{Q}^2+\bigl(g\|\nabla v\|_{1,Q}-(\Pi_g(z),\nabla v)_{Q}\bigr)\,.
\end{aligned}  
\end{align*}
Moreover, due to the alternative Bregman divergence representation
\eqref{rem:alternative_representation_bregman_divergence.6} (which is applicable since both $I_G\circ\nabla$ and $I_F$ are continuous) and \eqref{lem:primal_gap_identity_bingham_flow_pipe.0.5} (applied with $(v,w,y)=(v,u,y)$), we have that
\begin{align*}
    \mathcal{D}_{I_{E^*}}^{u}
    (-\partial_t v,-\partial_t u) 
    &=
    \inf_{\substack{
        y\in L^2(I;Y^*)\\
        L^*y=f-\partial_t v\text{ in }L^2(I;V^*)
    }}
    \bigl\{
            \tfrac{1}{2\nu}\| y-\Pi_g(y)-\nu \nabla u\|_{Q}^2\\[-5mm]&\qquad\qquad\qquad\qquad\qquad\quad+\smash{\bigl(g\|\nabla u\|_{1,Q}-(\Pi_g(y),\nabla u)_{Q}\bigr)}
    \bigr\}\,.
\end{align*}
Therefore, the Bregman-type representation of the primal error measure
\eqref{lem:bregman_representation_primal_error_measure.0} yields
\begin{align}\label{lem:primal_gap_identity_bingham_flow_pipe.2}
\begin{aligned} 
    \rho_{\mathcal{E}}^2(v)
    &=
    \tfrac{\nu}{2}\|\nabla(v-u)\|_Q^2
    +\bigl(g\|\nabla v\|_{1,Q}-(\Pi_g(z),\nabla v)_{Q}\bigr)+
    \tfrac{1}{2}\|(v-u)(t_{\mathtt{fin}})\|_H^2 \\
    &\quad+
    \inf_{\substack{
        y\in L^2(I;Y^*)\\
        L^*y=f-\partial_t v\text{ in }L^2(I;V^*)
    }}
    \bigl\{
            \tfrac{1}{2\nu}\| y-\Pi_g(y)-\nu \nabla u\|_{Q}^2+\bigl(g\|\nabla u\|_{1,Q}-(\Pi_g(y),\nabla u)_{Q}\bigr)
    \bigr\}\,.
    \end{aligned}
\end{align}

\emph{$\bullet$ Primal gap estimator.}
Using the representation
\eqref{lem:representation_of_primal_gap_estimator.0} of the primal gap
estimator \eqref{def:primal_gap_estimator} (which is applicable since both $I_G\circ\nabla$ and $I_F$ are continuous) and \eqref{lem:primal_gap_identity_bingham_flow_pipe.0.5} (applied with $(v,w,y)=(v,v,y)$),  we find that
\begin{align}\label{lem:primal_gap_identity_bingham_flow_pipe.3}
    \eta_{\mathcal{E}}^2(v)
    =
    \inf_{\substack{y\in L^2(I;Y^*)\\ L^*y=f-\partial_t v\text{ in }L^2(I;V^*)}}
    \bigl\{
            \tfrac{1}{2\nu}\| y-\Pi_g(y)-\nu \nabla v\|_{Q}^2+\bigl(g\|\nabla v\|_{1,Q}-(\Pi_g(y),\nabla v)_{Q}\bigr)
    \bigr\}\,.
\end{align}

Finally, using the representations
\eqref{lem:primal_gap_identity_bingham_flow_pipe.2} and
\eqref{lem:primal_gap_identity_bingham_flow_pipe.3} and the general primal
gap identity \eqref{thm:primal_gap_identity.0},~we~arrive at
the claimed representation \eqref{lem:primal_gap_identity_bingham_flow_pipe.0} of the primal gap identity for the unsteady Bingham flow through a pipe.
\end{proof}

Next, we derive the corresponding dual~gap~identity.\vspace{-0.5mm}

\begin{lemma}[Dual gap identity for the unsteady Bingham flow through a pipe]
\label{lem:dual_gap_identity_bingham_flow_pipe}
For every $(y,\lambda)\in L^2(I;Y^*)\times\mathcal{W}(I)$ with
\begin{align}\label{lem:dual_gap_identity_bingham_flow_pipe.-1}
    \operatorname{div}y=\partial_t\lambda-f
    \quad\text{in }L^2(I;V^*)\,,
\end{align}
there holds
\begin{align}\label{lem:dual_gap_identity_bingham_flow_pipe.0}
\begin{aligned}
    &\tfrac{1}{2\nu}\|(y-z)-(\Pi_g(y)-\Pi_g(z))\|^2_{Q}+\bigl(g\|\nabla u\|_{1,Q}-(\Pi_g(y),\nabla u)_{Q}\bigr)
    \\
    &\quad+
    \tfrac{\nu}{2}\|\nabla(\lambda-u)\|_Q^2
    +
    \bigl(g\|\nabla\lambda\|_{1,Q}-(\Pi_g(z),\nabla\lambda)_Q\bigr)
    +
    \tfrac{1}{2}\|(\lambda-u)(t_{\mathtt{fin}})\|_H^2
    \\
    &= 
        \tfrac{1}{2\nu}
        \|
        y-\Pi_g(y)-\nu \nabla\lambda
        \|_{Q}^2
        +
        \bigl(g\|\nabla\lambda\|_{1,Q}
        -
       (\Pi_g(y),\nabla\lambda)_{Q}\bigr)
    +
    \tfrac{1}{2}\|\lambda(0)-u_0\|_H^2\,.
\end{aligned}
\end{align}
\end{lemma}

\begin{proof}
Let $(y,\lambda)\in L^2(I;Y^*)\times\mathcal{W}(I)$ with
\eqref{lem:dual_gap_identity_bingham_flow_pipe.-1} be fixed, but arbitrary.

\emph{$\bullet$ Optimal strong convexity measure.}
Due to the alternative Bregman divergence representation
\eqref{rem:alternative_representation_bregman_divergence.8},
\eqref{subsec:bingham_flow_pipe.opt_explicit.1}, and a binomial formula, we have
that
\begin{align*} 
    \mathcal{D}_{I_{G^*}}^{\nabla u}(y,z)
    &=\tfrac{1}{2\nu}\|y-\Pi_g(y)-\nu \nabla u\|^2_{Q}+\bigl(g\|\nabla u\|_{1,Q}-(\Pi_g(y),\nabla u)_{Q}\bigr)
    \\&=
 \tfrac{1}{2\nu}\|(y-z)-(\Pi_g(y)-\Pi_g(z))\|^2_{Q}+\bigl(g\|\nabla u\|_{1,Q}-(\Pi_g(y),\nabla u)_{Q}\bigr)\,,
\end{align*}
due to the alternative Bregman divergence representation
\eqref{rem:alternative_representation_bregman_divergence.10}, $I_F= -\langle f,\cdot\rangle_{L^2(I;V)}$,
$I_{F^*}=\chi_{\{-f\}}$,
\eqref{lem:dual_gap_identity_bingham_flow_pipe.-1}, and
\eqref{subsec:bingham_flow_pipe.opt.2}, we have that
\begin{align*}
    \mathcal{D}_{I_{F^*}}^u
    (-\partial_t\lambda-L^*y,-\partial_tu-L^*z)
    &=\smash{\chi_{\{-f\}}}(-\partial_t\lambda-L^*y)-\langle f-\partial_t\lambda-L^*y,u\rangle_{L^2(I;V)}
    \\&=
    0\,.
\end{align*}
and, due to the alternative Bregman divergence representation
\eqref{rem:alternative_representation_bregman_divergence.4} (which is
applicable since $\mu=u$ in $\mathcal{W}(I)$),
\eqref{lem:primal_gap_identity_bingham_flow_pipe.0.5} (applied with
$(v,w,y)=(u,\lambda,z)$), and
\eqref{subsec:bingham_flow_pipe.opt_explicit.1}, we have that
\begin{align*}
\begin{aligned}
    \mathcal{D}_{I_E}^{-\partial_tu}(\lambda,u)
    &=
    \tfrac{1}{2\nu }\|z-\Pi_g(z)-\nu \nabla \lambda\|^2_{Q}+\bigl(g\|\nabla \lambda\|_{1,Q}-(\Pi_g(z),\nabla \lambda)_{Q}\bigr)
    \\
    &=
    \tfrac{\nu}{2}\|\nabla (\lambda-u)\|^2_{Q}+\bigl(g\|\nabla \lambda\|_{1,Q}-(\Pi_g(z),\nabla \lambda)_{Q}\bigr)\,.
\end{aligned}
\end{align*}
Therefore, the Bregman-type representation of the dual error measure
\eqref{lem:bregman_representation_dual_error_measure.0} yields
\begin{align}\label{lem:dual_gap_identity_bingham_flow_pipe.1}
\begin{aligned}
    \rho_{-\mathcal{D}}^2(y,\lambda)
    &=   \tfrac{1}{2\nu}\|(y-z)-(\Pi_g(y)-\Pi_g(z))\|^2_{Q}+\bigl(g\|\nabla u\|_{1,Q}-(\Pi_g(y),\nabla u)_{Q}\bigr)
    \\
    &\quad+
    \tfrac{\nu}{2}\|\nabla(\lambda-u)\|_Q^2
    +
    \bigl(g\|\nabla\lambda\|_{1,Q}-(\Pi_g(z),\nabla\lambda)_Q\bigr)
    +
    \tfrac{1}{2}\|(\lambda-u)(t_{\mathtt{fin}})\|_H^2\,.
\end{aligned}
\end{align}

\emph{$\bullet$ Dual gap estimator.}
Using the representation \eqref{lem:representation_of_dual_gap_estimator.0} of the dual gap estimator
\eqref{def:dual_gap_estimator} and \eqref{lem:primal_gap_identity_bingham_flow_pipe.0.5} (which is applicable due to \eqref{lem:dual_gap_identity_bingham_flow_pipe.-1}),  we find that
\begin{align}\label{lem:dual_gap_identity_bingham_flow_pipe.2}
    \eta_{-\mathcal{D}}^2(y,\lambda)
    =
      \tfrac{1}{2\nu}
        \|
        y-\Pi_g(y)-\nu \nabla\lambda
        \|_{Q}^2
        +
        \bigl(g\|\nabla\lambda\|_{1,Q}
        -
       (\Pi_g(y),\nabla\lambda)_{Q}\bigr)
    +
    \tfrac{1}{2}\|\lambda(0)-u_0\|_H^2\,.
\end{align}

Finally, using the representations
\eqref{lem:dual_gap_identity_bingham_flow_pipe.1} and
\eqref{lem:dual_gap_identity_bingham_flow_pipe.2} and the general dual gap
identity \eqref{thm:dual_gap_identity.0}, we arrive at the claimed representation
\eqref{lem:dual_gap_identity_bingham_flow_pipe.0} of the dual gap identity for the unsteady Bingham flow through a pipe.
\end{proof}\newpage

\subsection{The unsteady obstacle problem}\label{subsec:obstacle_problem}

\hspace{5mm}In this subsection, we consider the unsteady \emph{obstacle problem} (\textit{cf}.\ \cite{DuvautLions1972,KinderlehrerStampacchia1980}), the parabolic counterpart of the classical obstacle problem for an elastic membrane constrained by a rigid obstacle, which describes the evolution of a diffusion-type state subject to a unilateral pointwise constraint.

Let $\ell=1$ and $p=2$, \textit{i.e.},
$V=W^{1,2}_D(\Omega;\mathbb{R}^1)$, $Y=L^2(\Omega;\mathbb{R}^d)$, and
$H=L^2(\Omega;\mathbb{R}^1)$, where $V$ is equipped with~the~gradient norm $\|\cdot\|_V\coloneqq\|\nabla(\cdot)\|_\Omega$ in $V$. 
Moreover, let the energy densities
$\phi\colon\mathbb{R}^d\to\mathbb{R}$~and $\psi\colon Q\times \mathbb{R}\to \mathbb{R}\cup\{+\infty\}$, for a.e.\ $(t,x)\in Q$, every
$a\in\mathbb{R}^d$, and $b\in \mathbb{R}$, respectively, be defined by 
\begin{subequations}\label{subsec:obstacle_problem.2}
\begin{align}\label{subsec:obstacle_problem.2.1}
\phi(a)
&\coloneqq
\tfrac{1}{2}|a|^2\,,\\\label{subsec:obstacle_problem.2.2}
\psi(t,x,b)
&\coloneqq
\chi_{[0,+\infty)}(b-\zeta(t,x))\,, 
\end{align}
\end{subequations}
for some $\zeta\in L^2(I;V)\cap C^0(\overline{I};H)$, so that $F\colon I\times V\to \mathbb{R}\cup\{+\infty\}$, for some $f\in  L^2(I;V^*)$,~is~given~via $F(t,v)
\coloneqq
\chi_{K(t)}(v)
-
\langle f(t),v\rangle_V$ for a.e.\ $t\in I$ and all $v\in V$. Here, we interpret $\chi_K\coloneqq I_{\psi}^{\Omega}\colon I\times V\to \mathbb{R}\cup\{+\infty\}$ as spatial integral reduction (\textit{cf}.\ Lemma \ref{lem:integral_functionals_as_normal_integrands}), which, for a.e.\ $t\in I$ and every $v\in V$,~is~given~via\vspace{-0.5mm}
\begin{align*}
\chi_{K(t)}(v)
\coloneqq
\begin{cases}
    0&\text{ if }v\in K(t)\,,\\
    +\infty&\text{ else}\,,
\end{cases} \\[-6mm]\notag
\end{align*}
where\vspace{-0.5mm}
\begin{align*}
K(t)
\coloneqq \bigl\{v\in V\mid v\ge \zeta(t) \text{ a.e.\ in }\Omega\bigr\}\,. \\[-6mm]\notag
\end{align*} 
For the choice \eqref{subsec:obstacle_problem.2}, Assumption~\ref{ass:energy_densities}~is~readily~\mbox{verified}, while Assumption~\ref{ass:sufficient_for_conjugation} is violated. Therefore,
the present application illustrates that Assumption~\ref{ass:sufficient_for_conjugation}
is sufficient, but not necessary, for the identification of the dual
problem in Theorem~\ref{thm:duality}. In Theorem~\ref{thm:duality}, this assumption is used to identify
$\mathcal{F}^*(-L^*(y,\lambda))$, which in the present
application can instead be identified directly, as shown~in~Lemma~\ref{lem:convex_conjugation_obstacle}~below.

For a.e.\ $t\in I$, the steady primal energy functional
$E(t,\cdot)\colon V\to \mathbb{R}\cup\{+\infty\}$, for every~${v\in V}$,~is~given~via 
\begin{align}\label{subsec:obstacle_problem.5}
    \smash{E(t,v)
    =
    \tfrac{1}{2}\|\nabla v\|_{\Omega}^{2}
    +
    \chi_{K(t)}(v)
    -
    \langle f(t),v\rangle_V\,,}
\end{align}
and, due to \eqref{eq:E_prime_steady}, its (Fenchel) conjugate $E^*(t,\cdot)\colon V^*\to \mathbb{R}\cup\{+\infty\}$, for every
$v^*\in V^*$, is given via 
\begin{subequations}\label{subsec:obstacle_problem.6} 
\begin{align}\label{subsec:obstacle_problem.6.1} 
    E^*(t,v^*)
   & =
    \inf_{y\in Y^*}
    \bigl\{
        \tfrac{1}{2}\|y\|_{\Omega}^{2}
        +
        \chi_K^*(f(t)+v^*-L^*y)
    \bigr\}
    \\&= \inf_{\eta\in V^*}
    \bigl\{
        \tfrac{1}{2}\|f(t)+v^*-\eta\|_{V^*}^{2}
        +
        \chi_K^*(\eta)
    \bigr\}\,,\label{subsec:obstacle_problem.6.2} 
\end{align} 
\end{subequations}
where we used  $\smash{F^*(t,v^*)=\chi_{K(t)}^*(f(t)+v^*)}$ for a.e.\ $t\in I$ and all $v^*\in V^*$ for \eqref{subsec:obstacle_problem.6.1} and, subsequently, the substitution $\eta\coloneqq f(t)+v^*-L^*y$ in $V^*$ for a.e.\ $t\in I$, all $v^*\in V^*$, and all $y\in Y^*$ for \eqref{subsec:obstacle_problem.6.2}. 

Using the Laplace operator $-\Delta\hspace{-0.1em}\coloneqq \hspace{-0.1em}L^*\circ\nabla\hspace{-0.1em}\colon \hspace{-0.1em}V\hspace{-0.1em}\to\hspace{-0.1em} V^*$ introduced above and, for a.e.\ $t\hspace{-0.1em}\in\hspace{-0.1em} I$ and~every~$v\hspace{-0.1em}\in\hspace{-0.1em} V$,
denoting the normal cone to $K(t)$ by
\begin{align*}
    \smash{N_{K(t)}(v)
    \coloneqq
    \partial\chi_{K(t)}(v)
    =
    \bigl\{
       v^*\in V^*
        \mid
        \langle v^*,\varphi-v\rangle_V\leq 0
        \text{ for all }\varphi\in K(t)
    \bigr\}\,,}
\end{align*}
the abstract subgradient flow \eqref{eq:subgradient_flow}, given an arbitrary initial
datum $u_0\in K(0)$, is given by the unsteady obstacle problem: find $u\in \mathcal{W}(I)$ such that
\begin{subequations}\label{subsec:obstacle_problem.8}
\begin{alignat}{2}
    \partial_t u(t)+(-\Delta)u(t)+N_{K(t)}(u(t))
    &\ni f(t)
    &&\quad\text{in }V^*\quad\text{for a.e.\ }t\in I\,,
    \label{subsec:obstacle_problem.8.1}
    \\
    u(0)
    &=u_0
    &&\quad\text{in }H\,.
    \label{subsec:obstacle_problem.8.2}
\end{alignat}
\end{subequations}
If, in addition, $\partial_t\zeta,\Delta\zeta,f\in L^2(Q;\mathbb{R}^1)$, according to \cite[Lem.\ 3.1]{BoegeleinDuzaarMingione2011}, 
the  unsteady obstacle~\mbox{problem}~\eqref{subsec:obstacle_problem.8} admits a unique solution $u\in \mathcal{W}(I)$. Motivated by \cite[Lem.\ 3.1]{BoegeleinDuzaarMingione2011}, but without
imposing these additional regularity assumptions throughout, in what follows, we assume that there exists a solution $u\in\mathcal{W}(I)$.
 According to the
Br\'ezis--Ekeland--Nayroles principle
(\textit{cf}.\ Proposition~\ref{prop:brezis_ekeland_nayroles}), this solution
is equivalently characterized as the primal solution, \textit{i.e.}, as a
minimizer of the unsteady primal energy functional
$\mathcal{E}\colon\mathcal{W}(I)\to\mathbb{R}\cup\{+\infty\}$, for every
$v\in\mathcal{W}(I)$ defined by\vspace{-0.5mm}
\begin{align}\label{subsec:obstacle_problem.9}
    \mathcal{E}(v)
    =\tfrac{1}{2}\|\nabla v\|_{Q}^2+I_{\chi_K}(v)
    -\langle f,v\rangle_{L^2(I;V)}+I_{E^*}(-\partial_t v) 
    +
    \tfrac{1}{2}\|v(t_{\mathtt{fin}})\|_H^2
    +
    \chi_{\{u_0\}}(v(0))\,.
\end{align}
Here, 
due to \eqref{eq:E_prime_steady},  
for every $v^*\in L^2(I;V^*)$,~we~have~that
\begin{align*}
    I_{E^*}(v^*)
   & =
    \inf_{y\in L^2(I;Y^*)}
    \bigl\{
        \tfrac{1}{2}\|y\|_{Q}^{2}
        +I_{\chi_K}^*(f+v^*-L^*y)
    \bigr\}
    \\&= \inf_{\eta\in L^2(I;V^*)}
    \bigl\{
        \tfrac{1}{2}\|f+v^*-\eta\|_{L^2(I;V^*)}^{2}
        +
        I_{\chi_K}^*(\eta)
    \bigr\}\,.
\end{align*} 

The unsteady obstacle problem \eqref{subsec:obstacle_problem.8} (and \eqref{subsec:obstacle_problem.9}, respectively) is not covered by~\mbox{Assumption}~\ref{ass:sufficient_for_conjugation}, since the  functional
$I_{F}\coloneqq I_{\chi_K}-\langle f,\cdot\rangle_{L^2(I;V)}\colon L^2(I;V)\to \mathbb{R}\cup\{+\infty\}$ is neither~\mbox{continuous}~nor~bounded. In the proof of
Theorem~\ref{thm:duality}, this assumption is used to identify the conjugate
term $\mathcal{F}^*(-\mathcal{L}^*(y,\lambda))$, where the functional $\mathcal{F}\colon \mathcal{W}(I)\to \mathbb{R}\cup\{+\infty\}$, for every $v\in \mathcal{W}(I)$, is defined by
\begin{align}\label{subsec:obstacle_problem.11} 
    \mathcal{F}(v)\coloneqq I_{\chi_K}(v)-\langle f,v\rangle_{L^2(I;V)}+\tfrac{1}{2}\|v(t_{\mathtt{fin}})\|_H^2+\chi_{\{u_0\}}(v(0))\,.
\end{align}
For the unsteady obstacle problem \eqref{subsec:obstacle_problem.8}, however,  
this term can be identified similarly~if~$\lambda\in\mathcal{W}(I)$.

\begin{lemma}[Convex conjugation formula]
	\label{lem:convex_conjugation_obstacle}
	Assume, in addition, that $	\zeta\in L^2(I;V)\cap H^1(I;H)$ and $\zeta(t_{\mathtt{fin}})\in V$. 
	Then, for every
	$(y,\lambda)\in L^2(I;Y^*)\times\mathcal{W}(I)$, there holds
	\begin{align}\label{lem:convex_conjugation_obstacle.0}
		\begin{aligned}
			\mathcal{F}^*(-\mathcal{L}^*(y,\lambda))
			&=
			I_{\chi_K^*}(f-\partial_t\lambda-L^*y)
			\\&\quad+
			\tfrac{1}{2}\|\lambda(t_{\mathtt{fin}})\|_H^2
			-
			\tfrac{1}{2}
			\operatorname{dist}_H^2
			(
			\lambda(t_{\mathtt{fin}}),
			\operatorname{cl}_H K(t_{\mathtt{fin}}))
			-
			(\lambda(0),u_0)_H\,,
		\end{aligned}
	\end{align}
    where 
    \begin{align*}
        \operatorname{dist}_H^2
			(
			\lambda(t_{\mathtt{fin}}),
			\operatorname{cl}_H K(t_{\mathtt{fin}}))\coloneqq\min_{h\in \operatorname{cl}_H K(t_{\mathtt{fin}})}{\bigl\{\|\lambda(t_{\mathtt{fin}})-h\|_H^2\bigr\}}\,.
    \end{align*}
	Moreover, if $\lambda(t)\in  K(t)$ for a.e.\ $t\in  I$, then $	\operatorname{dist}_H(\lambda(t_{\mathtt{fin}}),\operatorname{cl}_H K(t_{\mathtt{fin}}))=0$. 
\end{lemma}

Once the conjugation formula of
Lemma~\ref{lem:convex_conjugation_obstacle} has been established,
Assumption~\ref{ass:sufficient_for_conjugation}~is~no~longer~needed
in the present application. Indeed, formula \eqref{lem:convex_conjugation_obstacle.0} yields the dual\
functional below in the form of~\eqref{eq:dual}, while the corresponding
optimality inclusions and strong duality relation are verified directly below.
As~a~result, the arguments underlying Lemma~\ref{lem:bregman_representation_dual_error_measure}, Theorem~\ref{thm:dual_gap_identity}, and
Lemma~\ref{lem:representation_of_dual_gap_estimator} carry over~to~the~present application, since their proofs only
use formula \eqref{lem:convex_conjugation_obstacle.0}, the corresponding optimality inclusions,
and strong duality. The primal gap identity of Theorem~\ref{thm:primal_gap_identity} does not
require Assumption~\ref{ass:sufficient_for_conjugation}~in~the~first~place.

\begin{proof}[Proof (of Lemma \ref{lem:convex_conjugation_obstacle}).]
	Let $y\in L^2(I;Y^*)$ and $\lambda\in \mathcal{W}(I)$  
	be fixed, but arbitrary. Then, denoting by $\widehat{u}_0\in L^2(I;V)\cap H^1(I;H)$ a trace lift of the initial datum $u_0\hspace{-0.1em}\in\hspace{-0.1em} K(0)$, \textit{i.e.}, there holds $\widehat{u}_0(0)\hspace{-0.1em}=\hspace{-0.1em}u_0$ in $H$, 
	using the integration-by-parts formula in time \eqref{subsubsec:function_spaces_unsteady.3}, we find that
	\begin{align}\label{lem:convex_conjugation_obstacle.1} 
		\begin{aligned}
			\mathcal{F}^*(-\mathcal{L}^*(y,\lambda)) &=\sup_{v\in\mathcal{W}(I)}
			\bigl\{
				-\langle y,\nabla v\rangle_{L^2(I;Y)}
				+
				\langle \partial_t v,\lambda\rangle_{L^2(I;V)}
				\\[-2.5mm]&\qquad\qquad\qquad-
				I_F(v)
				-
				\tfrac{1}{2}\|v(t_{\mathtt{fin}})\|_H^2-\chi_{\{u_0\}}(v(0))
				\bigr\}
			\\&=\sup_{\substack{v\in\mathcal{W}_0(I)\\ v(t)\in (K-\widehat{u}_0)(t)\text{ for a.e.\ }t\in I}}
			\bigl\{
			\langle f-L^*y-\partial_t \lambda,v+\widehat{u}_0\rangle_{L^2(I;V)}\\[-5mm]&\qquad\qquad\qquad\qquad\qquad\qquad-
			\tfrac{1}{2}\|(v+\widehat{u}_0)(t_{\mathtt{fin}})\|_H^2 
			\\&\qquad\qquad\qquad\qquad\qquad\qquad+(\lambda(t_{\mathtt{fin}}), (v+\widehat{u}_0)(t_{\mathtt{fin}}))_H-(\lambda(0),u_0)_H
			\bigr\}\,.
        \end{aligned}
	\end{align}
	Next, we verify a constrained version of the density result from
	Lemma~\ref{lem:dense_range}. Namely, we claim that
	\begin{align}\label{lem:convex_conjugation_obstacle.2}
		\begin{aligned}
			&\operatorname{cl}_{L^2(I;V)\times H}
			\big\{
			(v,v(t_{\mathtt{fin}}))
			\mid
			v\in\mathcal{W}_0(I),\
			v(t)\in(K-\widehat{u}_0)(t)
			\text{ for a.e.\ }t\in I
			\big\}
			\\
			&\qquad=
			\big\{
			v\in L^2(I;V)
			\mid
			v(t)\in(K-\widehat{u}_0)(t)
			\text{ for a.e.\ }t\in I
			\big\}
			\times
			\operatorname{cl}_H
			(K-\widehat{u}_0)(t_{\mathtt{fin}})\,.
		\end{aligned}		
	\end{align}

	$\bullet$ \emph{The inclusion ``$\subseteq$'' in \eqref{lem:convex_conjugation_obstacle.2}.} 
	The inclusion ``$\subseteq$'' in
	\eqref{lem:convex_conjugation_obstacle.2} is evident. In fact, if
	$v\in\mathcal{W}_0(I)$~satisfies $v(t)\in (K-\widehat{u}_0)(t)$ (\textit{i.e.},
	$v(t)\ge(\zeta-\widehat{u}_0)(t)$ a.e.\ in $\Omega$) for a.e.\ $t\in I$, then
	\begin{align*}
		v-(\zeta-\widehat{u}_0)
		\in\mathcal{W}(I)
		\hookrightarrow C^0(\overline I;H)\,.
	\end{align*}
	The closedness of the positive cone in $H$, therefore, implies that $v(t_{\mathtt{fin}})
	\in\operatorname{cl}_H(K-\widehat{u}_0)(t_{\mathtt{fin}})$.\pagebreak
	
	$\bullet$ \emph{The inclusion ``$\supseteq$'' in \eqref{lem:convex_conjugation_obstacle.2}.} Let $v\in L^2(I;V)$ with $v(t)\in (K-\widehat{u}_0)(t)$ for a.e.\ $t\in I$ and let $h\in \operatorname{cl}_H(K-\widehat{u}_0)(t_{\mathtt{fin}})$ be fixed, but arbitrary. 
	By Lemma~\ref{lem:dense_range}, there exists a sequence
	$\{v_n\}_{n\in\mathbb{N}}\subseteq\mathcal{W}_0(I)$ such that
	\begin{subequations} \label{lem:convex_conjugation_obstacle.4}
	\begin{alignat}{3}\label{lem:convex_conjugation_obstacle.4.1}
		v_n&\to v
		&&\quad\text{ in }L^2(I;V)&&\quad(n\to \infty)\,,
		\\\label{lem:convex_conjugation_obstacle.4.2}
		v_n(t_{\mathtt{fin}})&\to h
		&&\quad\text{ in }H&&\quad(n\to \infty)\,.
	\end{alignat}
	\end{subequations}
    In view of the construction in the proof of
	Lemma~\ref{lem:dense_range}, we may assume $v_n\in W^{1,\infty}(I;V)$ for all $n\in \mathbb{N}$. Then, if  we 
	define the sequence $\{\widetilde{v}_n\}_{n\in \mathbb{N}}
	\coloneqq
	\{\max\{v_n,\zeta-\widehat{u}_0\}\}_{n\in \mathbb{N}}
	=
	\{\zeta-\widehat{u}_0+
	(v_n-(\zeta-\widehat{u}_0))_+\}_{n\in \mathbb{N}}$,\linebreak by the assumption  $\zeta\in L^2(I;V)\cap H^1(I;H)$ and the Sobolev  chain rule for the positive-part~mapping, we have that $\widetilde{v}_n
	\in
	L^2(I;V)\cap H^1(I;H)
	\subseteq\mathcal{W}(I)$ for all $n\in \mathbb{N}$.
	Moreover, since $u_0\ge\zeta(0)$ a.e.\ in $\Omega$, we have that $\widetilde{v}_n(0)
	=
	\max\{0,\zeta(0)-u_0\}
	=
	0$ for all $n\in \mathbb{N}$, 
	so that $\widetilde{v}_n\in\mathcal{W}_0(I)$, and, by construction,~$\widetilde{v}_n(t)
	\in(K-\widehat{u}_0)(t)$ for a.e.\ $t\in I$. Since the positive-part mappings $(v\mapsto \max\{v,\zeta-\widehat{u}_0\})\colon L^2(I;V)\to L^2(I;V)$ and $(h\mapsto \max\{h,(\zeta-\widehat{u}_0)(t_{\mathtt{fin}})\})\colon H\to H$ are continuous, from \eqref{lem:convex_conjugation_obstacle.4.1} together with
	$v(t)\ge(\zeta-\widehat{u}_0)(t)$ a.e.\ in $\Omega$ for a.e.\ $t\in I$ (as $v(t)\in (K-\widehat{u}_0)(t)$ for a.e.\ $t\in I$) and \eqref{lem:convex_conjugation_obstacle.4.2} together with
	$h\ge(\zeta-\widehat{u}_0)(t_{\mathtt{fin}})$ a.e.\ in $\Omega$ (as $h\in \operatorname{cl}_H(K-\widehat{u}_0)(t_{\mathtt{fin}})$), we obtain
	\begin{alignat*}{3}
		\widetilde{v}_n=\max\{v_n,\zeta-\widehat{u}_0\}
	&	\to
		\max\{v,\zeta-\widehat{u}_0\}
		=
		v
	&&	\quad\text{in }L^2(I;V)&&\quad(n\to \infty)\,,\\
		\widetilde{v}_n(t_{\mathtt{fin}})
		=
		\max\{
		v_n(t_{\mathtt{fin}}),
		(\zeta-\widehat{u}_0)(t_{\mathtt{fin}})
		\}
	&	\to
		\max\{
		h,
		(\zeta-\widehat{u}_0)(t_{\mathtt{fin}})
		\}
		=
		h
		&&\quad\text{in }H&&\quad(n\to \infty)\,,
	\end{alignat*} 
	which proves \eqref{lem:convex_conjugation_obstacle.2}.
	
	Due to \eqref{lem:convex_conjugation_obstacle.2},
	from \eqref{lem:convex_conjugation_obstacle.1}, we infer that
	\begin{align*}
			\mathcal{F}^*(-\mathcal{L}^*(y,\lambda))&=\sup_{\substack{v\in L^2(I;V)\\ v(t)\in (K-\widehat{u}_0)(t)\text{ for a.e.\ }t\in I}}
			\bigl\{
			\langle f-L^*y-\partial_t \lambda,v+\widehat{u}_0\rangle_{L^2(I;V)}\bigr\}
			\\&\quad+\sup_{h\in \operatorname{cl}_H (K-\widehat{u}_0)(t_{\mathtt{fin}})}
			\bigl\{(\lambda(t_{\mathtt{fin}}), h+\widehat{u}_0(t_{\mathtt{fin}}))_H-
			\tfrac{1}{2}\|h+\widehat{u}_0(t_{\mathtt{fin}})\|_H^2-(\lambda(0),u_0)_H
			\bigr\}
			\\&=\sup_{v\in L^2(I;V)}
			\bigl\{
			\langle f-L^*y-\partial_t \lambda,v\rangle_{L^2(I;V)} 
			-
			I_{\chi_K}(v)\bigr\}
			\\&\quad+\sup_{h\in \operatorname{cl} K(t_{\mathtt{fin}})}
			\bigl\{(\lambda(t_{\mathtt{fin}}), h)_H-
			\tfrac{1}{2}\|h\|_H^2
			\bigr\}
			\\&=
			(I_{\chi_K})^*(f-L^*y-\partial_t \lambda)\\&\quad+
			\tfrac{1}{2}\|\lambda(t_{\mathtt{fin}})\|_H^2-\tfrac{1}{2}\operatorname{dist}_H^2(\lambda(t_{\mathtt{fin}}),\operatorname{cl}_HK(t_{\mathtt{fin}}))-(\lambda(0),u_0)_H\,,
        \hspace{-5mm}\\[-6mm]\notag
	\end{align*} 
	which is the claimed convex conjugation formula \eqref{lem:convex_conjugation_obstacle.0}.
	
	If $\lambda(t)\in K(t)$ for a.e.\ $t\in I$,  due to $\lambda\in \mathcal{W}(I)\subseteq C^0(\overline{I};H)$, we have that ${\lambda(t_{\mathtt{fin}})\in \operatorname{cl}_HK(t_{\mathtt{fin}})}$; hence, we conclude that ${\operatorname{dist}_H^2(\lambda(t_{\mathtt{fin}}),\operatorname{cl}_HK(t_{\mathtt{fin}}))=0}$.
\end{proof}

Consequently, although Assumption~\ref{ass:sufficient_for_conjugation} is not
satisfied for  the  functional
$I_{F}\colon L^2(I;V)\to \mathbb{R}\cup\{+\infty\}$, Lemma~\ref{lem:convex_conjugation_obstacle}
provides the required conjugation formula at least for all $(y,\lambda)\in L^2(I;Y^*)\times\mathcal{W}(I)$.~Hence, the corresponding (possibly restricted) unsteady dual energy functional
${\mathcal{D}\colon \hspace{-0.15em}L^2(I;Y^*)\hspace{-0.25em}\times\hspace{-0.25em}\mathcal{W}(I)\hspace{-0.15em}\to\hspace{-0.15em}
\mathbb{R}\hspace{-0.15em}\cup\hspace{-0.15em}\{-\infty\}}$, for every
$(y,\lambda)\in L^2(I;Y^*)\times\mathcal{W}(I)$, is given via\vspace{-0.5mm}
\begin{align}\label{subsec:obstacle_problem.12} 
    \begin{aligned} 
    \mathcal{D}(y,\lambda)
    &\coloneqq
    -\tfrac{1}{2}\|y\|_Q^2
    -
    I_{\chi_K}^*(f-\partial_t\lambda-L^*y)
    -
    \tfrac{1}{2}\|\nabla\lambda\|_Q^2
    -
    I_{\chi_K}(\lambda)
    +
    \langle f,\lambda\rangle_{L^2(I;V)}
    \\&\quad-
    \tfrac{1}{2}\|\lambda(t_{\mathtt{fin}})\|_H^2
    +
    (\lambda(0),u_0)_H\,,
     \end{aligned}\\[-6mm]\notag
\end{align} 
where, 
for every
$v^*\in L^2(I;V^*)$, we have that 
\begin{align*}
   I_{\chi_K^*}(v^*)
    =
    \begin{cases}
        \langle v^*,\zeta\rangle_{L^2(I;V)}
        &
        \text{ if }\langle v^*(t),v\rangle_{V}\leq 0\text{ for all }v\in V\text{ with }v\ge 0\text{ a.e.\ in }\Omega\text{ for a.e.\ }t\in I \,,
        \\
        +\infty
        &
        \text{ else}\,.
    \end{cases}
\end{align*}
Note that the term $-I_{\chi_K}(\lambda)$, which does not result from $\mathcal{F}^*(-\mathcal{L}^*(y,\lambda))$, enforces $\lambda(t)\hspace{-0.1em}\in\hspace{-0.1em} K(t)$~for~a.e.~${t\hspace{-0.1em}\in\hspace{-0.1em} I}$, so \hspace{-0.1mm}that, \hspace{-0.1mm}by \hspace{-0.1mm}Lemma~\hspace{-0.1mm}\ref{lem:convex_conjugation_obstacle}, \hspace{-0.1mm}for \hspace{-0.1mm}the \hspace{-0.1mm}terminal \hspace{-0.1mm}distance \hspace{-0.1mm}term \hspace{-0.1mm}in \hspace{-0.1mm}\eqref{lem:convex_conjugation_obstacle.0}, \hspace{-0.1mm}there~\hspace{-0.1mm}holds~\hspace{-0.1mm}${\operatorname{dist}_H^2(\lambda(t_{\mathtt{fin}}),\operatorname{cl}_H\!K(t_{\mathtt{fin}}))\hspace{-0.1em}=\hspace{-0.1em}0}$.\pagebreak

\noindent Moreover, if $u\in\mathcal{W}(I)$ solves the unsteady obstacle problem
\eqref{subsec:obstacle_problem.8}, then the corresponding dual solution
is given by $(z,\mu)=(\nabla u,u)$. More precisely, there exists a
Lagrange multiplier $\Lambda\in L^2(I;V^*)$ such that
\begin{subequations}\label{subsec:obstacle_problem.opt}
\begin{alignat}{2}
    z&=\nabla u
    &&\quad\text{a.e.\ in }Q\,,
    \label{subsec:obstacle_problem.opt.1}
    \\
    \Lambda
    \coloneqq f-\partial_tu-L^*z
    &\in N_{K(t)}(u(t))
    &&\quad\text{for a.e.\ }t\in I\,,
    \label{subsec:obstacle_problem.opt.2}
    \\
    \mu&=u
    &&\quad\text{in }\mathcal{W}(I)\,.
    \label{subsec:obstacle_problem.opt.3}
\end{alignat}
\end{subequations}Indeed, the normal-cone relation
\eqref{subsec:obstacle_problem.opt.2} is the equality condition in the
Fenchel--Young inequality for $I_{\chi_K}$
(\textit{cf}.\ \cite[Prop.\ 51.2]{Zeidler1985III}). Hence, inserting
\eqref{subsec:obstacle_problem.opt} into
\eqref{subsec:obstacle_problem.12} and using the integration-by-parts
formula in time \eqref{subsubsec:function_spaces_unsteady.3}, we obtain
\begin{align*}
    \mathcal{D}(z,\mu)
    &=
    \tfrac{1}{2}\|u_0\|_H^2
    =
    \mathcal{E}(u)\,,
\end{align*}
so that, by weak duality, $(z,\mu)=(\nabla u,u)$  is a maximizer of  \eqref{subsec:obstacle_problem.12}  and a
strong duality~relation~applies.

\hspace{-1mm}Next, \hspace{-0.1mm}we \hspace{-0.1mm}record \hspace{-0.1mm}the \hspace{-0.1mm}corresponding \hspace{-0.1mm}primal \hspace{-0.1mm}and \hspace{-0.1mm}dual \hspace{-0.1mm}gap \hspace{-0.1mm}identities
\hspace{-0.1mm}(\textit{cf}.\ \hspace{-0.1mm}Theorems~\ref{thm:primal_gap_identity} \hspace{-0.1mm}and
\hspace{-0.1mm}\ref{thm:dual_gap_identity}).~\hspace{-0.1mm}To~\hspace{-0.1mm}this~\hspace{-0.1mm}end, we note that for every $v,w\in\mathcal{W}(I)$ with
$w(t)\in K(t)$ for a.e.\ $t\in I$ and every
$y\in L^2(I;Y^*)$, due to $I_G=I_{G^*}=\tfrac{1}{2}\|\cdot\|_Q^2$, $I_F=I_{\chi_K}-\langle f,\cdot\rangle_{L^2(I;V)}$, $I_{F^*}
    =
    I_{\chi_K^*}(f+\cdot)$, and a binomial formula, 
there holds
\begin{align}\label{subsec:obstacle_problem.fenchel_gap}
\begin{aligned}
    &I_{G^*}(y)
    -
    \langle y,\nabla w\rangle_{L^2(I;Y)}
    +
    I_G(\nabla w)
    \\
    &\quad+
    I_{F^*}(-L^*y-\partial_t v)
    -
    \langle-L^*y-\partial_t v,w\rangle_{L^2(I;V)}
    +
    I_F(w)
    \\
    &=
    \tfrac{1}{2}\|y-\nabla w\|_Q^2
    +
    I_{\chi_K^*}(f-\partial_t v-L^*y)
    -
    \langle
        f-\partial_t v-L^*y,
        w
    \rangle_{L^2(I;V)}\,.
\end{aligned}
\end{align}

First, using repeatedly the elementary identity \eqref{subsec:obstacle_problem.fenchel_gap}, we derive the corresponding primal~gap~\mbox{identity}.

\begin{lemma}[Primal gap identity for the unsteady obstacle problem]\label{lem:primal_gap_identity_obstacle_problem}
For every 
$v\in \mathcal{W}(I)$ with $v(0)=u_0$ in $H$ and $v(t)\in K(t)$ for a.e.\ $t\in I$, there holds
\begin{align}\label{lem:primal_gap_identity_obstacle_problem.0}
\begin{aligned} 
   &\tfrac{1}{2}\|\nabla (v-u)\|_Q^2
    +
    \langle
        -\Lambda,
        v-u
    \rangle_{L^2(I;V)}+
    \tfrac{1}{2}\|(v-u)(t_{\mathtt{fin}})\|_H^2 \\
    &\quad+
    \inf_{ y\in L^2(I;Y^*)}
    \bigl\{
            \tfrac{1}{2}\|y-\nabla u\|_Q^2
    +
    I_{\chi_K^*}(f-\partial_t v-L^*y)
    -
    \langle
        f-\partial_t v-L^*y,
        u
    \rangle_{L^2(I;V)}
    \bigr\}
    \\&=\inf_{y\in L^2(I;Y^*)}
\bigl\{
    \tfrac{1}{2}\|y-\nabla v\|_Q^2
    +
    I_{\chi_K}^*(f-\partial_t v-L^*y) -
    \langle f-\partial_t v-L^*y,v\rangle_{L^2(I;V)}
\bigr\}
   \\&=\inf_{\eta\in L^2(I;V^*)}{\bigl\{
        \tfrac{1}{2}\|\partial_t v+(-\Delta)v+\eta-f\|_{L^2(I;V^*)}^2
        +I_{\chi_K}^*(\eta)+
        \langle- \eta,v\rangle_{L^2(I;V)}\bigr\}}\,.
        \end{aligned}
\end{align}
\end{lemma}

 \begin{proof} 
    Let $v\in \mathcal{W}(I)$ with $v(0)=u_0$ in $H$ and $v(t)\in K(t)$ for a.e.\ $t\in I$ be fixed, but arbitrary.

\emph{$\bullet$ Optimal strong convexity measure.} Due to the alternative Bregman divergence representation
\eqref{rem:alternative_representation_bregman_divergence.4} (which is applicable since $\mu = u$), \eqref{subsec:obstacle_problem.fenchel_gap} (applied with $(v,w,y)=(u,v,z)$), and the optimality inclusions \eqref{subsec:obstacle_problem.opt.1}--\eqref{subsec:obstacle_problem.opt.2}, we have that
\begin{align*}
\begin{aligned} 
    \mathcal{D}_{I_E}^{-\partial_t u}(v,u) 
    &= \tfrac{1}{2}\|z-\nabla v\|_Q^2
    +
    I_{\chi_K^*}(f-\partial_t u-L^*z)
    -
    \langle
        f-\partial_t u-L^*z,
        v
    \rangle_{L^2(I;V)} 
    \\&=\tfrac{1}{2}\|\nabla (v-u)\|_Q^2
    +
    \langle
        -\Lambda,
        v-u
    \rangle_{L^2(I;V)} \,.
\end{aligned}  
\end{align*}
Moreover, due to the alternative Bregman divergence representation
\eqref{rem:alternative_representation_bregman_divergence.6} (which is applicable since  $I_G\circ\nabla$ is continuous and $I_F$ is proper) and \eqref{subsec:obstacle_problem.fenchel_gap} (applied with $(v,w,y)=(v,u,y)$), we have that
\begin{align*}
    \mathcal{D}_{I_{E^*}}^{u}
    (-\partial_t v,-\partial_t u) 
    &=
    \inf_{ y\in L^2(I;Y^*)}
    \bigl\{
            \tfrac{1}{2}\|y-\nabla u\|_Q^2
    +
    I_{\chi_K^*}(f-\partial_t v-L^*y)
    -
    \langle
        f-\partial_t v-L^*y,
        u
    \rangle_{L^2(I;V)}
    \bigr\}\,.
\end{align*}
Therefore, the Bregman-type representation of the primal error measure
\eqref{lem:bregman_representation_primal_error_measure.0} yields
\begin{align}\label{lem:primal_gap_identity_obstacle_problem.2}
\begin{aligned} 
    \rho_{\mathcal{E}}^2(v)
    &=
    \tfrac{1}{2}\|\nabla (v-u)\|_Q^2
    +
    \langle
        -\Lambda,
        v-u
    \rangle_{L^2(I;V)}+
    \tfrac{1}{2}\|(v-u)(t_{\mathtt{fin}})\|_H^2 \\
    &\quad+
    \inf_{ y\in L^2(I;Y^*)}
    \bigl\{
            \tfrac{1}{2}\|y-\nabla u\|_Q^2
    +
    I_{\chi_K^*}(f-\partial_t v-L^*y)
    -
    \langle
        f-\partial_t v-L^*y,
        u
    \rangle_{L^2(I;V)}
    \bigr\}\,.
    \end{aligned}
\end{align}\pagebreak

\emph{$\bullet$ Primal gap estimator.}
Using the representation
\eqref{lem:representation_of_primal_gap_estimator.0} of the primal gap
estimator \eqref{def:primal_gap_estimator} (which is applicable since $I_G\circ\nabla$ is continuous and $I_F$ is proper) and \eqref{subsec:obstacle_problem.fenchel_gap} (applied with $(v,w,y)=(v,v,y)$),  we find that
\begin{align}\label{lem:primal_gap_identity_obstacle_problem.3}
    \begin{aligned} 
    \eta_{\mathcal{E}}^2(v)
    &=
    \inf_{y\in L^2(I;Y^*)}
\bigl\{
    \tfrac{1}{2}\|y-\nabla v\|_Q^2
    +
    I_{\chi_K}^*(f-\partial_t v-L^*y) -
    \langle f-\partial_t v-L^*y,v\rangle_{L^2(I;V)}
\bigr\}
   \\&=\inf_{\eta\in L^2(I;V^*)}{\bigl\{
        \tfrac{1}{2}\|\partial_t v+(-\Delta)v+\eta-f\|_{L^2(I;V^*)}^2
        +I_{\chi_K}^*(\eta)+
        \langle- \eta,v\rangle_{L^2(I;V)}\bigr\}}\,.
    \end{aligned}
\end{align}

Finally, using the representations
\eqref{lem:primal_gap_identity_obstacle_problem.2} and
\eqref{lem:primal_gap_identity_obstacle_problem.3} and the general primal
gap identity \eqref{thm:primal_gap_identity.0}, we arrive at
the claimed representation \eqref{lem:primal_gap_identity_obstacle_problem.0} of the primal gap identity for the unsteady~obstacle~problem.
\end{proof}

Next, we derive the corresponding dual~gap~identity.\enlargethispage{5mm}\vspace{-0.5mm}

\begin{lemma}[Dual gap identity for the unsteady obstacle problem]\label{lem:dual_gap_identity_obstacle_problem}
For every $(y,\lambda,\eta)\in L^2(I;Y^*)\times \mathcal{W}(I)\times L^2(I;V^*)$
with $\lambda(t)\in K(t)$ for a.e.\ $t\in I$,
$I_{\chi_K^*}(\eta)<+\infty$, and
\begin{align}\label{lem:dual_gap_identity_obstacle_problem.-1}
    \operatorname{div}y=\partial_t \lambda+\eta-f\quad\text{ in }L^2(I;V^*)\,,
\end{align}
there holds
\begin{align}\label{lem:dual_gap_identity_obstacle_problem.0}
\begin{aligned}
    &\tfrac{1}{2}\|y-z\|_{Q}^2
    +
    \tfrac{1}{2}\|\nabla(\lambda-u)\|_{Q}^2
    +
    \tfrac{1}{2}\|(\lambda-u)(t_{\mathtt{fin}})\|_{H}^2
    +
    \langle -\eta,u-\zeta\rangle_{L^2(I;V)}
    +
    \langle -\Lambda,\lambda-\zeta\rangle_{L^2(I;V)}
    \\
    &\quad =
    \tfrac{1}{2}\|y-\nabla \lambda\|_{Q}^2
    +
    \langle -\eta,\lambda-\zeta\rangle_{L^2(I;V)}
    +
    \tfrac{1}{2}\|\lambda(0)-u_0\|_{H}^2\,.
\end{aligned}
\end{align}
\end{lemma} 

\begin{proof}
Let $(y,\lambda,\eta)\in L^2(I;Y^*)\times\mathcal{W}(I)\times L^2(I;V^*)$
with $\lambda(t)\in K(t)$ for a.e.\ $t\in I$,
$I_{\chi_K^*}(\eta)<+\infty$, and \eqref{lem:dual_gap_identity_obstacle_problem.-1} be fixed, but arbitrary.

\emph{$\bullet$ Optimal strong convexity measure.}
Due to the alternative Bregman divergence representation
\eqref{rem:alternative_representation_bregman_divergence.8},
\eqref{subsec:obstacle_problem.opt.1}, and a binomial formula, we have
that
\begin{align*} 
    \mathcal{D}_{I_{G^*}}^{\nabla u}(y,z)
    &=
\tfrac{1}{2}\|y-\nabla u\|_Q^2
\\&=\tfrac{1}{2}\|y-z\|_Q^2\,,
\end{align*}
due to the alternative Bregman divergence representation
\eqref{rem:alternative_representation_bregman_divergence.10}, $I_{F}\hspace{-0.15em}=\hspace{-0.15em}I_{\chi_{K}}-\langle f,\cdot\rangle_{L^2(I;V)}$,~${I_{F^*}\hspace{-0.15em}=\hspace{-0.15em}I_{\chi_{K}^*}(f\hspace{-0.1em}+\hspace{-0.1em}\cdot)}$,
\eqref{lem:dual_gap_identity_obstacle_problem.-1}, and
\eqref{subsec:obstacle_problem.opt.2}, we have that
\begin{align*}
    \mathcal{D}_{I_{F^*}}^u
    (-\partial_t\lambda-L^*y,-\partial_tu-L^*z)
    &=I_{\chi_K^*}(f-\partial_t\lambda-L^*y)-\langle f-\partial_t\lambda-L^*y, u\rangle_{L^2(I;V)}
    \\&=\langle -\eta, u-\zeta\rangle_{L^2(I;V)}
    \,,
\end{align*}
and, due to the alternative Bregman divergence representation
\eqref{rem:alternative_representation_bregman_divergence.4} (which is
applicable since $\mu=u$ in $\mathcal{W}(I)$),
\eqref{subsec:obstacle_problem.fenchel_gap} (applied with
$(v,w,y)=(u,\lambda,z)$), and
\eqref{subsec:obstacle_problem.opt.1}--\eqref{subsec:obstacle_problem.opt.2}, we have that
\begin{align*}
\begin{aligned}
    \mathcal{D}_{I_E}^{-\partial_tu}(\lambda,u)
    &= \tfrac{1}{2}\|z-\nabla \lambda\|_Q^2
    +
    I_{\chi_K^*}(f-\partial_t u-L^*z)
    -
    \langle
        f-\partial_t u-L^*z,
        \lambda
    \rangle_{L^2(I;V)}
    \\
    &=
    \tfrac{1}{2}\|\nabla (\lambda-u)\|_Q^2 
    +
    \langle
        -\Lambda,
        \lambda-u
    \rangle_{L^2(I;V)}\,.
\end{aligned}
\end{align*}
Therefore, the Bregman-type representation of the dual error measure
\eqref{lem:bregman_representation_dual_error_measure.0} yields
\begin{align}\label{lem:dual_gap_identity_obstacle_problem.1}
\begin{aligned}
    \rho_{-\mathcal{D}}^2(y,\lambda)
    &=   \tfrac{1}{2}\|y-z\|_{Q}^2
    +
    \tfrac{1}{2}\|\nabla(\lambda-u)\|_{Q}^2
    +
    \tfrac{1}{2}\|(\lambda-u)(t_{\mathtt{fin}})\|_{H}^2
    \\&\quad+
    \langle -\eta,u-\zeta\rangle_{L^2(I;V)}
    +
    \langle -\Lambda,\lambda-\zeta\rangle_{L^2(I;V)}\,.
\end{aligned}
\end{align}

\emph{$\bullet$ Dual gap estimator.}
Using the representation \eqref{lem:representation_of_dual_gap_estimator.0} of the dual gap estimator
\eqref{def:dual_gap_estimator} and \eqref{subsec:obstacle_problem.fenchel_gap} (which is applicable due to \eqref{lem:dual_gap_identity_obstacle_problem.-1}),  we find that
\begin{align}\label{lem:dual_gap_identity_obstacle_problem.2}
\begin{aligned} 
    \eta_{-\mathcal{D}}^2(y,\lambda)
    &=
    \tfrac{1}{2}\|y-\nabla \lambda\|_Q^2
    \\&\quad+
    I_{\chi_K^*}(f-\partial_t \lambda-L^*y)
    -
    \langle
        f-\partial_t \lambda-L^*y,
        \lambda
    \rangle_{L^2(I;V)}
    +
    \tfrac{1}{2}\|\lambda(0)-u_0\|_H^2
    \\&=\tfrac{1}{2}\|y-\nabla \lambda\|_Q^2
    +
    \langle
        -\eta ,
        \lambda-\zeta
    \rangle_{L^2(I;V)}
    +
    \tfrac{1}{2}\|\lambda(0)-u_0\|_H^2\,.
\end{aligned}
\end{align}

Finally, using the representations
\eqref{lem:dual_gap_identity_obstacle_problem.1} and
\eqref{lem:dual_gap_identity_obstacle_problem.2} and the general dual gap
identity \eqref{thm:dual_gap_identity.0}, we arrive at the claimed representation 
\eqref{lem:dual_gap_identity_obstacle_problem.0} of the dual gap identity for the unsteady~obstacle~problem.
\end{proof}\newpage

 \subsection{The unsteady elasto-plastic torsion problem}
\label{subsec:unsteady_elasto_plastic_torsion}\enlargethispage{3.5mm}

\hspace{5mm}In this subsection, we consider the unsteady \emph{elasto-plastic torsion problem} (\textit{cf}.\ \cite{DuvautLions1972}, see also \cite{AntilBartelsKaltenbachKhandelwal2025,ChoulyHild2026,ChoulyGustafssonHild2026}), a classical model problem in structural mechanics originating in Prandtl's theory of plastic torsion (\textit{cf}.\ \cite{Prandtl1903}), which, in the present time-dependent setting, describes the evolution of the torsional response of a prismatic bar made of an elastic-perfectly plastic material under time-dependent loading and is formulated as a variational problem with a pointwise gradient constraint encoding the yield condition.

The present problem naturally leads to the endpoint case $p=\infty$.
Although Sections~\ref{sec:preliminaries}--\ref{sec:duality_based_a_posteriori_error_control} are formulated for $p\in(1,\infty)$, several
of the underlying convex-analytic arguments extend directly to this
setting. We exploit these extensions below and follow the structure
of the preceding applications, indicating explicitly where an
additional argument is required.

Let $\ell=1$ and $p=\infty$, \textit{i.e.}, $V=W^{1,\infty}_D(\Omega;\mathbb{R}^1)$, $Y=L^\infty(\Omega;\mathbb{R}^d)$, and $H=L^2(\Omega;\mathbb{R}^1)$, where $V$ is equipped with the gradient norm $\|\cdot\|_V\coloneqq \|\nabla (\cdot)\|_{\infty,\Omega}$ in $V$. In particular,  the dual and predual~space~of~$Y$\linebreak are characterized by $Y^*\cong\mathrm{ba}(\Omega;\mathbb{R}^d)$ and $Y_*\cong L^1(\Omega;\mathbb{R}^d)$ (\textit{i.e.}, $(Y_*)^* = Y$), respectively. Here, $\mathrm{ba}(\Omega;\mathbb{R}^d)$\linebreak denotes the space of bounded
and finitely additive vector measures on $\mathcal{L}^d(\Omega)$~that~vanish~on~\mbox{Lebesgue-null} sets (\textit{cf}.~\mbox{\cite[Def.~4.1]{Toland20}}). Since $\nabla V \subseteq Y=Y_*^*$ is weakly-$*$ closed, a predual~space~of~$V$~is~characterized by  $V_*\hspace{-0.1em}\cong\hspace{-0.1em}  Y_*/{}^\perp(\nabla V)$, where ${{}^\perp(\nabla V)\hspace{-0.1em}\coloneqq\hspace{-0.1em} \{y\hspace{-0.1em}\in\hspace{-0.1em}  Y_*\mid (y,\nabla v)_{\Omega}\hspace{-0.1em}=\hspace{-0.1em}0\text{ for all } v\hspace{-0.1em}\in\hspace{-0.1em} V\}}$~(\textit{cf}.~\cite[Thms.~4.7(b)~and~4.9(b)]{Rudin1991}). 

Moreover, let the energy densities
$\phi\colon Q\times\mathbb{R}^d\to\mathbb{R}\cup\{+\infty\}$ and
$\psi\colon Q\times\mathbb{R}\to\mathbb{R}$, for a.e.\ $(t,x)\in Q$
and every $a\in\mathbb{R}^d$ and $b\in\mathbb{R}$, respectively, be
defined by
\begin{subequations}\label{subsec:unsteady_elasto_plastic_torsion.1}
    \begin{align}\label{subsec:unsteady_elasto_plastic_torsion.1.1}
    \phi(t,x,a)
    &\coloneqq
    \tfrac{1}{2}|a|^2+\chi_{[0,+\infty)}(\zeta(t,x)-\vert a\vert)\,, 
    \\\label{subsec:unsteady_elasto_plastic_torsion.1.2}
    \psi(t,x,b)
    &\coloneqq
    0\,,
    \end{align}
\end{subequations}
for some
$\zeta\in L^\infty(Q)\cap C_{\mathrm{w}}^0(I;H)$ with
$\zeta\geq\zeta_0$ a.e. in $Q$, where $\zeta_0\in(0,+\infty)$, so that   $G\colon  I\times Y\to \mathbb{R}\cup\{+\infty\}$ is given via $G(t,y)\hspace{-0.15em}\coloneqq \hspace{-0.15em} \tfrac{1}{2}\|y\|_{\Omega}^2+\chi_{K(t)}(y)$ for a.e.\ $t\hspace{-0.15em}\in\hspace{-0.15em} I$ and all $y\hspace{-0.15em}\in\hspace{-0.15em} Y$ and  ${F\colon \hspace{-0.15em}I\hspace{-0.15em}\times \hspace{-0.15em}V\hspace{-0.175em}\to\hspace{-0.15em} \mathbb{R}}$,~for~some~${f\hspace{-0.175em}\in\hspace{-0.175em} L^1(I;V_*)}$, is given via $F(t,v)\coloneqq -\langle f(t),v\rangle_V$ for a.e.\ $t\in I$~and~all~$v\in V$.
Here, we interpret $\chi_K\coloneqq I_{\phi-\frac{1}{2}\vert \cdot\vert^2}^{\Omega}\colon I\times Y\to \mathbb{R}\cup\{+\infty\}$ as spatial integral reduction (\textit{cf}.\ Lemma \ref{lem:integral_functionals_as_normal_integrands}), which, for a.e.\ $t\in I$ and every $y\in Y$,~is~given~via\vspace{-0.5mm}
\begin{align*}
\chi_{K(t)}(y)
\coloneqq
\begin{cases}
    0&\text{ if }y\in K(t)\,,\\
    +\infty&\text{ else}\,,
\end{cases}\\[-6mm]\notag
\end{align*}
where\vspace{-0.5mm}
\begin{align*}
    K(t)\coloneqq \smash{\bigl\{y\in Y\mid \vert y\vert \leq \zeta(t)\text{ a.e.\ in }\Omega\bigr\}}\,.
\end{align*}

For a.e.\ $t\hspace{-0.1em}\in\hspace{-0.1em} I$, the steady primal
energy functional $E(t,\cdot)\colon \hspace{-0.1em}V\hspace{-0.1em}\to\hspace{-0.1em}\mathbb{R}\cup\{+\infty\}$, for every $v\hspace{-0.1em}\in\hspace{-0.1em} V$,~is~given~via 
\begin{align}\label{subsec:unsteady_elasto_plastic_torsion.4}
    \smash{E(t,v)
    =
    \tfrac{1}{2}\|\nabla v\|_\Omega^2
    +
    \chi_{K(t)}(\nabla v)
    -
    \langle f(t),v\rangle_{V}\,.}
\end{align} 

Since, for a.e.\ $(t,x)\in Q$, the Fenchel conjugate functional (with respect to the last argument) $\phi^*\colon Q\times \mathbb{R}^d\to \mathbb{R}$ of \eqref{subsec:unsteady_elasto_plastic_torsion.1.1}, for a.e.\ $(t,x)\in Q$ and  every $a\in \mathbb{R}^d$, is given via 
\begin{align}\label{subsec:unsteady_elasto_plastic_torsion.4.5}
    \smash{\phi^*(t,x,a)
=
\tfrac{1}{2}|a|^2-\tfrac{1}{2}(|a|-\zeta(t,x))_+^2\,,}
\end{align} 
by the convex conjugation formula for integral functionals defined on $Y$ (\textit{cf}.\ \cite[Thm.\ 1]{Rockafellar1971II}), for a.e.\ $t\in I$,\linebreak the Fenchel conjugate functional $G^*(t,\cdot)\colon Y^*\to \mathbb{R}$, for every $y=y^{\mathrm{a}}\otimes\mathrm{d}x+y^{\mathrm{s}}\in Y^*$, where
$y^{\mathrm{a}}\in Y_*$ (the $\sigma$-additive part) and $y^{\mathrm{s}}\in Y^*$ (the purely finitely additive part) with $y^{\mathrm{s}}\perp y^{\mathrm{a}}\otimes \mathrm{d}x$ form the unique Yosida--Hewitt decomposition of $y \in \mathrm{ba}(\Omega;\mathbb{R}^d)$~(\textit{cf}.~\mbox{\cite[Thm.~4.13]{Toland20}}), is given via  
\begin{align}\label{subsec:unsteady_elasto_plastic_torsion.5}
    \smash{G^*(t,y)
    =I_{\phi^*}^{\Omega}(t,y^{\mathrm{a}}) 
    +|\zeta(t)y^{\mathrm{s}}|(\Omega)\,.}
\end{align} 
As a consequence, since, due to $\zeta\geq\zeta_0>0$ a.e.\ in $Q$, for a.e.\ $t\in I$,
$G(t,\cdot)$ is continuous at $0=L0$ in $Y$, the Fenchel conjugate
$E^*(t,\cdot)\colon V^*\to\mathbb{R}\cup\{+\infty\}$ of \eqref{subsec:unsteady_elasto_plastic_torsion.4}, 
for every $v^*\in V^*$, is given via the infimal convolution (\textit{cf}.~\cite[Thm.~9.4.1]{AttouchButtazzoMichaille2014})\vspace{-0.5mm}
\begin{align}\label{subsec:unsteady_elasto_plastic_torsion.6}
    E^*(t,v^*)
    =
    \inf_{\substack{y=y^{\mathrm{a}}\otimes \mathrm{d}x+y^{\mathrm{s}}\in\mathrm{ba}(\Omega;\mathbb{R}^d)\\
    L^*y=f(t)+v^*\text{ in }V^*}}
    \bigl\{I_{\phi^*}^{\Omega}(t,y^{\mathrm{a}})+
    |\zeta(t)y^{\mathrm{s}}|(\Omega)
    \bigr\}\,,\\[-6mm]\notag
\end{align}
where $L^*\colon Y^*\to V^*$ denotes the adjoint operator of the gradient operator $L\coloneqq \nabla \colon V\to Y$. 

If the Laplace operator $-\Delta\coloneqq L^*\circ (\nabla (\cdot)\otimes \mathrm{d}x)\colon V\to V^*$, for every  $v,w\in V$, is defined by\enlargethispage{1mm}
\begin{align*}
   \smash{\langle(-\Delta)v,w\rangle_V
    \coloneqq
    (\nabla v,\nabla w)_\Omega}\,,
\end{align*} and, for a.e.\ $t\in I$~and~every~$y\in Y$,
denoting the normal cone to $K(t)$ by
\begin{align*}
    \smash{N_{K(t)}(y)
    \coloneqq
    \partial\chi_{K(t)}(y)
    =
    \{
        y^*\in Y^*
        \mid
        \langle y^*,\varphi-y\rangle_Y\leq 0
        \text{ for all }\varphi\in K(t)
    \}\,,}
\end{align*}
the abstract subgradient-flow problem \eqref{eq:subgradient_flow}, given an arbitrary initial
datum $u_0\in V$ with $\nabla u_0\in K(0)$, is given by the unsteady elasto-plastic torsion problem: find $u\in  \mathcal{W}_\infty(I)$, 
where (\textit{cf}.\ \cite{Pedregal1997})
\begin{align*}
    \mathcal{W}_\infty(I)
    &\coloneqq
    \bigl\{v\in L^\infty_{\mathrm{w}^*}(I;V)\mid \partial_t v\in L^1(I;V_*)\bigr\}\hookrightarrow C^0(\overline{I};H)\,,\\
    L^\infty_{\mathrm{w}^*}(I;V)&\coloneqq\bigl\{v\colon I\hspace{-0.1em}\to\hspace{-0.1em} V\mid 
         \langle \nabla v(\cdot),w\rangle_{Y_*}\colon\hspace{-0.1em} I\hspace{-0.1em}\to\hspace{-0.1em} \mathbb{R}\text{ is }\mathcal{L}^1(I)\text{-measurable for all }w\hspace{-0.1em}\in\hspace{-0.1em} Y_*\,,\;\|v(\cdot)\|_V\hspace{-0.1em}\in\hspace{-0.1em} L^\infty(I)\bigr\}\,,
\end{align*}
where $L^\infty_{\mathrm{w}^*}(I;V)$ forms a Banach space equipped with the norm $\|\cdot\|_{L^\infty_{\mathrm{w}^*}(I;V)}\coloneqq \text{ess\,sup}_{t\in I}{\{\|(\cdot)(t)\|_V\}}$~and, by the separability of $V_*\cong Y_*/{}^\perp(\nabla V)$ 
satisfies $(L^1(I;V_*))^*\cong L^\infty_{\mathrm{w}^*}(I;V)$ (\textit{cf}.\ \cite[Thm.\ 6.14]{Pedregal1997}),~such~that\vspace{-4.5mm}
\begin{subequations}\label{subsec:elasto_plastic_torsion_problem.0}
\begin{alignat}{2}
    \partial_t u(t)+(-\Delta)u(t)+L^*N_{K(t)}(\nabla u(t))
    &\ni f(t)
    &&\quad\text{ in }V^*
    \quad\text{ for a.e.\ }t\in I\,,
    \label{subsec:elasto_plastic_torsion.0.1}
    \\
    u(0)&=u_0
   && \quad\text{ in }H\,.
    \label{subsec:elasto_plastic_torsion.0.2}
\end{alignat}
\end{subequations}
If, in addition, $ \zeta\in W^{1,\infty}(I;L^\infty(\Omega))$ and $f\in L^2(Q)$,
according to \cite[Thm.~2.10]{MirandaRodriguesSantos2020}, the unsteady~elasto-plastic torsion problem \eqref{subsec:elasto_plastic_torsion_problem.0} admits a unique solution $u\in \mathcal{W}_\infty(I)$. More precisely, \cite[Thm.~2.10]{MirandaRodriguesSantos2020} yields a unique solution $u\in L^2(I;W_D^{1,2}(\Omega;\mathbb{R}^1))\cap H^1(I;H)$ satisfying $|\nabla u|\leq\zeta$ a.e.\ in $Q$,  which implies that $u\in\mathcal{W}_\infty(I)$, but, in general, not that  $u\in L^\infty(I;V)$ due to a possible lack of Bochner measurability.\linebreak
Motivated \hspace{-0.15mm}by \hspace{-0.15mm}\cite[\hspace{-0.2mm}Thm.~\hspace{-0.2mm}2.10]{MirandaRodriguesSantos2020}, \hspace{-0.15mm}without \hspace{-0.15mm}imposing \hspace{-0.15mm}additional \hspace{-0.15mm}regularity \hspace{-0.15mm}assumptions~\hspace{-0.15mm}on~\hspace{-0.15mm}${\zeta\hspace{-0.25em}\in \hspace{-0.2em}L^\infty(Q)\hspace{-0.15em}\cap\hspace{-0.15em} C_{\mathrm{w}}^0(I;H)}$ and $f\in L^1(I;V_*)$, we assume that there exists a solution $u\in\mathcal{W}_\infty(I)$.

The argument of the Br\'ezis--Ekeland--Nayroles principle
(\textit{cf}.\ Proposition~\ref{prop:brezis_ekeland_nayroles}) extends to the
present setting. Therefore, the solution $u\in\mathcal{W}_\infty(I)$
is equivalently characterized as the primal~\mbox{solution},~\textit{i.e.},~a~mi\-nimizer of the unsteady primal energy functional
$\mathcal{E}\colon\hspace{-0.1em}\mathcal{W}_\infty(I)\hspace{-0.1em}\to\hspace{-0.1em}\mathbb{R}\cup\{+\infty\}$,~for~every~$v\hspace{-0.1em}\in\hspace{-0.1em}\mathcal{W}_\infty(I)$~defined~by\vspace{-0.5mm}
\begin{align}\label{subsec:elasto_plastic_torsion_problem.primal}
\begin{aligned}
    \mathcal{E}(v)
    &\coloneqq
    \tfrac{1}{2}\|\nabla v\|_Q^2
    +
    I_{\chi_K}(\nabla v)
    -
    \langle f,v\rangle
    \\
    &\quad+
    \inf_{\substack{
        y=y^{\mathrm{a}}\,\mathrm{d}t\mathrm{d}x+y^{\mathrm{s}}\in\operatorname{ba}(Q;\mathbb{R}^d)\\
        L^*y=f-\partial_t v\text{ in }(L^\infty_{\mathrm{w}^*}(I;V))^*}}
    \bigl\{
        I_{\phi^*}(y^{\mathrm{a}})
        +
        |\zeta y^{\mathrm{s}}|(Q)
    \bigr\}
    +
    \tfrac{1}{2}\|v(t_{\mathtt{fin}})\|_H^2
    +
    \chi_{\{u_0\}}(v(0))\,,
\end{aligned}\\[-6mm]\notag
\end{align}
satisfying\vspace{-0.5mm}
\begin{align}\label{subsec:elasto_plastic_torsion_problem.primal.2}
    \smash{\mathcal{E}(u)=\tfrac{1}{2}\|u_0\|_H^2\,.}
\end{align}
Here, $L^*\colon (L^\infty(Q;\mathbb{R}^d))^*\to (L^\infty_{\mathrm{w}^*}(I;V))^*$, where $(L^\infty(Q;\mathbb{R}^d))^*\cong \operatorname{ba}(Q;\mathbb{R}^d)$ (\textit{cf}.\ \cite[Thm.\ 3.1]{Toland20}), denotes the adjoint operator of the gradient operator $\nabla \colon L^\infty_{\mathrm{w}^*}(I;V)\to L^\infty_{\mathrm{w}^*}(I;Y)$, where $L^\infty_{\mathrm{w}^*}(I;Y)\cong L^\infty(Q;\mathbb{R}^d)$ (\textit{cf}.\ \cite[Thm.\ 6.14]{Pedregal1997}).\enlargethispage{2.5mm}

Proceeding as in the proof of Theorem~\ref{thm:duality}(\hyperlink{thm:duality.i}{i}), we derive the restricted unsteady dual energy~functional  $\mathcal{D}
    \colon 
    \operatorname{ba}(Q;\mathbb{R}^d)\times\mathcal{W}_\infty(I)
    \to \mathbb{R}\cup\{-\infty\}$,   for every 
$y=y^{\mathrm{a}}\,\mathrm{d}t\mathrm{d}x+y^{\mathrm{s}}\in
\operatorname{ba}(Q;\mathbb{R}^d)$, where $y^{\mathrm{a}}\in L^1(I;Y_*)\cong L^1(Q;\mathbb{R}^d)$ \hspace{-0.15mm}(the \hspace{-0.15mm}$\sigma$-additive \hspace{-0.15mm}part) \hspace{-0.15mm}and \hspace{-0.15mm}$y^{\mathrm{s}}\hspace{-0.175em}\in \hspace{-0.175em}\mathrm{ba}(Q;\mathbb{R}^d)$ \hspace{-0.15mm}(the \hspace{-0.15mm}purely \hspace{-0.15mm}finitely \hspace{-0.15mm}additive \hspace{-0.15mm}part)~\hspace{-0.15mm}with~\hspace{-0.15mm}${y^{\mathrm{s}}\hspace{-0.2em}\perp\hspace{-0.2em} y^{\mathrm{a}}\,\mathrm{d}t\mathrm{d}x}$~\hspace{-0.15mm}form the \hspace{-0.1mm}unique \hspace{-0.1mm}Yosida--Hewitt \hspace{-0.1mm}decomposition \hspace{-0.1mm}of \hspace{-0.1mm}$y\hspace{-0.15em}\in\hspace{-0.15em}\operatorname{ba}(Q;\mathbb{R}^d) $ \hspace{-0.1mm}(\textit{cf}.~\hspace{-0.1mm}\mbox{\cite[\hspace{-0.1mm}Thm.~\hspace{-0.1mm}4.13]{Toland20}}),
\hspace{-0.1mm}and
\hspace{-0.1mm}$\lambda\hspace{-0.15em}\in\hspace{-0.15em}\mathcal{W}_\infty(I)$~\hspace{-0.1mm}\mbox{defined}~\hspace{-0.1mm}by
\begin{align}\label{subsec:elasto_plastic_torsion_problem.dual}
\begin{aligned}
    \mathcal{D}(y,\lambda)
    &\coloneqq 
    -I_{\phi^*}(y^{\mathrm{a}})
    -|\zeta y^{\mathrm{s}}|(Q)
    -\chi_{\{-f\}}
        (-L^*y-\partial_t\lambda)
    \\
    &\quad
    -\tfrac{1}{2}\|\nabla\lambda\|_Q^2
    -I_{\chi_K}(\nabla\lambda)
    +\langle f,\lambda\rangle_{\smash{L^\infty_{\mathrm{w}^*}(I;V)}}
    -\tfrac{1}{2}\|\lambda(t_{\mathtt{fin}})\|_H^2
    +(\lambda(0),u_0)_H\,.
\end{aligned}
\end{align}
Note that the non-restricted unsteady dual energy functional is defined on $\operatorname{ba}(Q;\mathbb{R}^d)\times L^\infty_{\mathrm{w}^*}(I;V)$.

Since $\zeta\geq\zeta_0>0$ a.e.\ in $Q$, the functional
$I_G\colon L^\infty(Q;\mathbb{R}^d)\to\mathbb{R}\cup\{+\infty\}$
is continuous at $0$. Therefore, since $-\partial_tu\in\partial I_E(u)$,  
the standard subdifferential calculus (\textit{cf}.\ \cite[Props.\ 5.6-7, p.\ 26--27] {EkelandTemam1999}) yields the existence of
$z=z^{\mathrm{a}}\,\mathrm{d}t\mathrm{d}x+z^{\mathrm{s}}\in
\operatorname{ba}(Q;\mathbb{R}^d)$, where $z^{\mathrm{a}}\in L^1(I;Y_*)$ and $z^{\mathrm{s}}\in \mathrm{ba}(Q;\mathbb{R}^d)$ with $z^{\mathrm{s}}\perp z^{\mathrm{a}}\,\mathrm{d}t\mathrm{d}x$ form a unique Yosida--Hewitt decomposition of $z\in\operatorname{ba}(Q;\mathbb{R}^d) $ (\textit{cf}.~\mbox{\cite[Thm.~4.13]{Toland20}}), 
such that
\begin{subequations}\label{subsec:elasto_plastic_torsion_problem.opt}
\begin{align}\label{subsec:elasto_plastic_torsion_problem.opt.1}
    z&\in\partial I_G(\nabla u)\,,\\
    L^*z+\partial_tu&=f\quad\text{ in }(L^\infty_{\mathrm{w}^*}(I;V))^*\,.
        \label{subsec:elasto_plastic_torsion_problem.opt.2}
\end{align}
\end{subequations}
Since $I_G^*(z)=I_{\smash{I_{\phi^*}^{\Omega}}}(z^{\mathrm{a}})+\vert \zeta z^{\mathrm{s}}\vert(Q)$ (\textit{cf}.\ \cite[Thm.\ 1]{Rockafellar1971II}),
the optimality inclusion \eqref{subsec:elasto_plastic_torsion_problem.opt.1} is equivalent to
\begin{align}
    \nabla u&=\mathrm{D}_a\phi^*(t,x,z^{\mathrm{a}})=\Pi_\zeta(z^{\mathrm{a}})
        \quad\text{ a.e.\ in }Q\,,
        \label{subsec:elasto_plastic_torsion_problem.opt.3}\\
    |\zeta z^{\mathrm{s}}|(Q)
        &=\langle z^{\mathrm{s}},\nabla u\rangle_{L^\infty(Q;\mathbb{R}^d)}\,.
        \label{subsec:elasto_plastic_torsion_problem.opt.4}
\end{align}
Hence, setting $\mu\coloneqq u\in \mathcal{W}_\infty(I)$, the pair $(z,\mu)\in \mathrm{ba}(Q;\mathbb{R}^d)\times \mathcal{W}_\infty(I)$ is admissible for~the~dual~problem. The Fenchel equality corresponding to
\eqref{subsec:elasto_plastic_torsion_problem.opt.2} and the
integration-by-parts formula in time yield
\begin{align}\label{subsec:elasto_plastic_torsion_problem.opt.5}
    \mathcal{D}(z,\mu)
    =
    \tfrac{1}{2}\|u_0\|_H^2
    =
    \mathcal{E}(u)\,.
\end{align} 
Thus, by weak duality, $(z,\mu)$ is a maximizer of \eqref{subsec:elasto_plastic_torsion_problem.dual}
 and a strong duality relation applies.

\hspace{-0.5mm}Next, \hspace{-0.1mm}we \hspace{-0.1mm}record \hspace{-0.1mm}the \hspace{-0.1mm}corresponding \hspace{-0.1mm}primal \hspace{-0.1mm}and \hspace{-0.1mm}dual \hspace{-0.1mm}gap \hspace{-0.1mm}identities
\hspace{-0.1mm}(\textit{cf}.\ \hspace{-0.1mm}Theorems~\hspace{-0.1mm}\ref{thm:primal_gap_identity} \hspace{-0.1mm}and
\hspace{-0.1mm}\ref{thm:dual_gap_identity}).~\hspace{-0.1mm}To~\hspace{-0.1mm}this~\hspace{-0.1mm}end, we note that, for every $v,w\in\mathcal{W}_\infty(I)$
with $\nabla w(t)\in K(t)$ for a.e.\ $t\in I$ and  $y=y^{\mathrm{a}}\,\mathrm{d}t\mathrm{d}x+y^{\mathrm{s}}\in \mathrm{ba}(Q;\mathbb{R}^d)$ with $L^*y=f-\partial_tv$ in $(L^\infty_{\mathrm{w}^*}(I;V))^*$,
due to $I_G=I_{I_{\phi}^{\Omega}}$ with \eqref{subsec:unsteady_elasto_plastic_torsion.1.1}, $(I_G)^*(y)
    =
    I_{\phi^*}(y^{\mathrm{a}})+|\zeta y^{\mathrm{s}}|(Q)$ with \eqref{subsec:unsteady_elasto_plastic_torsion.4.5}, $I_F=-\langle f,\cdot\rangle_{\smash{L^\infty_{\mathrm{w}^*}(I;V)}}$, and $(I_F)^*=\chi_{\{-f\}}$, there holds
\begin{align}
\label{subsec:elasto_plastic_torsion.fenchel_gap}
\begin{aligned}
    &(I_G)^*(y)
    -
    \langle y,\nabla w\rangle_{L^\infty(Q;\mathbb{R}^d)}
    +
    I_G(\nabla w)
    \\
    &\quad+
    (I_F)^*(-L^*y-\partial_tv)
    -
    \langle-L^*y-\partial_tv,w\rangle_{\smash{L^\infty_{\mathrm{w}^*}(I;V)}}
    +
    I_F(w)
    \\
    &=
    I_{\phi^*}(y^{\mathrm{a}})
-
(y^{\mathrm{a}},\nabla w)_Q
+
\tfrac{1}{2}\|\nabla w\|_Q^2+
\smash{\bigl(|\zeta y^{\mathrm{s}}|(Q)
-
\langle y^{\mathrm{s}},\nabla w\rangle_{L^\infty(Q;\mathbb{R}^d)}\bigr)}\,.
\end{aligned}
\end{align}

First, \hspace{-0.1mm}using \hspace{-0.1mm}repeatedly \hspace{-0.1mm}the \hspace{-0.1mm}elementary \hspace{-0.1mm}identity
\hspace{-0.1mm}\eqref{subsec:elasto_plastic_torsion.fenchel_gap}, \hspace{-0.1mm}we \hspace{-0.1mm}derive \hspace{-0.1mm}the
\hspace{-0.1mm}corresponding~\hspace{-0.1mm}primal~\hspace{-0.1mm}gap~\hspace{-0.1mm}\mbox{identity}.

\begin{lemma}[Primal gap identity for the unsteady elasto-plastic
torsion problem]
\label{lem:primal_gap_identity_elasto_plastic_torsion}
For every $v\in\mathcal{W}_\infty(I)$ with $v(0)=u_0$ in $H$ and
$\nabla v(t)\in K(t)$ for a.e.\ $t\in I$, there holds
\begin{align}
    &\tfrac{1}{2}\|\nabla(v-u)\|_Q^2
    +
    \smash{\bigl(
        \tfrac{|z^{\mathrm{a}}|}{\zeta}-1,
        \zeta^2-\nabla u\cdot\nabla v
    \bigr)}_{\smash{\{|\nabla u|=\zeta\}}} +
    |\zeta z^{\mathrm{s}}|(Q)
    -
    \langle z^{\mathrm{s}},\nabla v\rangle_{L^\infty(Q;\mathbb{R}^d)}
    +
    \tfrac{1}{2}\|(v-u)(t_{\mathtt{fin}})\|_H^2\notag
    \\
    &\quad+
    \inf_{\substack{
        y=y^{\mathrm{a}}\,\mathrm{d}t\,\mathrm{d}x+y^{\mathrm{s}}\in \mathrm{ba}(Q;\mathbb{R}^d)\\
        L^*y=f-\partial_tv\ \text{in }(L^\infty_{\mathrm{w}^*}(I;V))^*
    }}
    \Bigl\{
        I_{\phi^*}(y^{\mathrm{a}})
        -
        ( y^{\mathrm{a}},\nabla u)_{Q}
        +
        \tfrac{1}{2}\|\nabla u\|_Q^2 
        +
        |\zeta y^{\mathrm{s}}|(Q)
        -
        \langle y^{\mathrm{s}},\nabla u\rangle_{L^\infty(Q;\mathbb{R}^d)}
    \Bigr\}\notag
    \\
    &=
    \inf_{\substack{
        y=y^{\mathrm{a}}\,\mathrm{d}t\,\mathrm{d}x+y^{\mathrm{s}}\in \mathrm{ba}(Q;\mathbb{R}^d)\\
        L^*y=f-\partial_tv\ \text{in }(L^\infty_{\mathrm{w}^*}(I;V))^*
    }}
    \Bigl\{
        I_{\phi^*}(y^{\mathrm{a}})
        -
        (y^{\mathrm{a}},\nabla v)_Q
        +
        \tfrac{1}{2}\|\nabla v\|_Q^2 
        +
        |\zeta y^{\mathrm{s}}|(Q)
        -
        \langle y^{\mathrm{s}},\nabla v\rangle_{L^\infty(Q;\mathbb{R}^d)}
    \Bigr\}\,. \label{lem:primal_gap_identity_elasto_plastic_torsion.0}
\end{align}
\end{lemma}

\begin{proof}
Let $v\in\mathcal{W}_\infty(I)$ with $v(0)=u_0$ in $H$ and
$\nabla v(t)\in K(t)$ for a.e.\ $t\in I$ be fixed, but arbitrary.

\emph{$\bullet$ Optimal strong convexity measure.}
Due to the endpoint analogue of the alternative~Bregman
divergence \hspace{-0.15mm}representation
\hspace{-0.15mm}\eqref{rem:alternative_representation_bregman_divergence.4}
\hspace{-0.15mm}(which \hspace{-0.15mm}is \hspace{-0.15mm}applicable \hspace{-0.15mm}since \hspace{-0.15mm}$\mu\hspace{-0.175em}=\hspace{-0.175em}u$),
\hspace{-0.15mm}\eqref{subsec:elasto_plastic_torsion.fenchel_gap}
\hspace{-0.15mm}(applied~\hspace{-0.15mm}with~\hspace{-0.15mm}${(v,w,y)\hspace{-0.175em}=\hspace{-0.175em}(u,v,z)}$), \eqref{subsec:elasto_plastic_torsion_problem.opt.4}, that $I_{\phi^*}(z^{\mathrm{a}})
    -
    (z^{\mathrm{a}},\nabla u)_Q
    +
    \tfrac{1}{2}\|\nabla u\|_Q^2
    =0$ (which is equivalent to \eqref{subsec:elasto_plastic_torsion_problem.opt.3}), and a binomial formula, we have that
\begin{align*}
\begin{aligned}
    \mathcal{D}_{I_E}^{-\partial_tu}(v,u)
    &=
    I_{\phi^*}(z^{\mathrm{a}})
    -
    (z^{\mathrm{a}},\nabla v)_Q
    +
    \tfrac{1}{2}\|\nabla v\|_Q^2
    +
    |\zeta z^{\mathrm{s}}|(Q)
    -
    \langle z^{\mathrm{s}},\nabla v\rangle_{L^\infty(Q;\mathbb{R}^d)}
    \\
    &=
    \tfrac{1}{2}\|\nabla(v-u)\|_Q^2
    +
    \smash{\bigl(
        \tfrac{|z^{\mathrm{a}}|}{\zeta}-1,
        \zeta^2-\nabla u\cdot\nabla v
    \bigr)}_{\{|\nabla u|=\zeta\}}+
    |\zeta z^{\mathrm{s}}|(Q)
    -
    \langle z^{\mathrm{s}},\nabla v\rangle_{L^\infty(Q;\mathbb{R}^d)} \,.
\end{aligned}
\end{align*}
Moreover, due to the endpoint analogue of the alternative Bregman
divergence representation
\eqref{rem:alternative_representation_bregman_divergence.6}
and
\eqref{subsec:elasto_plastic_torsion.fenchel_gap}
(applied with $(v,w,y)=(v,u,y)$), we have that
\begin{align*}
\begin{aligned}
    \mathcal{D}_{(I_E)^*}^{u}
    (-\partial_t v,-\partial_t u)
    &=
    \inf_{\substack{
        y=y^{\mathrm{a}}\,\mathrm{d}t\mathrm{d}x+y^{\mathrm{s}}\in\operatorname{ba}(Q;\mathbb{R}^d)
        \\
        L^*y=f-\partial_t v
        \ \text{in }(L^\infty_{\mathrm{w}^*}(I;V))^*
    }}
    \bigl\{
        I_{\phi^*}(y^{\mathrm{a}})
        -
        (y^{\mathrm{a}},\nabla u)_Q
        +
        \tfrac{1}{2}\|\nabla u\|_Q^2
        \\[-5mm]
        &\qquad\qquad\qquad\qquad\qquad\qquad
        +
        |\zeta y^{\mathrm{s}}|(Q)
        -
        \langle y^{\mathrm{s}},\nabla u\rangle_{L^\infty(Q;\mathbb{R}^d)}
    \bigr\}\,.
\end{aligned}
\end{align*}
Hence, the endpoint analogue of the Bregman-type
representation of the primal~error~measure~\eqref{lem:bregman_representation_primal_error_measure.0}~yields
\begin{align}
\label{lem:primal_gap_identity_elasto_plastic_torsion_problem.1}
\begin{aligned}
    \rho_{\mathcal{E}}^2(v)
    &=
    \tfrac{1}{2}\|\nabla(v-u)\|_Q^2
    +
    \smash{\bigl(
        \tfrac{|z^{\mathrm{a}}|}{\zeta}-1,
        \zeta^2-\nabla u\cdot\nabla v
    \bigr)}_{\{|\nabla u|=\zeta\}}
    \\
    &\quad+
    |\zeta z^{\mathrm{s}}|(Q)
    -
    \langle z^{\mathrm{s}},\nabla v\rangle_{L^\infty(Q;\mathbb{R}^d)}
    +
    \tfrac{1}{2}\|(v-u)(t_{\mathtt{fin}})\|_H^2
    \\
    &\quad+
    \inf_{\substack{
        y=y^{\mathrm{a}}\,\mathrm{d}t\mathrm{d}x+y^{\mathrm{s}}\in\operatorname{ba}(Q;\mathbb{R}^d)
        \\
        L^*y=f-\partial_t v
        \ \text{in }(L^\infty_{\mathrm{w}^*}(I;V))^*
    }}
    \bigl\{
        I_{\phi^*}(y^{\mathrm{a}})
        -
        (y^{\mathrm{a}},\nabla u)_Q
        +
        \tfrac{1}{2}\|\nabla u\|_Q^2
        \\[-5mm]
        &\qquad\qquad\qquad\qquad\qquad\qquad
        +
        |\zeta y^{\mathrm{s}}|(Q)
        -
        \langle y^{\mathrm{s}},\nabla u\rangle_{L^\infty(Q;\mathbb{R}^d)}
    \bigr\}\,.
\end{aligned}
\end{align}

\emph{$\bullet$ Primal gap estimator.}
Using the endpoint analogue of the representation
\eqref{lem:representation_of_primal_gap_estimator.0}
of the primal gap estimator
\eqref{def:primal_gap_estimator}
and
\eqref{subsec:elasto_plastic_torsion.fenchel_gap}
(applied with $(v,w,y)=(v,v,y)$), we find that
\begin{align}
\label{lem:primal_gap_identity_elasto_plastic_torsion_problem.2}
\begin{aligned}
    \eta_{\mathcal{E}}^2(v)
    &=
    \inf_{\substack{
        y=y^{\mathrm{a}}\,\mathrm{d}t\mathrm{d}x+y^{\mathrm{s}}\in\operatorname{ba}(Q;\mathbb{R}^d)
        \\
        L^*y=f-\partial_t v
        \ \text{in }(L^\infty_{\mathrm{w}^*}(I;V))^*
    }}
    \bigl\{
        I_{\phi^*}(y^{\mathrm{a}})
        -
        (y^{\mathrm{a}},\nabla v)_Q
        +
        \tfrac{1}{2}\|\nabla v\|_Q^2
        \\[-5mm]
        &\qquad\qquad\qquad\qquad\qquad\qquad
        +
        |\zeta y^{\mathrm{s}}|(Q)
        -
        \langle y^{\mathrm{s}},\nabla v\rangle_{L^\infty(Q;\mathbb{R}^d)}
    \bigr\}\,.
\end{aligned}
\end{align}

Finally, using the representations
\eqref{lem:primal_gap_identity_elasto_plastic_torsion_problem.1}
and
\eqref{lem:primal_gap_identity_elasto_plastic_torsion_problem.2}
and the endpoint analogue of the general primal gap identity
\eqref{thm:primal_gap_identity.0}, the proof of which carries over
verbatim since
$\mathcal{E}(u)=\tfrac{1}{2}\|u_0\|_H^2$ (\textit{cf}.\ \eqref{subsec:elasto_plastic_torsion_problem.primal.2}), we arrive at the
claimed primal gap identity for the unsteady elasto-plastic
torsion problem.
\end{proof}

Next, we derive the corresponding dual gap identity.\enlargethispage{2.5mm}

\begin{lemma}[Dual gap identity for the unsteady elasto-plastic
torsion problem]
\label{lem:dual_gap_identity_elasto_plastic_torsion}
For every
$y=y^{\mathrm{a}}\,\mathrm{d}t\mathrm{d}x+y^{\mathrm{s}}\in \mathrm{ba}(Q;\mathbb{R}^d)$ and
$\lambda\in\mathcal{W}_\infty(I)$ with 
$\nabla\lambda(t)\in K(t)$ for a.e.\ $t\in I$ and
\begin{align}
\label{lem:dual_gap_identity_elasto_plastic_torsion.-1}
    L^*y+\partial_t\lambda
    =
    f
    \quad\text{in }(L^\infty_{\mathrm{w}^*}(I;V))^*\,,
\end{align}
there holds
\begin{align}
\label{lem:dual_gap_identity_elasto_plastic_torsion.0}
\begin{aligned}
    &\tfrac{1}{2}(\mathbb{H}_{\phi^*}^{y^{\mathrm{a}},z^{\mathrm{a}}}(y^{\mathrm{a}}-z^{\mathrm{a}}),y^{\mathrm{a}}-z^{\mathrm{a}})_{Q} 
    +
    |\zeta y^{\mathrm{s}}|(Q)
    -
    \langle y^{\mathrm{s}},\nabla u\rangle_{L^\infty(Q;\mathbb{R}^d)}
    \\
    &\quad+
    \tfrac{1}{2}\|\nabla(\lambda-u)\|_Q^2
    +
    \smash{\bigl(
        \tfrac{|z^{\mathrm{a}}|}{\zeta}-1,
        \zeta^2-\nabla u\cdot\nabla\lambda
    \bigr)}_{\{|\nabla u|=\zeta\}}
    \\
    &\quad+
    |\zeta z^{\mathrm{s}}|(Q)
    -
    \langle z^{\mathrm{s}},\nabla\lambda\rangle_{L^\infty(Q;\mathbb{R}^d)}
    +
    \tfrac{1}{2}\|(\lambda-u)(t_{\mathtt{fin}})\|_H^2
    \\
    &=
    I_{\phi^*}(y^{\mathrm{a}})
    -
    (y^{\mathrm{a}},\nabla\lambda)_Q
    +
    \tfrac{1}{2}\|\nabla\lambda\|_Q^2
    +
    |\zeta y^{\mathrm{s}}|(Q)
    -
    \langle y^{\mathrm{s}},\nabla\lambda\rangle_{L^\infty(Q;\mathbb{R}^d)}+
    \tfrac{1}{2}\|\lambda(0)-u_0\|_H^2\,.
\end{aligned}
\end{align}
Here, the mapping $\mathbb{H}_{\phi^*}^{y^{\mathrm{a}},z^{\mathrm{a}}}\coloneqq
    2\int_0^1
    (1-s)
    H_{\zeta}
    (sy^{\mathrm{a}}+(1-s)z^{\mathrm{a}})
    \,\mathrm{d}s\colon Q\to \mathbb{R}^{d\times d}$, where $H_{\zeta}\colon Q\times \mathbb{R}^d\to \mathbb{R}^{d\times d}$, for a.e.\ $(t,x)\in Q$ and every $q\in \mathbb{R}^d$, is defined by 
    \begin{align*}
    H_{\zeta}(t,x,q)
    \coloneqq
    \begin{cases}
        I_d
        & |q|<\zeta(t,x)\,,\\
        O_d& |q|=\zeta(t,x)\,,\\
        \tfrac{\zeta(t,x)}{|q|}(
            I_d-\tfrac{q\otimes q}{|q|^2})\,,
        & |q|>\zeta(t,x)\,,
    \end{cases}
\end{align*}
where $I_d=(\delta_{ij})_{i,j=1,\ldots,d},O_d=(0)_{i,j=1,\ldots,d}\in \mathbb{R}^{d\times d}$ denote the identity and the zero matrix,~respectively.
\end{lemma}  

\begin{proof}
Let $y=y^{\mathrm{a}}\,\mathrm{d}t\mathrm{d}x+y^{\mathrm{s}}\in\operatorname{ba}(Q;\mathbb{R}^d)$
and $\lambda\in\mathcal{W}_\infty(I)$ with
$\nabla\lambda(t)\in K(t)$ for a.e.\ $t\in I$ and \eqref{lem:dual_gap_identity_elasto_plastic_torsion.-1} 
be fixed, but arbitrary.

\emph{$\bullet$ Optimal strong convexity measure.}
Since, due to the endpoint analogue of the alternative Bregman
divergence representation
\eqref{rem:alternative_representation_bregman_divergence.8}, \eqref{subsec:elasto_plastic_torsion_problem.opt.3}, 
and the integral form of Taylor's formula applied pointwise for a.e.\  $(t,x)\in Q$
(which is applicable since
$\mathrm{D}_a\phi^*(t,x,\cdot)
=\Pi_{\zeta(t,x)}$
is Lipschitz continuous and, hence, absolutely continuous along line
segments, and its derivative along such segments is given~via~$H_\zeta(t,x,\cdot)$; the assigned value on $\vert q\vert =\zeta(t,x)$ is immaterial), we have that
\begin{align*}
\begin{aligned}
    \mathcal{D}_{(I_G)^*}^{\nabla u}(y,z)
    &=
    I_{\phi^*}(y^{\mathrm{a}})
    -
    I_{\phi^*}(z^{\mathrm{a}})
    -
    (\nabla u,y^{\mathrm{a}}-z^{\mathrm{a}})_Q+
    |\zeta y^{\mathrm{s}}|(Q)
    -
    \langle y^{\mathrm{s}},\nabla u\rangle_{L^\infty(Q;\mathbb{R}^d)} 
    \\&=
    I_{\phi^*}(y^{\mathrm{a}})
    -
    I_{\phi^*}(z^{\mathrm{a}})
    -
    (\mathrm{D}_a\phi^*(\cdot,\cdot,z^{\mathrm{a}}),y^{\mathrm{a}}-z^{\mathrm{a}})_Q+
    |\zeta y^{\mathrm{s}}|(Q)
    -
    \langle y^{\mathrm{s}},\nabla u\rangle_{L^\infty(Q;\mathbb{R}^d)} 
    \\&=\tfrac{1}{2} (\mathbb{H}_{\phi^*}^{y^{\mathrm{a}},z^{\mathrm{a}}}(y^{\mathrm{a}}-z^{\mathrm{a}}),
        y^{\mathrm{a}}-z^{\mathrm{a}})_Q +
    |\zeta y^{\mathrm{s}}|(Q)
    -
    \langle y^{\mathrm{s}},\nabla u\rangle_{L^\infty(Q;\mathbb{R}^d)}\,,
\end{aligned}
\end{align*}
due to the endpoint analogue of the alternative
Bregman divergence representation
\eqref{rem:alternative_representation_bregman_divergence.10}, $I_F= -\langle f,\cdot\rangle_{\smash{L^\infty_{\mathrm{w}^*}(I;V)}}$,
$(I_F)^*=\chi_{\{-f\}}$, \eqref{lem:dual_gap_identity_elasto_plastic_torsion.-1},
and \eqref{subsec:elasto_plastic_torsion_problem.opt.2}, we have that
\begin{align*}
    \mathcal{D}_{(I_F)^*}^{u}
    (-\partial_t\lambda-L^*y,-\partial_tu-L^*z)
    &=\chi_{\{-f\}}(-\partial_t\lambda-L^*y)-\langle f-\partial_t\lambda-L^*y,u\rangle_{\smash{L^\infty_{\mathrm{w}^*}(I;V)}}
    \\&=0\,,
\end{align*}
and, due to the endpoint analogue of the alternative
Bregman divergence representation
\eqref{rem:alternative_representation_bregman_divergence.4},
\eqref{subsec:elasto_plastic_torsion.fenchel_gap}
(applied with $(v,w,y)=(u,\lambda,z)$), \eqref{subsec:elasto_plastic_torsion_problem.opt.4}, that $I_{\phi^*}(z^{\mathrm{a}})
    -
    (z^{\mathrm{a}},\nabla u)_Q
    +
    \tfrac{1}{2}\|\nabla u\|_Q^2
    =0$ (which is equivalent to \eqref{subsec:elasto_plastic_torsion_problem.opt.3}), and a binomial formula, we have that
\begin{align*}
\begin{aligned}
    \mathcal{D}_{I_E}^{-\partial_tu}(\lambda,u)
    &=
    I_{\phi^*}(z^{\mathrm{a}})
    -
    (z^{\mathrm{a}},\nabla\lambda)_Q
    +
    \tfrac{1}{2}\|\nabla\lambda\|_Q^2
    +
    |\zeta z^{\mathrm{s}}|(Q)
    -
    \langle z^{\mathrm{s}},\nabla\lambda\rangle_{L^\infty(Q;\mathbb{R}^d)}
    \\
    &=
    \tfrac{1}{2}\|\nabla(\lambda-u)\|_Q^2
    +
    \smash{\bigl(
        \tfrac{|z^{\mathrm{a}}|}{\zeta}-1,
        \zeta^2-\nabla u\cdot\nabla\lambda
    \bigr)}_{\{|\nabla u|=\zeta\}}
    +
    |\zeta z^{\mathrm{s}}|(Q)
    -
    \langle z^{\mathrm{s}},\nabla\lambda\rangle_{L^\infty(Q;\mathbb{R}^d)}\,,
\end{aligned}
\end{align*}
the endpoint analogue of the Bregman-type
representation of the dual error measure
\eqref{lem:bregman_representation_dual_error_measure.0}
yields
\begin{align}
\label{lem:dual_gap_identity_elasto_plastic_torsion_problem.1}
\begin{aligned}
    \rho_{-\mathcal{D}}^2(y,\lambda)
    &=
    \tfrac{1}{2}
    (
        \mathbb{H}_{\phi^*}^{y^{\mathrm{a}},z^{\mathrm{a}}}(y^{\mathrm{a}}-z^{\mathrm{a}}),
        y^{\mathrm{a}}-z^{\mathrm{a}}
    )_Q
    +
    |\zeta y^{\mathrm{s}}|(Q)
    -
    \langle y^{\mathrm{s}},\nabla u\rangle_{L^\infty(Q;\mathbb{R}^d)}
    \\
    &\quad+
    \tfrac{1}{2}\|\nabla(\lambda-u)\|_Q^2
    +
    \smash{\bigl(
        \tfrac{|z^{\mathrm{a}}|}{\zeta}-1,
        \zeta^2-\nabla u\cdot\nabla\lambda
    \bigr)}_{\{|\nabla u|=\zeta\}}
    \\
    &\quad+
    |\zeta z^{\mathrm{s}}|(Q)
    -
    \langle z^{\mathrm{s}},\nabla\lambda\rangle_{L^\infty(Q;\mathbb{R}^d)}
    +
    \tfrac{1}{2}\|(\lambda-u)(t_{\mathtt{fin}})\|_H^2\,.
\end{aligned}
\end{align}

\emph{$\bullet$ Dual gap estimator.}
Using the endpoint analogue of the representation
\eqref{lem:representation_of_dual_gap_estimator.0}
of the dual gap estimator
\eqref{def:dual_gap_estimator}
and
\eqref{subsec:elasto_plastic_torsion.fenchel_gap},
which is applicable due to the balance condition, we find that
\begin{align}
\label{lem:dual_gap_identity_elasto_plastic_torsion_problem.2}
\begin{aligned}
    \eta_{-\mathcal{D}}^2(y,\lambda)
    &=
    I_{\phi^*}(y^{\mathrm{a}})
    -
    (y^{\mathrm{a}},\nabla\lambda)_Q
    +
    \tfrac{1}{2}\|\nabla\lambda\|_Q^2
    \\
    &\quad+
    |\zeta y^{\mathrm{s}}|(Q)
    -
    \langle y^{\mathrm{s}},\nabla\lambda\rangle_{L^\infty(Q;\mathbb{R}^d)}
    +
    \tfrac{1}{2}\|\lambda(0)-u_0\|_H^2\,.
\end{aligned}
\end{align}

Finally, using the representations
\eqref{lem:dual_gap_identity_elasto_plastic_torsion_problem.1}
and
\eqref{lem:dual_gap_identity_elasto_plastic_torsion_problem.2}
and the endpoint analogue of the general dual gap identity
\eqref{thm:dual_gap_identity.0}, the proof of which carries over
verbatim since
$\mathcal{D}(z,u)=\tfrac{1}{2}\|u_0\|_H^2$ (\textit{cf}.\ \eqref{subsec:elasto_plastic_torsion_problem.opt.5}), we arrive at the
claimed dual gap identity for the unsteady elasto-plastic
torsion problem.\enlargethispage{5mm}
\end{proof}

{\setlength{\bibsep}{0pt plus 0.0ex}\small
		
		\bibliographystyle{aomplain}
		\bibliography{references}
		
	}

\end{document}